\documentclass[12pt]{article}
\usepackage{amssymb}
\usepackage{mathtools}
\usepackage{enumerate}
\usepackage{amsmath}
\usepackage{mathrsfs}
\usepackage{graphics}
\usepackage{graphicx}
\usepackage{xcolor}
\usepackage{pdfsync}
\usepackage{subfigure}
\usepackage[T1]{fontenc}
\usepackage{yhmath}
\usepackage{latexsym,amssymb,amsmath,amsfonts,amsthm}\usepackage{txfonts}
\numberwithin{equation}{section}
\newcommand{\R}{{\mathbb R}}
\def\dif{{\mathord{{\rm d}}}}

\newcommand{\E}{{\mathbb E}}
\newcommand{\PM}{{\mathbb P}}

\newcommand{\bae}{\begin{equation}\begin{aligned}}
\newcommand{\eae}{\end{aligned}\end{equation}}
\newcommand{\baee}{\begin{equation*}\begin{aligned}}
\newcommand{\eaee}{\end{aligned}\end{equation*}}
\newcommand{\be}{\begin{eqnarray}}
\newcommand{\ben}{\begin{eqnarray*}}
\newcommand{\en}{\end{eqnarray}}
\newcommand{\enn}{\end{eqnarray*}}

\newtheorem{theorem}{Theorem}[section]
\newtheorem{Lemma}{Lemma}[section]
\newtheorem{prp}[theorem]{Proposition}

\newtheorem{dfn}{Definition}[section]

\newtheorem{remark}{Remark}[section]

\DeclareMathOperator{\Ima}{Im}
\makeatletter
\let\@@span\span
\def\sp@n{\@@span\omit\advance\@multicnt\m@ne}
\renewcommand{\span}{span}

\begin{document}
	\renewcommand{\theequation}{\arabic{section}.\arabic{equation}}
	\begin{titlepage}
		\title{\bf On the stability of rarefaction for stochastic viscous conservation law}
		\author{Zhao Dong$^{2,3}$,\ Feimin Huang$^{2,3}$,\ Houqi Su$^{1,2,3}$  \\
			{\small $^{1 }$	Capital Normal University, Beijing, China}\\
   {\small $^2$ University of Chinese Academy of Sciences, Beijing, China}\\
    {\small $^3$ Academy of Mathematics and Systems Science, Chinese Academy of Sciences, Beijing, China}\\
			{\small {\sf dzhao@amt.ac.cn},\ {\sf fhuang@amt.ac.cn},\ {\sf marksu@amss.ac.cn} }}
		\date{}
	\end{titlepage}
	\maketitle
	\noindent\textbf{Abstract}:
	We study the asymptotic stability of rarefaction waves for one-dimensional stochastic viscous conservation laws driven by nonlinear conservative noise. In a critical scaling where stochastic energy injection and viscous dissipation compete at comparable magnitudes, standard kinetic and viscosity frameworks encounter obstructions due to regularity gaps and non-integrable profiles. To address this, we introduce a stochastic area inequality controlling accumulated energy fluctuations, a local $L^1$
contraction principle via stochastic Kružkov doubling-of-variables that yields pathwise uniqueness without global integrability, and a modified Galerkin scheme preserving the $H^2$
energy structure. Assuming local $H^2$ regularity, we prove almost sure algebraic convergence to the rarefaction wave. For sufficiently small initial perturbations, we establish global well-posedness and sharp decay estimates in expectation. The smallness condition identifies a regime where viscous dissipation dominates stochastic injection, reflecting a structural stability threshold rather than a technical artifact. Our approach extends the analytical framework for conservative SPDEs with rough fluxes.
 
 \
 
 \newpage
\tableofcontents
 
\section{Introduction}

The analysis of SPDEs driven by nonlinear conservative noise has emerged as a central frontier in stochastic analysis, motivated by the need to rigorously bridge the gap between microscopic interacting particle systems and macroscopic fluid limits \cite{fehrman2023non}. Unlike additive perturbations which act as external forcing, conservative noise of the form $\nabla \cdot (\sigma(u) \circ \dot{W})$ fundamentally alters the flux structure, introducing rough, derivative-dependent fluctuations that challenge the classical theories of well-posedness, see \cite{fehrman2019well}.  While the last decade has seen breakthrough developments in kinetic solutions \cite{lions2013scalar} and pathwise regularization theories \cite{gassiat2019regularization,gassiat2020speed,gess2017long}, the stability of fundamental hydrodynamic patterns on the unbounded line remains largely unexplored. In this paper, we address the stability of the rarefaction wave—a non-integrable profile—under rough conservative noise. We identify a critical regime where the competition between stochastic energy injection and viscous dissipation mirrors the focusing heat equation, requiring a new analytical framework beyond standard kinetic or viscosity solution techniques.

To formalize this stability problem, we first recall the classical Riemann problem for the underlying deterministic system. Since the seminal work of Riemann \cite{riemann1860fortpflanzung}, the analysis of the Riemann problem \eqref{riemann}—an initial value problem with scale-invariant data—has played a key role in understanding the wave structure of hyperbolic partial differential equations.
\bae\label{riemann}
\left\{\begin{aligned}
		& \partial_t u+\partial_{x}f(u) = 0,\\
		&\left.u\right|_{t=0}= \begin{cases}u_{-}, & x<0, \\ u_{+}, & x>0,\end{cases}~ x\in \R.
	\end{aligned}\right.
\eae
Furthermore, the Riemann problem \eqref{riemann} is the standard building block in  the numerical schemes for compressible flow, for example, see \cite{menikoff1989riemann}. If $u_-<u_+$, then the Riemann solution is the rarefaction wave $u^r$, which reads 
$$
u^r(t, x)=\left\{\begin{array}{cl}
	u_{-}\,, & \frac{x}{t}<f^{\prime}\left(u_{-}\right), \\
	\left(f^{\prime}\right)^{-1}\left(\frac xt\right)\,, & f^{\prime}\left(u_{-}\right)<\frac{x}{t}<f^{\prime}\left(u_{+}\right), \\
	u_{+}\,, & \frac{x}{t}>f^{\prime}\left(u_{+}\right).
\end{array}\right.
$$

In the presence of viscosity, it is commonly conjectured that the long-time behavior of solutions is governed by the smoothed Riemann solution of the corresponding inviscid system. For the deterministic Burgers equation, Ilin and Oleinik \cite{il1958behaviour} first demonstrated that solutions converge to  $u^r$ in the $L^\infty$ norm, with Hattori and Nishihara \cite{hattori1991note} later establishing a precise algebraic decay rate of $(1+t)^{-1/2}$.

While the asymptotic stability of this profile is well-established for deterministic viscous conservation laws, it remains an interesting question whether this fundamental hydrodynamic structure survives rough, derivative-dependent stochastic perturbations. This leads us to consider the stochastic viscous conservation law on $\mathbb{R}$ driven by nonlinear conservative noise.

In this paper, we investigate the stability of this fundamental hydrodynamic structure under stochastic perturbation. We consider the stochastic viscous conservation law on $\mathbb{R}$ driven by nonlinear conservative noise
\begin{equation}\label{conservative}
\begin{cases}
du = \partial_x \big(b(u)u_x\big) dt+ \partial_x \big[- f(u)dt + \sum_{i=1}^\infty \sigma_i(u) \circ dB_i(t)\big], \\
u(0,x) = u_0(x), \, x \in \mathbb{R},\quad u_0(\pm \infty)=u_\pm, \quad u_-<u_+,
\end{cases}
\end{equation}
Here $\circ$ denotes the Stratonovich integral, $B_i$ are i.i.d. Brownian motions over a stochastic basis $\left(\Omega, \mathcal{F},\left\{\mathcal{F}_t\right\}_{t \geq 0}, \mathbb P\right)$.  \textit{The central question of this work is whether the rarefaction wave—a non-integrable profile—retains its stability in this rough stochastic environment.}

In the rest of the paper, we assume that
	\begin{enumerate}
        \item[\rm{(H1)}] { The viscosity $b(\cdot)\in C^1$ is \textit{uniformly positive}: $\mu:=\displaystyle\inf_{u\in \R} b(u)>0$.
        }
        \item[\rm{(H2)}] {The flux $f(\cdot)\in C^2$ is \textit{strictly convex}:  $\alpha^*:=\displaystyle\inf_{u\in \R}f^{\prime\prime}(u)>0$.}
		\item[\rm{(H3)}] { For any $N>0$, $\displaystyle\sum_{i=1}^\infty \|\sigma_i^\prime \|^2_{H^2(-N,N)}<\infty.	$
		}
\end{enumerate}

\subsection{Main Results}

We resolve the well-posedness and asymptotic stability of the system through a unified framework comprising three main components: stability, existence, and uniqueness.

First, we establish a general stability result utilizing the newly introduced Stochastic Area Inequality. The following theorem asserts that any sufficiently regular solution that avoids local blow-up inevitably converges to the rarefaction wave, transferring the question of asymptotic stability to one of global regularity.
\begin{theorem}\label{as L^infty decay}
    Assume {\rm{(H1)-(H3)}} hold. If $u(t)-\bar u(t)\in L^2_{{\rm{{loc}}}}\big((0,\infty),H^2(\mathbb R)\big)$ almost surely. Then for all $\epsilon>0$, we have
\baee
\|u(t,\cdot)-u^r(t,\cdot)\|_{L^\infty(\mathbb R)}=O\big(t^{-\frac14+\epsilon}\big),\ \mathbb P\text{-a.s..}
\eaee
\end{theorem}
This result transfers the question of asymptotic stability to one of global regularity, showing that stochastic perturbations cannot ``shake'' the solution away from the rarefaction attractor provided sufficient smoothness is maintained.

{Second}, we resolve the blow-up issue by identifying a {small data regime}. In the following theorem, we show that if the initial perturbation is sufficiently small, the solution maintains global $H^2$ regularity and sharp decay estimates in expectation
\begin{theorem}\label{th2.2}
	Assume {\rm{(H1)-(H3)}} hold. Let  $u(0,\cdot)-\bar u(0,\cdot)\in H^1(\R)$ and set
$$
	\mathfrak M\coloneqq\|u(0,\cdot)-\bar u(0,\cdot)\|_{L^\infty(\R)}+u_+-u_-.
$$
	 Then there exists a constant $C^*(\mathfrak M,\mu, b',\sigma'')>0,$ such that, if
	\baee\label{1.4}
	\left\|u(0,\cdot)-\bar u(0,\cdot)\right\|_{L^{\infty}(\mathbb{R})} \leq C^*(\mathfrak M,\mu, b',\sigma''),
	\eaee
	then for any $T>0,$ there exists a $u-\bar u\in L^2\big((0,T)\times \Omega,H^2(\mathbb R)\big)$ solves the equation \eqref{perturbation} with the following decay estimates
	\baee\label{derive-inftydecay}
 \mathbb E \| u(t,\cdot)-u^r(t,\cdot) \|_{ L^\infty (\mathbb R)}  &\leq C(1+t)^{-\frac{1}{4}}\ln^\frac{1}{2}(2+t),\\
	\E \|u_x(t,\cdot)- u^r_x(t,\cdot)\|_{L^2(\R)}&\leq C (1+t)^{-\frac12} \ln^\frac12(2+t).
	\eaee
\end{theorem}
\begin{remark}
   We emphasize that the smallness condition in Theorem {\rm{\ref{th2.2}}} is likely a physical necessity rather than a technical artifact. It is tempting to conjecture that a blow-up (where $\|\partial_{x}\phi\|_{L^2}\to\infty$) might actually enhance stability, analogous to how shock waves in deterministic conservation laws dissipate entropy via an ``infinite'' gradient mechanism. Indeed, if the viscous term $-\mu||\partial_{xx}\phi||_{L^2}^{2}$ were the sole dominant factor, the singularity would effectively drain the energy and ``reset'' the system. However, as illustrated by the toy model \eqref{regu}, the conservative noise introduces a focusing nonlinearity roughly proportional to $||\partial_{x}\phi||_{L^4}^{4}$ which competes directly with the viscous dissipation. In the deterministic setting, Levine {\rm{\cite{levine1973some}}} showed that such focusing nonlinearities lead to finite-time blow-up where the solution leaves the energy space entirely. In our stochastic context, if the injection dominates, the singularity acts as an energy source rather than a dissipative sink. Thus, our result identifies the ``basin of attraction'' where viscosity suffices to suppress this stochastic energy injection, ensuring the global persistence of the rarefaction wave.
\end{remark}

{Third}, we address the critical issue of {uniqueness}. While our existence results are derived in an $L^2$-based Sobolev framework, standard $L^2$ energy methods fail to yield uniqueness for conservative SPDEs due to the nonlinear noise. In the following theorem, we establish a new $\rm{L_{loc}^{1}}$ contraction principle.
\begin{theorem}\label{loccon}
     Assume {\rm{(H1)-(H3)}} hold and $u_0-\bar u_0$ and $v_0-\bar u_0\in H^1(\R)$. Let $u,v$ be the solutions of \eqref{conservative} with initial data $u_0,v_0$ respectively. Then for $r>0$ and almost all $t>0$,
   \baee
\mathbb E\int_{|x|<r}|u(x, t)-v(x, t)| d x   \leq  \int_{|x|<r+S t}\left|u_0(x)-v_0(x)\right| d x,
\eaee
where 
$$
     M(u_0,v_0)\coloneqq\|u_0\|_{L^\infty(\R)}\vee\|v_0\|_{L^\infty(\R)}\text{ , \,}\displaystyle S\coloneqq  \sup_{|x|\leq M(u_0,v_0)}\left|f^\prime(x)\right|.$$
\end{theorem}

\ 

\textbf{Applications:} Beyond the specific stability of rarefaction waves, the analytical framework presented here may offer insights into certain challenges in the broader theory of SPDEs. The Stochastic Area Inequality suggests a mechanism to prove asymptotic stability in "critical" regimes—where noise and dissipation scale identically—which could be applicable to stochastic reaction-diffusion equations and hydrodynamics where standard Lyapunov methods face difficulties. Moreover, the Stochastic $\rm{L_{loc}^{1}}$ Contraction in Theorem \ref{loccon} provides a consistent uniqueness criterion that extends beyond the viscous setting. As detailed in Remarks \ref{invicidremark} and \ref{rm 3.2}, this contraction principle appears robust enough to handle inviscid conservation laws and polynomial growth coefficients, serving to bridge the gap between viscous strong solutions and inviscid entropy solutions. By relaxing the reliance on global $L^1$ integrability inherent in standard kinetic formulations or HJB approaches, this tool offers an alternative perspective for analyzing uniqueness in conservative SPDEs with non-integrable data or rough fluxes. Finally, the constructive existence scheme—based on a modified Galerkin method with smooth cut-offs—provides a rigorous basis for developing structure-preserving numerical discretizations for SPDEs driven by nonlinear conservative noise, a regime where ensuring stability in finite-dimensional approximations is significant.
\subsection{Main Difficulties and Strategies}\label{sec1.1}
\subsubsection{Main Difficulties }
The stability analysis of rarefaction waves under conservative noise faces several analytical barriers.

\paragraph{The Competition Between Noise and Dissipation}
The stability analysis of rarefaction waves under conservative noise is obstructed by a core physical instability. To illustrate this mechanism, consider the simplified case where $b=1$, $f=0$, $u_{\pm}=0$, and $\sigma(u)=u^2/2$, reducing the system to a heat equation with conservative noise. A formal It\^o calculus for $\|u_x\|^2_{L^2}$ yields
\begin{equation} \label{regu}
\frac{d}{dt}\mathbb{E}\|u_x\|_{L^2}^2 = 2\mathbb{E}\|u_x\|_{L^4}^4 - 2\mathbb{E}\|u_{xx}\|_{L^2}^2.
\end{equation}
The noise-induced term $2\mathbb{E}\|u_x\|_{L^4}^4$, being critical in the Sobolev embedding, acts as a powerful energy injection into high frequencies. This creates a critical scaling regime where noise and dissipation are of the same order. The structure of $u_x$ mirrors the deterministic focusing heat equation $w_t = w_{xx} + w^3$, which implies
\begin{equation*} \label{levin}
\frac{d}{dt}\|w\|_{L^2}^2 = 2\|w\|_{L^4}^4 - 2\|w_{x}\|_{L^2}^2.
\end{equation*}
For this equation, Levine \cite{levine1973some} established a rigorous blow-up criterion: solutions must develop singularities in finite time whenever the initial data satisfies $\|w_0\|_{L^4}^4 - 2\|w_0'\|_{L^2}^2 > 0$. Since this condition is invariably satisfied for \textit{sufficiently large data} (e.g., under scaling $\lambda w_0$ for large $|\lambda|$), it suggests that the finite-time blow-up is a generic feature for the toy model \eqref{regu}.

\paragraph{Limitations of Existing Frameworks}
Recent advances in stochastic scalar conservation laws (SSCL) have successfully utilized kinetic formulations and viscosity solutions. However, while these approaches are highly effective for integrable initial data, the specific structure of the rarefaction wave introduces unique analytical challenges:

\begin{itemize}
    \item \textit{Fractional vs. Higher-Order Regularity:} The kinetic formulation approach, advanced by Lions, Perthame, and Souganidis \cite{lions2013scalar} \cite{lions2014scalar} and later Gess and Souganidis \cite{gess2017long}, relies on averaging lemmas to extract regularity. However, these methods typically yield only fractional Sobolev regularity (e.g., $W^{s,1}$ for $s < 1$). However, unlike the inviscid case (nonlinear conservative noise actually adds regularity), in the viscous setting, the expected deterministic parabolic smoothing  becomes unattainable. An inherent feature of the kinetic formulation is that the source term $\partial_{\xi}m$ is singular in velocity. To recover the solution, one must transfer this derivative onto the heat kernel, an operation that mathematically consumes one full order of the parabolic smoothing. Furthermore, the presence of the Laplacian dampens high frequencies exponentially fast, effectively suppressing the stochastic phase cancellations that would otherwise provide a boost.  To prove the asymptotic stability of the rarefaction wave, we require sharp pathwise decay estimates (e.g., $t^{-1/4}$ in $L^\infty$), which necessitate controlling the second derivative (in $H^2$) to utilize embedding theorems. We cannot bridge the gap from the $W^{s,1}$ regularity provided by averaging lemmas to the $H^2$ regularity required to close our decay estimates.
    
    \item \textit{Global Integrability and Translation Invariance:} The viscosity solution method introduced by Gassiat and Gess \cite{gassiat2019regularization} is a powerful tool for demonstrating regularization by quadratic conservative noise and $L^{\infty}$ decay. While their work successfully treats viscous equations by lifting the problem to the primitive (antiderivative) variable, adapting this transformation for non-integrable profiles presents distinct difficulties. For a solution $u$ that is asymptotically constant and non-zero at infinity, the potential $U(x,t) = \int^x u$ grows linearly as $|x| \to \infty$. Furthermore, their pathwise regularization arguments rely on translation invariance, which is broken here by the time-evolving rarefaction profile $\bar{u}(x,t)$.
\end{itemize}
A common thread in these established approaches is the reliance on \textit{global $L^{1}$ integrability}, which the infinite mass of the rarefaction wave precludes.

\paragraph{The Uniqueness of the Solution in $L^2$ Frame} In the $L^2$ frame, formally, one has
\bae\label{monott}
\frac{d}{dt}\mathbb E \|u-v\|^2_{L^2}=& \mathbb E\int_\R(u-v)^2 u_x v_x \dif x -2\mathbb E\|u_x-v_x\|_{L^2}^2.
\eae
 The first term on the right  hand side of \eqref{monott}  can not be handled due to the lack of appropriate estimate when the diffusion coefficient $\sigma$ is nonlinear (see Lions et al. \cite{lions2013scalar}). 

\subsubsection{Strategies}
To overcome these limitations, our work departs from the kinetic and HJB frameworks and relies fundamentally on \textit{stochastic analysis} via high-order energy estimates. We introduce three specific new ideas to resolve the issues above. These tools allow us to capture the interplay between noise and dissipation, bridging the gap between stability and regularity that previous methods could not cross.

\begin{enumerate}

\item \textbf{The Stochastic Area Inequality(Resolving the Blow-up):} \textit{Global Integrability Constrains Local Blow-up.} To control the critical energy injection described in (1.3), we introduce the Stochastic Area Inequality (Lemma \ref{stochastic area}). This tool shifts the analysis from a pointwise competition (checking if viscosity dominates noise at every instant) to an accumulated behavior. We prove that because the ``total area''  of the fluctuation energy is finite almost surely ($\int_0^t \|u_x\|^2 dt < \ln (1+t)$), the blow-up mechanism is effectively ``starved'' of fuel. The inequality acts as a rigorous circuit breaker, ensuring that even if stochastic fluctuations cause local spikes, the global integrability constraint forces the solution to decay asymptotically.

  \item \textbf{Stochastic $\rm{L_{loc}^{1}}$ Contraction(Resolving Uniqueness):} While our existence results are in $L^2$, standard $L^2$ methods fail to yield uniqueness for conservative SPDEs due to the nonlinear noise structure. Conversely, standard entropy methods fail due to the infinite mass of the rarefaction wave. We resolve this by constructing a stochastic test function within the Kru\v{z}kov doubling-of-variables technique. This yields a novel $\rm{L_{loc}^{1}}$ contraction principle (Theorem \ref{loccon}) that is compatible with our $L^2$ framework without requiring global  $L^1$ integrability which is often assumed in kinetic or viscosity solution  approaches.

  \item \textbf{Constructive Regularity via Galerkin Approximation(Bridging the Gap):} To bypass the ``regularity gap'' of kinetic formulations, we forego the averaging lemma approach in favor of a constructive $H^2$ method. We employ a Galerkin approximation on a truncated domain with smooth cut-off coefficients. Unlike kinetic approximations, this construction is compatible with Itô’s formula for second derivatives, allowing us to rigorously preserve the $H^2$ energy structure at the approximation level. This ensures that the formal stability estimates derived in Section 2 are mathematically justified for the limiting strong solution.
\end{enumerate}
\subsection{Related Literature}

While the methodological challenges of conservative noise were discussed in Section \ref{sec1.1}, it is worth placing our stability result within the broader context of wave phenomena in conservation laws.

The analysis of stochastic scalar conservation laws (SSCL) has seen significant development in recent years, with a particular focus on well-posedness, long-time behavior, and regularization by noise. The kinetic formulation approach, advanced by Lions, Perthame, and Souganidis \cite{lions2013scalar} \cite{lions2014scalar}, has become a standard tool for handling rough fluxes.  Debussche and Vovelle \cite{debussche2015invariant} obtained the existence and uniqueness of the invariant measure for the associated stochastic kinetic solution of the stochastic conservation laws. The work of Gess and Souganidis \cite{gess2017long} has been highly successful in studying the long-time behavior of scalar conservation laws under conservative noise on the torus $\mathbb{T}^N$. For all  $p \in(1, \infty)$, $u_0\in L^\infty(\mathbb T^N)\footnote{$\bar{u}_0:=\int_{\mathbb{T}^N} u_0(x) d x$}$ and $\theta\in (0,1]\footnote{$\theta$ is determined by the flux $f,$ see \cite{gess2017long} for detailed definition.},$ they showed the pathwise $L^1(\mathbb{T}^N)$ convergence and expectation $L^p(\mathbb{T}^N)$ convergence
\bae\label{gs Lp}
&u\left(\cdot, t ; \cdot, u_0\right) \rightarrow \bar{u}_0 \quad \text { in } L^1\left(\Omega ; L^1(\mathbb{T}^N)\right) \text { and } \mathbb{P} \text {-a.s. in } L^1(\mathbb{T}^N),\\
&\mathbb{E}\left\|u\left( t\right)-\bar{u}_0\right\|_{L^p} \lesssim C_p t^{-\frac{\theta}{3+\theta}}.
\eae
 It is worth noting that in \eqref{gs Lp}, the pathwise convergence is established in $L^1$ \textit{without an explicit rate}, while the convergence in expectation \textit{does not extend to the $L^\infty$ norm.}

Parallel to the kinetic theory, viscosity solution methods for stochastic Hamilton-Jacobi (HJB) equations have also been developed. Gassiat and Gess \cite{gassiat2019regularization} investigated the regularizing effects of nonlinear stochastic perturbations on fully nonlinear PDEs. By lifting the problem to the primitive variable and employing pathwise viscosity solution techniques, they obtained sharp $L^\infty$ bounds on the second derivative, effectively proving a "regularization by noise" result where stochasticity prevents the formation of shocks that would otherwise occur in the deterministic setting. In the one dimensional case, they showed that for the equation
$$
\begin{aligned}
\begin{cases}
    \partial_t U+\frac{1}{2}\left|\partial_x U\right|^2 \circ d B(t)  =F\left(U_x, U_{x x} \right) d t, \\
U(0)  =U_0 \in B U C(\mathbb{R}),
\end{cases}
\end{aligned}
$$
it holds that $\| U_x(t, \cdot)\|_{{\infty}} \leq C \big(\max _{0 \leq s \leq t} |B(s)|\big)^\frac{-1}{2} .$
In their  approach,  $F$ is  required to satisfy some assumptions from the theory of viscosity solution. However, the ``diffusion coefficient'' needs to be \textit{quadratic on $U_x$}, since the key technical estimates rely heavily on the specific algebraic structure of the quadratic Hamiltonian $H(p)=p^2/2.$ While in our problem \eqref{conservative}, $\sigma_i$ is essentially required to be $C^2.$

In the deterministic setting, the stability of rarefaction waves is well-established. Xin \cite{xin1990asymptotic} analyzed the asymptotic stability of planar rarefaction waves in several space dimensions. For compressible Navier-Stokes equations, Matsumura and Nishihara \cite{matsumura1985stability} established convergence rates to the rarefaction wave, and Zumbrun \cite{zumbrun2005stability} et al. developed stability theories for large-amplitude shock waves. These works elucidate that rarefaction waves are generally robust structures.

In the stochastic context, the authors \cite{dong2025large} studied the long-time behavior of the stochastic Burgers equation, investigating the formation of viscous shock waves and turbulent fluctuations. Notably, we found that under transport noise, rarefaction waves remain stable while shock waves do not. Similarly, Ryzhik and Dunlap\cite{dunlap2021viscous} considered the stochastic Burgers equation under additive conservative noise, proving convergence to a ``stochastic'' shock profile only under a dynamic shift. Our work complements these findings by establishing that rarefaction waves retain their stability even under rough conservative flux noise, provided the regularity is preserved. Regarding the long-time statistics and ergodicity, Hairer et al. \cite{hairer2016optimal} and Dong et al. \cite{dong2023ergodicity}  established polynomial mixing and optimal convergence rates to the invariant measure for conservation laws driven by space-time white noise and multiplicative colored noise, respectively.  In the $L^1$ framework, the theory of stochastic strong entropy solutions was initiated by Feng and Nualart \cite{feng2008stochastic} and further developed for various conservation and balance laws by Debussche and Vovelle \cite{debussche2010scalar}, Chen et al. \cite{chen2012nonlinear}, and Dareiotis et al. \cite{dareiotis2019entropy,dareiotis2020nonlinear}.

Finally, regarding the construction of strong solutions, our approach departs from the kinetic formulations often used for rough fluxes. Instead, we employ a constructive Galerkin approximation combined with rigorous compactness arguments. We rely on the Aubin-Lions-Simon compactness theorem \cite{aubin1963theoreme,lions1969quelques,simon1986compact} to establish the necessary tightness in time-space norms. To recover strong convergence from pathwise uniqueness, we utilize the Gyöngy-Krylov characterization \cite{gyongy1996existence}. This probabilistic framework has proved essential in establishing strong solutions for complex stochastic fluid models \cite{debussche2011local}, as well as the stochastic Euler equations \cite{glatt2014local} and compressible fluid flows \cite{breit2018stochastically}, where standard low regularity methods may fail.

\textbf{Extension to Higher Dimensions.} While our analysis is presented in the one-dimensional setting to isolate the critical stochastic difficulties, the extension to multi-dimensional planar rarefaction waves follows established energy hierarchies. As demonstrated by Huang and Xu \cite{huang2022decay,huangdecay}  for deterministic viscous conservation laws, the stability of planar waves in $\mathbb{R}^n$ relies on $L^p$ energy estimates that parallel the 1D case, provided the stochastic analysis of the rough flux—the main results of this work—is handled correctly. Thus, we focus on the 1D setting to highlight the  interaction between the conservative noise and the rarefaction profile without the algebraic complexity of vector calculus.

\subsection{Organization of the Paper}

This paper is organized to prioritize the stability mechanisms and novel uniqueness criteria before detailing the construction of solutions. The content is organized as follows:

\begin{itemize}
    \item \textit{Section 2: A Priori Estimates} \\
    This section establishes key a priori estimates by introducing the \textit{Stochastic Area Inequality} and \textit{High-Order Energy Estimates}. These tools are used to demonstrate how viscous dissipation effectively controls noise-induced energy injection, allowing for the derivation of sharp decay rates and the identification of the small data regime.

    \item \textit{Section 3: Uniqueness} \\
   The uniqueness proof (Section 3) utilizes the Kružkov doubling-of-variables technique. To streamline the presentation, we isolate the standard mollifier convergence lemmas in the Appendix, allowing the main text to focus on the novel $\rm{L_{loc}^{1}}$ contraction estimates and the handling of the viscous-stochastic competition.

    \item \textit{Section 4: Global Existence} \\
    This section provides the rigorous construction of the strong solution. It employs a cut-off and \textit{Galerkin approximation scheme}, utilizing the bounds established in Section 2 to prove both global existence and convergence.

    \item \textit{Section 5: Appendix} \\
    The appendix contains technical proofs regarding the decay estimates, Kružkov doubling-of-variables estimates and non-adapted stochastic integrals.
\end{itemize}

\section{A Priori Estimates and Asymptotic Behavior}	

\subsection{Set Up and A Priori Estimates}
We introduce some notations for use of the later context. 
First, since $u^r$ is not smooth, we use the smoothed rarefaction wave $\bar u$ to approximate $u^r$.
	\bae\label{background}
	\left\{\begin{aligned}
		& \bar u_t + \partial_x f(\bar u) = 0,\\
		&\bar u(0,x)=\bar u_0(x)=\frac{u_{+}+u_{-}}{2}+\frac{u_{+}-u_{-}}{2} \frac{e^x-e^{-x}}{e^x+e^{-x}}.
	\end{aligned}\right.
	\eae
Now that we are entering the dense technical proofs, the following abbreviations will be adopted to save space. Set $B(r)=(-r,r)$, $L^p_r\coloneqq L^p(B(r)) $ and $H^s_r \coloneqq H^s_0(B(r))$. We shall omit the natural embedding from $H^s_r$ into $H^s(\R)$ in the later context. Let $X$ be a normed space with norm $\|\cdot\|_X$ and $g=(g_i)_{i=1}^\infty: X^{\mathbb N_+ }\rightarrow\mathbb R,$ we set
	\begin{equation*}
		\|g\|_{l^p(X)}\coloneqq \left(\sum_{i=1}^\infty \| g_i\|_X^p \right)^\frac{1}{p}, \quad g\in l^p(X).
	\end{equation*}  
Adopting a minor abuse of notation, we denote $\|\cdot\|_p$ as $\|\cdot\|_{L^p(\R)}$, $\|\cdot\|$ as $\|\cdot\|_{L^2(\R)}$ (or $\|\cdot\|_{L^2_N}$ respectively), and $\langle\cdot,\cdot \rangle$ as $\langle\cdot,\cdot \rangle_{L^2(\R)}$(or $\langle\cdot,\cdot \rangle_{L^2_N}$ respectively) in the following sections. Throughout the paper, we denote by $C$ a generic positive constant that may depend on the fixed parameters of the system (the flux convexity $\alpha$, the viscosity lower bound $\mu$, the noise coefficients $\sigma$, and the boundary values $u_{\pm}$), but is independent of the time variable $t$, the approximation parameters (such as $N, k, \delta, \epsilon$ in Sections 3 and 4), and the stochastic sample $\omega$, unless otherwise specified. The value of $C$ may change from line to line. However, to rigorously establish the global existence in the presence of conservative noise, the precise dependence on the viscosity $\mu$ and the initial data size $\mathfrak{M}$ is critical. Therefore, in the stability threshold estimates (specifically Lemma \ref{lm2.4} and Theorem \ref{th2.2}), we will identify the dependencies explicitly (e.g., denoting $C^*(\mathfrak{M})$ or $C_{\mu}$) to track the competition between the viscous dissipation and the stochastic energy injection.
\

Since the Riemann initial data $u_0$ of the conservative SPDE \eqref{conservative} is not integrable, we can't solve it directly in a proper space. By convention, we work on the perturbed equation instead. Let $\phi=u-\bar u$. Then $\phi$ satisfies
$$
d \phi+\partial_{x}\left[f(\phi+\bar{u})-f(\bar u)\right] d t=\partial_{x}\left[b(\phi+\bar{u}) \partial_x(\phi+\bar{u}) \right]d t+ \partial_{x}\left[\sum_{i=1}^\infty\sigma_{i}(\phi+\bar{u}) \circ d B_{i}(t)\right],
$$
rewriting it to It\^o's form gives that
\bae\label{perturbation}
\left\{\begin{aligned}
	& d \phi+\partial_x[f(\phi+\bar{u})-f(\bar u)]d t \\
	&= \partial_{x}\left[\left(b(\phi+\bar{u})+\frac{1}{2} \sum_{i=1}^{\infty} \sigma_{i}^{\prime 2}(\phi+\bar{u})\right)(\phi+\bar{u})_{x}\right] d t+ \partial_{x}\left[\sum_{i=1}^{\infty}\sigma_{i}\left(\phi +\bar{u}\right)\cdot d B_{i}(t)\right],\\
	&\phi(0)=u(0)-\bar u(0)=:\phi_0.
\end{aligned}\right.
\eae
In the Introduction, we referred to the standard solution spaces. Henceforth, we denote the solution space for the perturbation equation \eqref{perturbation}
\begin{equation*}
	X_i(T) \coloneqq L ^2\left( [0,T]\times \Omega,\dif t \otimes d\PM , H^i(\R) \right)\bigcap {C\left( [0,T], H^{i-1}(\R) \right)}\quad i=1,2.
\end{equation*}
Before we formulate the well-posedness results of the equation \eqref{perturbation}, we give the definition of solutions.
\begin{dfn} \label{def:weak_solution}
A continuous $L^2(\mathbb{R})$-valued $\mathcal{F}_t$-adapted process $(\phi(t))_{t \in [0, T]}$ is called a \textit{(probability strong) solution} of \eqref{perturbation}, if we have $\bar{\phi} \in X_1(T)$ and for every test function $\varphi \in C_c^\infty(\mathbb{R})$, the following equality holds {\rm{a.s.}} for all $t \in [0, T]$:
\begin{equation*} \label{eq:weak_form}
\begin{aligned}
\langle \phi(t), \varphi \rangle &= \langle \phi(0), \varphi \rangle + \int_0^t \left\langle f(\bar{\phi} + \bar{u}) - f(\bar{u}), \partial_x \varphi \right\rangle ds \\
&\quad - \int_0^t \left\langle b(\bar{\phi} + \bar{u})\partial_x(\bar{\phi} + \bar{u}), \partial_x \varphi \right\rangle  ds \\
&\quad - \sum_{i=1}^\infty \int_0^t \left\langle \sigma_i(\bar{\phi} + \bar{u}), \partial_x \varphi \right\rangle \circ dB_i(s),
\end{aligned}
\end{equation*}
where  $\bar{\phi}$ is any $H^1(\mathbb{R})$-valued progressively measurable $dt \otimes P$-version of $\phi$.
\end{dfn}
\begin{remark}
    Note that while we establish the solution is ``strong'' in the probabilistic sense (adapted to the filtration), Definition \ref{def:weak_solution} corresponds to a ``weak'' solution in the analytic {\rm{(PDE)}} sense.
\end{remark}
 The following properties of $\bar u$ are basic in the subsequent sections.
	\begin{prp}[{\cite[Lemma 2.1]{matsumura1985stability}}]\label{lm2.1}
		${}$
 Let $\bar u$ be the smoothed rarefaction wave defined in \eqref{background}. Then the following hold for any $1\leq p\leq +\infty$,
		\begin{enumerate}
			\item[\rm{(i)}] {$u_-\leq \bar u\leq u_+,$ $\,\bar u_x\geq 0$,}
			\item[\rm{(ii)}] { $\|\bar u_x\|_{L^p(\R)}\leq C (1+t)^{\frac{-p+1}{p}} $,}
			\item [\rm{(iii)}]{$\|\bar u_{xx}\|_{L^p(\R)}\leq C (1+t)^{-1} $.}
		\end{enumerate}
	\end{prp}

We first provide a priori estimates for the solution $\phi$ of perturbation equation \eqref{perturbation}.
\begin{Lemma}\label{lmL^p}
		Let $\phi\in X_1(T)$ be the solution of \eqref{perturbation}. Then for any $T>0$ we have the following $L^p$ estimates.
		\begin{enumerate}
		    \item [\rm{1.}] {\noindent If $p=2$, then  \text { {\rm{a.s.}}}
            \begin{equation}\label{real-L^2-as}
			\begin{aligned}
				& \|\phi(t)\|^2 + \int_0^t \|\phi_x\|^2 ds + \int_0^t\int_{\mathbb{R}}\phi^2  \bar{u}_{x} d x ds\\
                \leq&   \| \phi_0 \|^2 +C \ln (1+t)+ 2 \int_0^t \int_ \mathbb R \sum_{i=1}^{\infty} \partial_x \sigma_i (\phi +\bar u) |\phi|^{p-2}  \dif x~\dif B_i(s),
			\end{aligned}
            \end{equation}
        and
		\begin{equation}\label{real-L^2}
			\begin{aligned}
				&\mathbb E \left[ \|\phi(t)\|^2 + \int_0^t \|\phi_x\|^2 ds + \int_0^t\int_{\mathbb{R}}\phi^2  \bar{u}_{x} d x ds\right]\leq   \| \phi_0 \|^2 +C \ln (1+t).
			\end{aligned}
		\end{equation}}
          \item [\rm{2.}]{If $2< p<\infty,$ then \text { {\rm{a.s.}}}
		\begin{equation}\label{L^p}
			\begin{aligned}
				& \|\phi (t)\|_p^{p} + \int_0^t \int|\phi |^{p} \bar{u}_{x} d x d s + \int_0^t\int_\R|\phi |^{p-2} \phi^{2}_{x} d x d s \\
				&\leq p C \int_0^t (1+s)^{-\frac{p}{2}} ds+ p C\int_0^t (1+s)^{-1}\|\phi \|_p^{p}ds + p \int_0^t \int_ \mathbb R \sum_{i=1}^{\infty} |\phi|^{p-2}\phi\partial_x \sigma_i (\phi +\bar u)   \dif x~\dif B_i(s).
			\end{aligned}
		\end{equation}}
         \item [\rm{3.}]{If $p=\infty$, then \text { {\rm{a.s.}}}
	\begin{equation}\label{mp}
		\begin{aligned}
			\sup_{0\le t\le T}	\|\phi(t)\|_{L^\infty(\R)} \leq  \| \phi_{0} \|_{L^\infty(\R)}.
		\end{aligned}
	\end{equation}}
		\end{enumerate}

\end{Lemma}
\begin{proof}[Formal Derivation]
   We present the formal energy balance to illustrate the dissipation mechanism imposed by the strictly convex flux. The rigorous justification is performed on the Galerkin cut-off approximations in Section \ref{sec cutoff} (Lemma \ref{Lemma4.3}).

Applying It\^o's formula to the $L^p$ norm $\frac{1}{p}\|\phi(t)\|_p^p$ where $p\geq2,$ we obtain:
\begin{equation*} \label{eq:Ito_Lp}
\begin{aligned}
    \frac{1}{p} d \int_{\mathbb{R}} |\phi|^p dx = & \underbrace{-\int_{\mathbb{R}} \partial_x [f(\phi+\bar{u}) - f(\bar{u})] |\phi|^{p-2}\phi dx}_{I_{\rm{flux}}} + \underbrace{\int_{\mathbb{R}} \mathcal{L}(\phi) |\phi|^{p-2}\phi dx}_{I_{\rm{diss}}}\\
    &+ \frac{p-1}{2} \int_{\mathbb{R}} \sum_{i=1}^\infty |\sigma_i'(\phi+\bar{u}) (\phi_x + \bar{u}_x)|^2 |\phi|^{p-2} dx dt + dM_p(t),
\end{aligned}
\end{equation*}
where $\mathcal{L}(\phi)$ represents the second-order viscous terms and It\^o's correction terms, and $$dM_p(t)= \int_ \mathbb R \sum_{i=1}^{\infty} |\phi|^{p-2}\phi\partial_x \sigma_i (\phi +\bar u)   \dif x~\dif B_i(s)$$ is a local martingale.

\textbf{Step 1. The Convex Flux $I_{\rm{flux}}$.}
Using Hypothesis (H2), the strict convexity of the flux $f'' \ge \alpha$ and the rarefaction structure $\bar {u}>0$, by the integration by parts on $I_{{\rm{flux}}}$, we have 
\[
I_{\rm{flux}} = \int_{\mathbb{R}} \left( f(\phi+\bar{u}) - f(\bar{u}) - f'(\bar{u})\phi \right) \partial_x(|\phi|^{p-2}\phi) dx \le -\frac{\alpha^*}{p} \int_{\mathbb{R}} |\phi|^p \bar{u}_x dx.
\]
This term provides the crucial decay mechanism appearing in the first estimate \eqref{real-L^2-as} and \eqref{real-L^2}.

\textbf{Step 2. Viscosity $I_{\rm{diss}}$.}
The term $I_{\rm{diss}}$ combines the physical viscosity and the It\^o correction. A direct computation shows that the parabolic term behaves as
\[
I_{\rm{diss}} \approx - \int_{\mathbb{R}} \left( b(\cdot) - \frac{1}{2}\sum_{i=1}^\infty |\sigma_i'(\cdot)|^2 \right) \phi_x^2 |\phi|^{p-2} dx.
\]
Under Hypothesis (H1) and (H3), the viscosity $\mu$ dominates the noise intensity. Using Young's inequality to absorb lower-order terms, we obtain the differential inequality
\begin{equation} \label{eq:diff_ineq_Lp}
    \frac{1}{p} d\|\phi\|_p^p + C \int_{\mathbb{R}} |\phi|^{p-2}\phi_x^2 dx dt \le C(1+t)^{-1} \|\phi\|_p^p dt + C (1+t)^{-p/2} dt + dM_p(t).
\end{equation}
Integrating this inequality yields the $L^2$ and $L^p$ bounds.

\textbf{Step 3. The $L^\infty$ estimation.}
For the $L^\infty$ estimate, we observe that as $p \to \infty$, the forcing terms in \eqref{eq:diff_ineq_Lp} vanish relative to the norm (rigorously detailed in the proof of Lemma 4.1), we recover the pathwise bound
\[
\|\phi(t)\|_\infty \le \|\phi_0\|_\infty, \quad \text { {\rm{a.s.}}.}
\]
This completes the formal derivation.
\end{proof}
\begin{remark}
   This pathwise The $L^\infty$ estimation \eqref{mp} is consistent with the inviscid theory, recovering the estimates established for rough fluxes in {\rm{\cite[Proposition 2.1]{lions2013scalar}}} and {\rm{\cite[Proposition 2.2]{gess2017long}}}.
\end{remark}
\begin{Lemma}[\cite{kawashima2004lp}]\label{ii} Let $2\leq p<\infty$ and $1\leq q\leq p$. Then for any $\phi\in L^p(\R)$, it holds that
	\bae\label{ii1}
		\|\phi\|_{p}\leq C(p,q,n)\left\|\nabla(|\phi|^{\frac{p}{2}})\right\|^{\frac{2\gamma}{1+\gamma p}}\|\phi\|^{\frac{1}{1+\gamma p}}_{q},
	\eae
	 where $\gamma=(n/2)(1/q-1/p)$ and $C(p,q,n)$ is a positive constant.
\end{Lemma}
We present the decay estimate of $\mathbb E \|\phi\|_p^p$. The derivation relies on weighted energy estimates and the Burkholder-Davis-Gundy inequality. 
\begin{prp}[$L^p$ decay]\label{t1}
 Let $\phi\in X_1(T) $ be the solution of \eqref{perturbation}. Then for any $4< p< \infty$ 
	\begin{eqnarray}\label{5.20}
		\mathbb{E} \| \phi \|_p^p \le C_{p}(1 + t)^{- \frac{p-2}{4}}\ln^\frac{p}{2}(2+t).
	\end{eqnarray}
\end{prp}
\begin{proof}
	
	For $L^p$ decay \eqref{5.20}, multiplying \eqref{L^p} by  $(2+t)^\frac{p}{2}$, we obtain  
	\begin{eqnarray} \label{5.21}
		\begin{array}{lll}
			&&\displaystyle d(2+t)^\frac{p}{2} \| \phi \|_ {p} ^p + (2+t)^\frac{p}{2}  \int_\mathbb R \left(|\phi|^p \bar u_x + |\phi|^{p-2} \phi_x^2\right) \dif x ~\dif t  \\
			&\leq& \displaystyle C_1 \dif t  + C_1(2+t)^{\frac{p}{2}-1}\| \phi \|_ {p} ^p \dif t  + (2+t)^{\frac{p}{2}}\dif M_p(t),
		\end{array}
	\end{eqnarray}
    where
    \bae\label{Mp def}
M_p(t):=p \int_0^t \int_ \mathbb R \sum_{i=1}^{\infty} \partial_x \sigma_i (\phi +\bar u) |\phi|^{p-2}\phi  \dif x~\dif B_i(s).
    \eae
	Thanks to the inequality \eqref{ii1}, we have 
	\begin{eqnarray*}\label{gn2}
		\begin{array}{lll}
			\displaystyle	 (2+t)^{\frac{p}{2}-1}\| \phi \|_ {p}^p &\leq & \displaystyle C_p (2+t)^{\frac{p}{2}-1} \|  \partial_x(\phi^{\frac{p}{2}})  \|^{\frac{2p-4}{p+2}} \| \phi \|^{\frac{4p}{p+2}}\\
			&\leq & \displaystyle \frac12 (2+t)^\frac{p}{2} \int_\mathbb R |\phi|^{p-2} \phi_x^2 \dif x  + C_p (2+t)^{\frac{p-2}4} \| \phi \|^p.
		\end{array}
	\end{eqnarray*}
	Integrating (\ref{5.21}) over $[0,t]$ gives that  
	\begin{equation} \label{5.22}
		(2+t)^\frac{p}{2} \| \phi(t) \|_ {p} ^p  \leq \| \phi_0 \|_ {p} ^p+ C_1t  + C_1 \int_0^t (2+s)^{\frac{p-2}4} \| \phi \|^p ds +  \int_0^t (2+s)^\frac{p}{2} \dif M_p(s).
	\end{equation}
	Taking expectation on both sides of \eqref{5.22}, we have 
	\begin{equation} \label{5.22new}
		\mathbb E\| \phi(t) \|_ {p} ^p  \leq (2+t)^{-\frac{p}{2}}\| \phi_0 \|_ {p} ^p+  C(2+t)^{\frac{2-p}{2} } + C (2+t)^{-\frac{p}{2}} \int_0^t (2+s)^{\frac{p-2}4} \mathbb E\| \phi \|^p ds.
	\end{equation}

Next we shall use the Burkholder-Davis-Gundy inequality and the decay property $\bar{u}_x\le \frac{C}{2+t}$ to estimate $\mathbb E\|\phi\|^p$ for any $p>2$.

Note that $\|\phi\|^2$ is a semi-martingale, the Burkholder-Davis-Gundy inequality and \eqref{real-L^2-as} yield that 
\begin{equation*}
	\begin{aligned}
		\mathbb E\|\phi\|^{p} &\le \mathbb E(\sup_{0\le s\le t}\|\phi\|^2)^{\frac{p}{2}} \le  C\left(\ln^{\frac{p}{2}} (2+t)+
		\mathbb E\left[ \int_0^t  \sum_{k}\left(
		\int_{\mathbb R} \phi \sigma_{k}(\phi+\bar{u})_{x} d x \right)^2 d s
		\right]^{\frac{p}{4}}\right) .
	\end{aligned}
\end{equation*}
Direct calculation gives
\baee\label{preBDG}
		&\sum_{k}\Big(\int_{\mathbb R} \phi \sigma_{k}(\phi+\bar{u})_{x} d x\Big)^{2}=\sum_{k}\Big(\int_{\mathbb R} \phi_{x} \sigma_{k}(\phi+\bar{u}) d x\Big)^{2}\\
		&=\sum_{k}\left(\int_{\mathbb R}(\phi+\bar{u})_{x} \sigma_{k}(\phi+\bar{u}) d x-\int_{\mathbb R} \bar{u}_{x} \sigma_{k}(\phi+\bar{u}) d x\right)^{2}\\
		&=\sum_{k}\left\{\int_{\mathbb R}(\phi+\bar{u})_{x} \sigma_{k}(\phi+\bar{u}) d x-\int_{\mathbb R} \bar{u}_{x} \sigma_{k}(\bar{u}) d x-\int_{\mathbb R}\left[\sigma_{k}(\phi+\bar{u})-\sigma_{k}(\bar{u}) \right] \bar{u}_{x} d x\right\}^{2}\\
		&=\sum_{k}\left[\int_{\mathbb R}\left[\sigma_{k}(\phi+\bar{u})-\sigma_{k}(\bar{u})\right] \bar{u}_{x} d x\right]^{2}=\sum_{k}\left(\int_{\R}\left(\int_{\bar{u}}^{\bar{u}+\phi} \sigma_{k}^{\prime}(y) d y\right) \bar{u}_{x} d x\right)^{2}\\
		&\leq \sum_{k}\Big(\int_{\mathbb R}\|\sigma_k^\prime\|_{L^\infty(\R)}|\phi| \bar{u}_{x} d x\Big)^{2}  \leq 2 \left\|\sigma^{\prime}\right\|^{2}_{l^2(H^{1}_M)}\left(\int_{\mathbb R}|\phi| \bar{u}_{x} d x\right)^{2}.
\eaee
Thus we have
\begin{equation}\label{bdg1}
	\begin{aligned}
		\mathbb E\|\phi\|^{p} &\le \mathbb E(\sup_{0\le s\le t}\|\phi\|^2)^{\frac{p}{2}} \le  C\left(\ln^{\frac{p}{2}} (2+t)+
		\mathbb E\left[ \int_0^t  \sum_{k}\left(
		\int_{\mathbb R} \phi \sigma_{k}(\phi+\bar{u})_{x} d x \right)^2 d s
		\right]^{\frac{p}{4}}\right) \\
		&\le  C\left(\ln^{\frac{p}{2}} (2+t)+2 \left\|\sigma^{\prime}\right\|^{2}_{l^2(H^{1}_M)}
	\mathbb E\left[ \int_0^t\left(\int_{\mathbb R}|\phi| \bar{u}_{x} d x\right)^{2}
		ds\right]^{\frac{p}{4}}\right).
	\end{aligned}
\end{equation}
 Setting $\beta=\frac{p-4}{p-2}\in(0,1)$, by Young's inequality, we compute that
\begin{equation}\label{ep}
	\begin{aligned}
		& \mathbb{E}\left[\int_{0}^{t}\left(\int_R |\phi| \bar u_{x} d x\right)^{2} d s\right]^{\frac{p}{4}} 
		= \mathbb{E}\left[\int_{0}^{t}\left(\int_R |\phi|^{\beta} (|\phi|^{1-\beta}\bar  u_{x}^{\frac12})\bar  u_{x}^{\frac12} d x\right)^{2}ds\right]^{\frac{p}{4}}\\
		\le&\mathbb E \left[\int_{0}^{t}\left(\int_R |\phi|^{2(1-\beta)} \bar u_{x} d x\right)\left\|\bar u_{x}\right\|_{\frac{1}{1-\beta}}\|\phi\|^{2 \beta}d s\right]^{\frac{p}{4}}\\
		\leq & \mathbb E \left\{\int_{0}^{t}\left[\int_{\mathbb R}|\phi|^{2(1-\beta)} \bar u_{x} d x\right]^{\frac{1}{1-\beta}} d s+C\left[\int_{0}^{t}(2+s)^{-1} \| \phi\| ^{2} d s\right]^{\frac{\beta p}{4-(1-\beta)p}}\right\}\\
		\leq& C\mathbb{E} \int_{0}^{t}\int_{\mathbb R}|\phi|^2 \bar{u}_{x} d s+C\E\left[\int_{0}^{t}(2+s)^{-1} \| \phi\| ^{2} d s\right]^{\frac p2}\\
		\leq& C\ln(2+t)+C_p+\frac{1}{2} \mathbb{E} \sup _{0 \leq s \leq t}\|\phi\|^p .
	\end{aligned}
\end{equation}
  Substituting \eqref{ep} into \eqref{bdg1}, we have
\begin{equation}\label{bdg}
	\begin{aligned}
		\mathbb E \sup _{0 \leq s \leq t}\|\phi\|^{p} &\leq C_p \ln^\frac{p}{2}(2+t).
	\end{aligned}
\end{equation}
Substituting \eqref{bdg} into \eqref{5.22new} gives \eqref{5.20}. Therefore, the proof is complete. 
\end{proof}

\subsection{High-Order Energy Estimates}
We define the local stability threshold function \bae\label{C*}
C^*(R) = C^*(R,\mu, b',\sigma''):= C_0 \sqrt{\frac{\mu^2}{\|b'\|_{L^\infty_R}^2 + \mu  \|\sigma''\|_{l^2(L^\infty_R)}^2 }},
\eae
where $C_0$ is a universal constant.
\begin{Lemma}[$H^2$ regularity]\label{lm2.4}
 Let $\phi\in X_2(T)$ be the solution of  \eqref{perturbation} and  $\mathfrak M=\|\phi_0\|_\infty+u_+-u_-$.  If 
	\bae\label{small data treshhold}
	  \|\phi_0\|_{\infty} < C^*(\mathfrak M),
	\eae
	then 
\bae\label{H^2}
\E \sup_{0 \leq s \leq T}\left\|\phi (s)\right\|_{H^1(\R)}^{2}+\E\int_0^T\left\|\phi_{x}\right\|_{H^1(\R)}^2 ds\leq \left\|\phi_0\right\|_{H^1(\R)}^{2}+C\ln (2+T).
\eae
\end{Lemma}
\begin{proof}
  Since  $\phi\in L^2_{\rm{loc}}\big((0,\infty),H^2(\mathbb R)\big),$  differentiating with respect to $x$ on both sides of \eqref{perturbation} and multiplying $\phi_x$. Applying It\^o's formula, we get
    \bae\label{4.71}
 \frac{1}{2} d \phi_x^2 =&-\phi_x \partial_{x x}^2[f(\phi+\bar{u})-f(\bar{u})] d t \\
&+\phi_x \partial_{x x}^2\big(b(u)(\phi+\bar{u})_x\big)dt+\frac{1}{2} \phi_x \sum_{i=1}^\infty \partial_{x x}^2\left[\sigma_i^{\prime 2}(u)(\phi+\bar{u})_x\right]dt \\
&+\frac{1}{2} \sum_{i=1}^\infty\left[\left(\sigma_i^{\prime}(u)\left(\phi_{x x}+\bar{u}_{x x}\right)+\sigma_i^{\prime \prime}(u)\left(\phi_x+\bar{u}_x\right)^2\right]^2 d t\right. \\
&+\sum_{i=1}^\infty \phi_x\left[\left(\sigma_i^{\prime}(u)\left(\phi_{x x}+\bar{u}_{x x}\right)+\sigma_i^{\prime \prime}(u)\left(\phi_x+\bar{u}_x\right)^2\right] d B_i(t)\right. \\
:=&\left(K_1+K_2+\sum_{i=1}^\infty K_3^i+\sum_{i=1}^\infty K_4^i\right) d t+d M{(t)},\\
    \eae
where 
\baee
M(t)=\sum_{i=1}^{\infty} \int_0^t \int_{\mathbb{R}} \phi_x \partial_{x x}^2 \sigma_i(u) d x d B_i(s).
\eaee
In the following, we estimate $K_1, K_2, K_3^i,K_4^i$ and $M(t)$. First, we have
\baee
K_1= & -\phi_x \partial_{x x}^2(f(u)-f(\bar{u})) \\
= & -\phi_x \partial_x\left(f^{\prime}(u) \phi_x\right)+\phi_x \partial_x\left[\left(f^{\prime}(u)-f^{\prime}(\bar{u})\right) \bar{u}_x\right] \\
= & -\partial_x\left[\frac{1}{2} f^{\prime}(\bar{u}) \phi_x^2+\left(f^{\prime}(u)-f^{\prime}(\bar{u})\right) \phi_x^2\right] \\
& -\frac{1}{2} f^{\prime \prime}(\bar{u}) \bar{u}_x \phi_x^2+\left[f^{\prime}(u)-f^{\prime}(\bar{u})\right] \phi_x \phi_{x x}.
\eaee
Then
\baee
\int_{\mathbb{R}} K_1 d x\leq -\frac{1}{2} \int f^{\prime \prime}(\bar{u}) \bar{u}_x \phi_x^2 d x+C \int\left|\phi \phi_x \phi_{x x}\right| d x.
\eaee
For $q>3$, using H\"older's inequality, Gagliardo–Nirenberg's interpolation inequality and Young's inequality, one has
\bae\label{K1}
& \int_{\mathbb{R}} K_1 d x\leq \int_{\mathbb R}\left|\phi \phi_x \phi_{x x}\right| d x \leq\|\phi\|_q\left\|\phi_x\right\|_\frac{2 q}{q-2}\left\|\phi_{x x}\right\| \\
& \leq C\|\phi\|_q^{\frac{5 q}{3q+2}}\left\|\phi_{x x}\right\| ^\frac{4 q+6}{3 q+2}  \leq C\|\phi\|_q^{\frac{5 q}{q-1}}+\varepsilon\left\|\phi_{x x}\right\|^2.
\eae
Next,
\bae\label{K2}
\int_{\mathbb{R}} K_2 d x= & \int_{\mathbb{R}} \phi_x \partial_{x x}^2\big(b(u)(\phi+\bar{u})_x\big) d x \\
= & -\int_{\mathbb{R}} \phi_{x x}\left(b^{\prime}(u)\left(\phi_{x}+ \bar{u}_x\right)^2+b(u)\left(\phi_{x x}+\bar{u}_{x x}\right)\right) d x \\
\leq &\, \|b^{\prime}\|_{L^\infty_{\mathfrak M}}\int_{\mathbb{R}}|\phi_{xx}|\left( \phi_x^2+ \bar{u}_x^2\right) d x 
-\int_{\R} b(u) \phi_{x x}^2 d x+\int_\R b(u)\big(\frac14 \phi_{x x}^2+C \bar{u}_{x x}^2\big) d x \\
\leq &\,\frac \mu2 \|\phi_{xx}\|^2+ C\frac{\|b^{\prime}\|^2_{L^\infty_{\mathfrak M}}}{\mu}(\|\phi_{xx}\|^2+\|\bar u_x\|_4^4) -\frac34 \mu \|\phi_{x x}\|^2 +C _{u_0, b} \| \bar{u}_{x x} \|^2 \\
\leq &- \frac{\mu}{4}\left\|\phi_{x x}\right\|^2+C\frac{\|b^{\prime}\|^2_{L^\infty_{\mathfrak M}}}{\mu}\left\|\phi_x\right\|_4^4+C _{u_0, b}\| \bar{u}_x \|_4^4 +\frac{\mu}{4}\left\|\phi_{x x}\right\|^2+C _{u_0, b} \| \bar{u}_{x x} \|^2.
\eae
For $K_3^i$ and $K_4^i$, we have
\bae\label{K3i}
2 K_3^i :=& \phi_x \partial_{x x}^2\left[\sigma_i^{\prime 2}(u)\left(\phi_x{+} \bar{u}_x\right]\right.=-\sigma_i^{\prime 2}(u) \phi_{x x}\left(\phi_{x x}+\bar{u}_{x x}\right) -2 \sigma_i^{\prime}(u) \sigma_i^{\prime \prime}(u) \phi_{x x}\left(\phi_x+\bar{u}_x\right)^2 \\
\eae
and
\bae\label{K4i}
2 K_4^i :=& {\left[\left(\sigma_i^{\prime}(u)\left(\phi_{x x}+\bar{u}_{x x}\right)+\sigma_i^{\prime \prime}(u)\left(\phi_x+\bar{u}_x\right)^2\right]^2\right.} \\
=& \sigma_i^{\prime 2}(u)\left(\phi_{x x}^2+\bar{u}_{x x}^2+2 \phi_{x x} \bar{u}_{x x}\right)+\left(\sigma_i^{\prime \prime}(u)\right)^2\left(\phi_x+\bar{u}_x\right)^4 \\
&+2 \sigma_i^{\prime}(u) \sigma_i^{\prime \prime}(n)\left(\phi_{x x}+\bar{u}_{x x}\right)\left(\phi_x+\bar{u}_x\right)^2 \\
=& \sigma_i^{\prime 2}(u) \phi_{x x}\left(\phi_{x x}+\bar{u}_{x x}\right)+2 \sigma_i^{\prime}(u) \sigma_i^{\prime \prime}(u) \phi_{x x}\left(\phi_x+\bar{u}_x\right)^2 \\
&+\left(\sigma_i^{\prime \prime}(u)\right)^2\left(\phi_x+\bar{u}_x\right)^4+(\sigma_i^{\prime }(u))^2\left(\bar{u}_{x x}^2+\phi_{x x} \bar{u}_{x x}\right) \\
&+2 \sigma_i^{\prime}(u) \sigma_i^{\prime \prime}(u) \bar{u}_{x x}\left(\phi_x+\bar{u}_x\right)^2.
\eae
Using \eqref{K3i} and \eqref{K4i}, one has
\bae\label{K34}
&2\sum_{i=1}^\infty \int_{\mathbb{R}} (K_3^i+K_4^i) d x \\
= &\sum_{i=1}^\infty \int_{\mathbb{R}}\left(\sigma_i^{\prime \prime}(u)\right)^2\left(\phi_x+\bar{u}_x\right)^4+(\sigma_{i}^{\prime }(u))^2\left(\bar{u}_{x x}^2+\phi_{x x} \bar{u}_{x x}\right) d x \\
& +\sum_{i=1}^\infty\int_{\mathbb{R}}\left(2 \sigma_i^{\prime}(u) \bar{u}_{x x}\right)\left(\sigma_i^{\prime \prime}(u)\left(\phi_x+\bar{u}_x\right)^2\right) d x \\
\leq & C \sum_{i=1}^\infty\int_{\mathbb{R}}\left(\sigma_i^{\prime \prime}(u)\right)^2\left(\phi_x^4+\bar{u}_x^4\right) d x+\sum_{i=1}^\infty\int_{\mathbb{R}}\left(\sigma^{\prime}_i(u)\right)^2\left(\varepsilon \phi_{x x}^2+C_{\varepsilon} \bar{u}_{x x}^2\right) d x \\
\leq & C \|\sigma''\|^2_{l^2(L^\infty_{\mathfrak M})} \left\|\phi_x\right\|_4^4+C_{u_0, \sigma}\left\|\bar{u}_x\right\|_4^4+\frac{\mu}{8}\left\|\phi_{x x}\right\|^2+C_{u_0, \sigma,b}\left\|\bar{u}_{x x}\right\|^2.
\eae

Finally, substituting \eqref{K1}, \eqref{K2} and\eqref{K34} into \eqref{4.71}, we conclude  that
\bae\label{quasi deriv}
d\left\|\phi_x\right\|^2 \leq & -\frac{\mu}{8}\left\|\phi_{xx}\right\|^2 d t+ C\big({\mu}^{-1}{\|b^{\prime}\|^2_{L^\infty_{\mathfrak M}}}+ \|\sigma''\|^2_{l^2(L^\infty_{\mathfrak M})}\big)\left\|\phi_x\right\|^4_4 d t+C\|\phi\|_q^{\frac{5 q}{q-1}} d t \\
& +C_{ u_0, b,\sigma}\left(\| \bar{u}_x\|_4^4+\| \bar{u}_{x x} \|^2\right) d t+2 d M_t.
\eae
Notice that, for all $0\leq t\leq T,$ one has
\bae\label{pre Quadra}
\langle M\rangle_t=&\int_0^t\sum_{i=1}^{\infty}\Big(\int_{\mathbb{R}} \phi_x \partial_{x x}^2 \sigma_i(u) d x  \Big)^2 ds=\int_0^t\sum_{i=1}^{\infty}\Big[\int_{\mathbb{R}} \phi_{xx}  \sigma_i^\prime(u)(\bar u_x+ \phi_x) d x  \Big]^2ds \\
\leq & C_{u_0,\sigma}\int_0^t(\|\bar u_x\|^2+\|\phi_x\|^2) \|\phi_{xx}\|^2ds,\quad {\rm{a.s.}},
\eae
where the last inequality is followed from {\rm{(H3)}} and \eqref{mp}. Using Proposition \ref{lm2.1}, \eqref{pre Quadra} and $\phi\in X_2(T)$, we conclude that $M(t)$ is a continuous local martingale. 

Utilizing the Gagliardo-Nirenberg interpolation inequality, we have
\begin{equation} \label{eq:GN_interpolation}
    \|\phi_x\|_{4}^4 \leq C \|\phi\|_{\infty}^2 \|\phi_{xx}\|^2, \quad {\rm{a.s.}}
\end{equation}
Crucially, the pathwise maximum principle \eqref{mp} established in Lemma \ref{lmL^p}  holds independently of the higher-order regularity. Therefore, almost surely for all $t \ge 0$, we have $||\phi(t)||_{L^{\infty}}\le||\phi_{0}||_{L^{\infty}}\le\mathfrak{M}$. Substituting \eqref{eq:GN_interpolation} into \eqref{quasi deriv} and utilizing this a priori $L^\infty$ bound yields
\bae\label{middle deriv}
    d\|\phi_x\|^2 + \left( \frac{\mu}{8} - C\big({\mu}^{-1}{\|b^{\prime}\|^2_{L^\infty_{\mathfrak M}}}+ \|\sigma''\|^2_{l^2(L^\infty_{\mathfrak M})}\big) \mathfrak M^2 \right) \|\phi_{xx}\|^2 dt \leq \mathcal{R}(t) dt + 2 dM_t,\quad {\rm{a.s.}}
\eae
where $\mathcal{R}(t) = C \|\phi\|_q^{\frac{5q}{q-1}} + C_{u_0, b, \sigma} (\|\bar{u}_x\|_4^4 + \|\bar{u}_{xx}\|^2)$ represents the integrable remainder terms.

From the condition \eqref{small data treshhold}
\[
    \|\phi_0\|_{\infty} < C_0 \sqrt{\frac{\mu^2}{{\|b^{\prime}\|^2_{L^\infty_{\mathfrak M}}}+ \mu\|\sigma''\|^2_{l^2(L^\infty_{\mathfrak M})}}}\,,
\]
 the coefficient of the dissipation term in \eqref{middle deriv} is positive $$c_0 := \frac{\mu}{8} - C\big({\mu}^{-1}{\|b^{\prime}\|^2_{L^\infty_{\mathfrak M}}}+ \|\sigma''\|^2_{l^2(L^\infty_{\mathfrak M})}\big) \mathfrak M^2 > 0.$$ \eqref{middle deriv} simplifies to
\[
    d\|\phi_x\|^2 + c_0 \|\phi_{xx}\|^2 dt \leq \mathcal{R}(t) dt + 2 dM_t.
\]

Finally, integrating the above inequality on $[0,t]\times\Omega$ with $q$ large enough and using Proposition \ref{lm2.1}, \ref{t1} and Lemma \ref{lmL^p}, we complete the proof.
\end{proof}

\subsection{The Stochastic Area Inequality}
The area inequality is a key technique for proving the decay rate. For the sake of clarity, we will first present the description of the deterministic case, and the proof will be placed in the Appendix \ref{appen area}.

   \begin{Lemma}\label{dfn area}
   Assume $X(t)\in C^1(\mathbb R_+,\mathbb R_+)$ and satisfies that
    $$
\left\{\begin{array}{l}
\dot{X_t} \leq X_t^\alpha+t^{-\beta}, \\
\int_0^t X_s d s \leq C (1+t)^p,
\end{array}\right.
$$
where $\alpha>2$ and $0\leq p<(\alpha-2) \big( \frac{\beta}{\alpha}\wedge \frac{1}{\alpha-1} \big)$. Then we have the sharp decay estimate $$X(t)=O\big(t^{-( \frac{\beta}{\alpha}\wedge \frac{1}{\alpha-1} )}\big).$$
\end{Lemma}

\ 

    For the stochastic case, assume  $X_t, Y_t\geq0$ are $\mathcal F_t$ adapted stochastic processes that satisfy
	\begin{eqnarray}\label{int-dif}
		\left\{\begin{array}{l}
			d X_t \leq-Y_t d t+ \mathcal C_t (1+t)^{\kappa} X^\alpha_t d t+\mathcal C_t (1+t)^{-\beta} d t+d M_t, \\
			\int_0^t X_s d s \leq \mathcal C_t (1+t)^p+\mathcal C_t,
		\end{array}\right.
	\end{eqnarray}
	where $\alpha>1$, $p\geq 0$ and $\kappa,\beta\in \mathbb R$. $ M_t$ is an $\mathcal F_t$ adapted continuous local martingale and  $\mathcal C_t=\mathcal C_t(\omega)$ is a positive increasing $\mathcal F_t$-adapted stochastic process. Furthermore, we assume that
    \begin{enumerate}
        \item {There exist $\zeta\in \mathbb R$ and $\theta\in(0,1)$ such that
    \bae\label{young para1}
     X_t\leq& \,\mathcal C_t^\theta  (1+t)^{\zeta\theta} Y_t^{1-\theta}.
	\eae}
    \item {  There exist $\xi_j\in \mathbb R,$ $\theta^\prime_j\in[0,1]$ and $\gamma_j>0$, for $j=1,\cdots,\mathfrak n$ such that
    \bae\label{young para2}
	\frac{d}{dt}\langle M\rangle_t \leq& \,C\sum_{j=1}^\mathfrak n (1+t)^{\xi_j }X_t^{\gamma_j\theta^\prime_j} Y_t^{1-\theta^\prime_j} .
	\eae}
    \end{enumerate}
    
   In the following, if $\theta_j^\prime> 0$, we set  $\psi_j:=\gamma_j-1/\theta^\prime_j-1$, $\xi := \max \{ \xi_j : \theta'_j = 0 \}$ and\footnote{In \eqref{rho^*}, we set $\min \emptyset=+\infty.$}
  \bae\label{rho^*}
  \bar\rho :=(1-p)\wedge \Big(\beta\wedge\big(\frac1{\theta}-\zeta\big)-1\Big)\wedge (-\xi)\wedge \min\Big\{\frac{\theta^\prime_j+\xi_j}{\theta^\prime_j \psi_j}:\theta^\prime_j>0\text{ and }\psi_j<0\Big\}.
  \eae

\begin{Lemma}[Stochastic area inequality]\label{stochastic area}
 Let $X_t$ satisfies \eqref{int-dif}-\eqref{young para2}, assume that  
  \begin{enumerate}
      \item[\rm{(A1) }]{ $
\xi_j < \theta^\prime_j(\bar\rho\psi_j -1), \text{ if } \theta^\prime_j>0 \text{ and }\psi_j\geq 0,\, j=1,\cdots, \mathfrak n.$}
      \item[\rm{(A2) }]{  $
\kappa<(\alpha-1)\bar\rho-1.$}
      \item[\rm{(A3) }]{$
\displaystyle\lim_{R \to +\infty} \mathbb P\big(\displaystyle\lim_{t\rightarrow+\infty}C_{t}(\omega) > R\big) = 0.$}  \end{enumerate}
Then for any $\rho<\bar \rho$, almost surely, we have
$X_t =O\big(t^{-\rho}\big).$
\end{Lemma}
	\begin{proof}
        Define the transformed process $ \bar X_t = t^{\rho} X_t$. By It\^o's formula, we have
$$
d \bar X_t = t^{\rho} dX_t + \rho t^{\rho-1} X_t  dt.
$$
Substituting \eqref{int-dif} into the above  inequality, we obtain
\bae\label{Xt transform}
d\bar X_t &\le t^{\rho}\big[- Y_t  + \mathcal C_t t^\kappa(t^{-\rho} \bar X_t)^{\alpha}  + \mathcal C_t t^{-\beta} \big] dt + \rho t^{\rho-1}X_t \, dt + t^{\rho} dM_t \\
       &= -t^{\rho} Y_t \, dt + \mathcal C_tt^{\kappa-(\alpha-1)\rho} \bar X_t^{\alpha} \, dt + \mathcal C_t t^{\rho-\beta} \, dt + \rho t^{\rho-1}X_t \, dt + t^{\rho} dM_t.
\eae
In the sequel, thanks to Assumption {\rm{(A3)}}, we adopt the convention that $\mathcal C_t$ denotes a generic positive process which may change from line to line. Specifically, we absorb all multiplicative constants or adding fixed constant into $\mathcal C_t$ without explicit relabeling. Using \eqref{young para1} and Young's inequality, it holds that
\bae\label{t-1 Xt}
t^{\rho-1}X_t\leq&\, t^{\rho-1}\mathcal C_t^\theta  (1+t)^{\zeta\theta} Y_t^{1-\theta}\\
=& \,\big(t^{\theta \rho-1}\mathcal C_t^\theta  (1+t)^{\zeta\theta}\big)\big(t^\rho Y_t\big)^{1-\theta}\\
=& \, \mathcal C_t t^{\rho-1/\theta} (1+t)^{\zeta}+\frac1{4\rho} t^\rho Y_t\\
\leq&\,  \mathcal C_t t^{\rho-1/\theta+\zeta} +\frac1{4\rho} t^\rho Y_t, \quad \text{for $t\geq1.$}
\eae
Using \eqref{young para2}, for any $y\in\mathbb R,$ $w\geq0$ and $t>1$, a direct computation shows 
\bae\label{qd weighted}
        &  t^{y}\bar{X}_t^{w} t^{2 \rho} \frac{d}{dt}\langle M\rangle_t \\
\leq&\, C t^{y} \bar{X}_t^{w} t^{2 \rho} \sum_{j=1}^\mathfrak n t^{\xi_j }X_t^{\gamma_j\theta^\prime_j} Y_t^{1-\theta^\prime_j}  \\
= &\, C \sum_{j=1}^\mathfrak n t^{y} \bar{X}_t^{w} t^{\rho(1+\theta_j^\prime)} t^{\xi_j }t^{-\rho \gamma_j \theta_j^\prime} \bar{X}_t^{\theta_j^\prime{\gamma_j}} (t^\rho Y_t)^{1-\theta_j^\prime} \\
= &\, C\sum_{j,\theta_j^\prime> 0}   t^{y+\rho+\xi_j+\rho\theta_j^\prime(1-\gamma_j)} \bar{X}_t^{w+\theta_j^\prime\gamma_j}(t^\rho Y_t)^{1-\theta^\prime_j} + C\sum_{j,\theta_j^\prime= 0}   t^{y+\rho+\xi_j} \bar{X}_t^{w}t^\rho Y_t \\
\leq&\,C \sum_{j,\theta_j^\prime> 0} t^{(y+\rho+\xi_j)/\theta^\prime_j+\rho(1-\gamma_j)} \bar{X}_t^{w/\theta^\prime_j+\gamma_j}+\frac12 t^\rho Y_t + C  t^{y+\rho+\xi} \bar{X}_t^{w}t^\rho Y_t,
    \eae
     where the last inequality follows from Young’s inequality.

Substituting \eqref{t-1 Xt} and \eqref{qd weighted} into \eqref{Xt transform}, one has
\bae\label{middle Xt}
d\bar X_t\leq& -\frac34 t^{\rho} Y_t \, dt + \mathcal C_tt^{\kappa-(\alpha-1)\rho} \bar X_t^{\alpha} \, dt + \mathcal C_t (t^{\rho-\beta}+t^{\rho-1/\theta+\zeta}) \, dt  \\
&+\frac12t^{y}\bar{X}_t^{w} t^{2 \rho} d\langle M\rangle_t-\frac12 t^{y}\bar{X}_t^{w} t^{2 \rho} d\langle M\rangle_t + t^{\rho} dM_t\\
\leq& -\frac14 t^{\rho} Y_t \, dt + \mathcal C_tt^{\kappa-(\alpha-1)\rho} \bar X_t^{\alpha} \, dt +\mathcal C_t (t^{\rho-\beta}+t^{\rho-1/\theta+\zeta}) \, dt \\
& +C \sum_{j,\theta_j^\prime> 0} t^{(y+\rho+\xi_j)/\theta^\prime_j+\rho(1-\gamma_j)} \bar{X}_t^{w/\theta^\prime_j+\gamma_j}dt+ C  t^{y+\rho+\xi} \bar{X}_t^{w}t^\rho Y_tdt-\frac12 t^{y}\bar{X}_t^{w} t^{2 \rho} d\langle M\rangle_t + t^{\rho} dM_t\\
\le& \big(-\frac14+ C  t^{y+\rho+\xi} \bar{X}_t^{w}  \big) t^{\rho} Y_tdt+\mathcal C_t \sum_{j\in \Theta} t^{-a_j} \bar X_t^{\alpha_j} \, dt + \mathcal C_t t^{-\mathfrak b} \, dt-\frac12 t^{y}\bar{X}_t^{w} t^{2 \rho} d\langle M\rangle_t+ t^{\rho} dM_t,
\eae
where in the above,  we set
\bae\label{eq:ab_def}
\Theta:=&\{0\}\cup\{j:\theta_j^\prime> 0\},\\
a_0 :=& \,(\alpha-1)\rho-\kappa,\\
\alpha_0:=&\,\alpha,\\
a_j :=&\, \psi_j\rho-(y+\xi_j)/\theta^\prime_j,\quad \theta_j^\prime> 0,\\
\alpha_j:=&\,w/\theta^\prime_j+\gamma_j,\qquad \qquad \theta_j^\prime> 0,\\  
\mathfrak b := &\,\beta \wedge(1/\theta-\zeta)- \rho,\\
\mathfrak c:=&\,y+\rho+\xi.
\eae
Then \eqref{middle Xt} becomes
\bae\label{final barX}
d \bar{X}_t \leq & \,\big(-\frac14+ C  t^{\mathfrak c} \bar{X}_t^{w}  \big) t^{\rho} Y_tdt+ \mathcal C_t \sum_{j\in \Theta} t^{-a_j} \bar X_t^{\alpha_j} \, dt + \mathcal C_t t^{-\mathfrak b} \, dt+d N_t,
\eae
where
$$
d N_t:=-\frac12 t^{y}\bar{X}_t^{w} t^{2 \rho} d\langle M\rangle_t+ t^{\rho} dM_t.
$$

For any $\delta,z>0$, define $T =T(z,\delta) :=\delta^{-z}$
and the stopping times
$$
\begin{aligned}
& \tau_T:=\inf \left\{t \geq T, \bar{X}_t \leq \delta\right\}, \\
& \bar{\tau}_T:=\inf \left\{t \geq \tau_T, \bar{X}_t \geq 2 \delta\right\}.\\
\end{aligned}
$$
From the second line of  \eqref{int-dif} , combined with the condition $p < 1 - \rho$ and Assumption {\rm{(A3)}} , it follows that $\tau_T < +\infty$ {\rm{a.s.}}
Moreover, it is easy to check that 
\bae\label{A subset}
\left\{\overline{\lim_{t \rightarrow \infty}  }\bar{X}_t>2\delta\right\} \subseteq\left\{\bar{\tau}_T<\infty\right\}.
\eae

For any $\lambda\in \mathbb R$, define the process $ L_\lambda(t) $ for all $ t \geq 0 $ by
\[
L_\lambda(t) :=
\begin{cases}
1, & \text{if } t < \tau_T, \\
\exp\left\{ -\dfrac{1}{2} \lambda^2 \displaystyle\int_{\tau_T}^t r^{2\rho} \, d\langle M \rangle_r + \lambda \displaystyle\int_{\tau_T}^t r^\rho \, dM_r \right\}, & \text{if } t \geq \tau_T.
\end{cases}
\]
Since $M_t$ is a local martingale and $\tau_T$ is almost surely finite, one has $L_\lambda(t)$ is an $\mathscr F_t$ nonnegative local   martingale, thus an $\mathscr F_t$ nonnegative  supermartingale. By Doob's supermartingale inequality, we conclude that
\bae\label{Doob sup}
\mathbb{P}\Big(\big\{\sup_{\tau_T\leq t\leq \bar\tau_T}L_{\lambda}(t) \geq \exp\{\delta \lambda/3\} \big\}\Big)\leq \exp\{-\delta \lambda/3\},\quad \forall \lambda\in\mathbb R.
\eae
Set $\lambda_*:=T^{y}(2 \delta)^{\frac{\alpha-\psi}{2}}.$ Note that $\langle M\rangle_t$ is increasing,  one has
\begin{equation}\label{delta prob}
\begin{aligned}
&\mathbb{P}\Big( \int_{\tau_T}^{\bar\tau_T} dN_t >\frac{1}{3}\delta \Big)\\
& =\mathbb{P}\Big(\int_{\tau_T}^{\bar\tau_T}\big(-\frac{1}{2} t^{y} \bar{X}^{w }_t t^{2 \rho} d \langle M\rangle_t+t^\rho d M_t\big)>\frac{1}{3} \delta\Big) \\
& \leq\mathbb{P}\Big(\int_{\tau_T}^{\bar\tau_T}\big(-\frac{1}{2} T^{y}(2 \delta)^{w} t^{2 \rho} d \langle M\rangle_t+t^\rho d M_t\big)>\frac{1}{3} \delta\Big) \\
& = \mathbb{P}\big(\big\{L_{\lambda_*}(\bar \tau_T) \geq \exp \{\delta \lambda_*/3\}\big\}\big) \leq \exp \left\{-\frac{1}{3} \delta \lambda_*\right\} \\
& = \exp \left\{-CT^{y} \delta^{w+1}\right\}=:\exp \left\{-C \delta^{-h}\right\},
\end{aligned}
\end{equation}
where the last inequality follows from \eqref{Doob sup} and $h:=yz-w-1.$

On the other hand, using \eqref{final barX} and \eqref{ab>1} on the event $\{\bar{\tau}_T<\infty\},$ one has
\begin{equation}\label{delta split}
\begin{aligned}
 \delta=&\int_{{\tau}_T}^{\bar{\tau}_T} d \bar X_t \\
 \leq& \int_{\tau_T}^{\bar{\tau}_T} \big(-\frac14+ C  t^{\mathfrak c} \bar{X}_t^{w}  \big) t^{\rho} Y_tdt
 +\int_{\tau_T}^{\bar{\tau}_T}\mathcal C_t \sum_{j\in \Theta} t^{-a_j} \bar X_t^{\alpha_j} d t+\int_{\tau_T}^{\bar{\tau}_T} \mathcal C_t t^{-\mathfrak b} d t+\int_{\tau_T}^{\bar{\tau}_T} d N_t \\
 \leq& \int_{\tau_T}^{\bar{\tau}_T} \big(-\frac14+ C  T^{\mathfrak c} \delta^{w}  \big) t^{\rho} Y_tdt+C_{\bar \tau_T} \sum_{j\in \Theta}\frac{(2 \delta)^{\alpha_j}}{a_j-1} \tau_T^{1-a_j}+\frac{ C_{\bar \tau_T}}{\mathfrak b-1} \tau_T^{1-\mathfrak b}+\int_{\tau_T}^{\bar{\tau}_T} d N_t\\
  \leq& \int_{\tau_T}^{\bar{\tau}_T} \big(-\frac14+ C   \delta^{-\mathfrak cz+w}  \big) t^{\rho} Y_tdt+C_{\bar \tau_T} \sum_{j\in \Theta}\frac{(2 \delta)^{\alpha_j}}{a_j-1} T^{1-a_j}+\frac{ C_{\bar \tau_T}}{\mathfrak b-1} T^{1-\mathfrak b}+\int_{\tau_T}^{\bar{\tau}_T} d N_t\\
\leq& \int_{\tau_T}^{\bar{\tau}_T} \big(-\frac14+ C   \delta^{-\mathfrak cz+w}  \big) t^{\rho} Y_tdt+\lim_{t\rightarrow+\infty}C_{t}\sum_{j\in \Theta} \delta^{\alpha_j+(a_j-1)z} +{\displaystyle\lim_{t\rightarrow+\infty}C_{t}} \delta^{(\mathfrak b-1)z} +\int_{\tau_T}^{\bar{\tau}_T} d N_t.
\end{aligned}
\end{equation}

Finally, for any $\epsilon>0$, using Assumption {\rm{(A3)}}, we can take $M>0$ such that
\bae\label{sm prob final}
\mathbb{P}\left\{\lim_{t\rightarrow+\infty}C_{t}\geq M\right\}\leq \frac12 \epsilon.
\eae
Rearranging \eqref{delta split}, using Lemma \ref{Coe lm} and the negativity of the drift term on the event $
\left\{\bar{\tau}_T<\infty\right\} \cap\left\{\displaystyle\lim_{t\rightarrow+\infty}C_{t}<M\right\},$ one has
$$
\begin{aligned}
\int_{\tau_T}^{\bar{\tau}_T} d N_t  \geq&\, \delta- \frac{2^\alpha M}{a-1} \delta^{\alpha+(a-1)z} -\frac{M}{\mathfrak b-1} \delta^{(\mathfrak b-1)z}+\int_{\tau_T}^{\bar{\tau}_T} \big(\frac14- C   \delta^{-\mathfrak cz+w}  \big) t^{\rho} Y_tdt \\
>& \,\delta-\frac{1}{3} \delta-\frac{1}{3} \delta +0 = \frac{1}{3} \delta,
\end{aligned}
$$
where the second inequality follows from \eqref{ab>1}, \eqref{odta1} and \eqref{odta2} by taking $\delta$ small enough.
Thus 
\bae\label{1/3delta sbset}
\left\{\bar{\tau}_T<\infty\right\} \cap\left\{\displaystyle\lim_{t\rightarrow+\infty}C_{t}<M\right\}\subseteq \left\{\int_{{\tau}_T}^{\bar{\tau}_T} d N_t>\frac{1}{3} \delta\right\}, \text{ for $\delta$ small enough.}
\eae
Using \eqref{A subset} and \eqref{1/3delta sbset}, we conclude that
$$\left\{\overline{\lim_{t \rightarrow \infty}  }\bar{X}_t>2\delta\right\} \subseteq \left\{\bar{\tau}_T<\infty\right\}\subseteq \left\{\int_{{\tau}_T}^{\bar{\tau}_T} d N_t>\frac{1}{3} \delta\right\} \bigcup \left\{\displaystyle\lim_{t\rightarrow+\infty}C_{t}<M\right\}^c.$$
Then
$$
\begin{aligned}
&\mathbb{P}\left(\overline{\lim_{t \rightarrow \infty}  }\bar{X}_t>2\delta \right) \leq \mathbb{P}\left\{\bar{\tau}_T<\infty\right\} \\
& \leq \mathbb{P}\left\{\int_{{\tau}_T}^{\bar{\tau}_T} d N_r>\frac{1}{3} \delta\right\}+\mathbb{P}\left\{\lim_{t\rightarrow+\infty}C_{t}<M\right\}^c \\
& \leq\exp \left\{-C \delta^{-h}\right\}+\frac12 \epsilon\leq \epsilon
\end{aligned}
$$
where the last inequality follows by choosing $\delta$ sufficiently small according to \eqref{delta prob} and \eqref{sm prob}, in conjunction with \eqref{sm prob final}. Specifically, for every $\epsilon > 0$, there exists $\delta_\epsilon > 0$ such that for all $\delta < \delta_\epsilon$, one has

\bae\label{delta sm}
\mathbb{P}\left(\overline{\lim_{t \rightarrow \infty}  }\bar{X}_t>2\delta \right)\leq \epsilon.
\eae
Notice that
\bae\label{osc class}
\left\{\overline{\lim_{t \rightarrow \infty}  }\bar{X}_t>0\right\} \subseteq\bigcap_{\delta>0}\left\{\overline{\lim_{t \rightarrow \infty}  }\bar{X}_t>2\delta\right\}.
\eae
Combining \eqref{delta sm} and \eqref{osc class}, one has $\mathbb P\left(\overline{\displaystyle\lim_{t \rightarrow \infty}  }\bar{X}_t>0 \right) =0,$ i.e., $$\lim_{t\rightarrow+\infty}\bar X_t=\lim_{t\rightarrow+\infty}t^\rho X_t=0, \,{\rm{a.s.}}.$$
The proof is complete.
\end{proof}

\subsection{Decay Estimates}
\begin{Lemma}\label{as L^p decay}
There exists a positive increasing $\mathcal F_t$-adapted stochastic process $\widetilde C_2(t,\omega)$   and constant $c=c(\alpha^*,\sigma,u_\pm,\mathfrak M)>0$, such that
    \bae\label{4-12}
	\left\|\phi(t)\right\|^{2} + \int_0^t\left\| \phi_{x} \right\|^2 d s \leq \left\|\phi_0\right\|^{2}+ C\ln(1+t)+\widetilde C_2(t,\omega), \quad {\rm{a.s.}}
    \eae
    with the limit  $\widetilde C_2(\omega):=\displaystyle\lim_{t\rightarrow\infty}\widetilde C_2(t,\omega)$ has exponential moment, i.e., $\mathbb E e^{c \tilde C_2(\omega)}<\infty.$

    Moreover, for any $p>4$, there exists a positive increasing $\mathcal F_t$-adapted stochastic process $\widetilde C_p(t,\omega)$ and constant $c_p=c_p(\alpha^*,\sigma,u_\pm,\mathfrak M,\|\phi_0\|,\|\phi_0\|_p)>0$, such that
    \baee\label{as lp final}
\| \phi(t) \|_ {p} ^p  \leq C (2+t)^{\frac{-p+3}{4}}\widetilde C_{p}(t,\omega),\quad {\rm{a.s.}}
    \eaee
    where the limit  $\widetilde C_p(\omega):=\displaystyle\lim_{t\rightarrow\infty}\widetilde C_p(t,\omega)$ satisfies $\mathbb E \exp\big\{c_p \widetilde C_{p}(\omega)^{\frac{2}{p}}\big\}<\infty.$
\end{Lemma}
\begin{proof}
Throughout this proof, we denote by $C$ (or $C_p$) a generic positive constant depending on the parameters listed in the lemma statement ($\alpha, \sigma, u_{\pm}, \mathfrak{M}$), which may vary from line to line.

By Lemma \ref{lmL^p}, one has
\begin{equation}
    \begin{aligned}\label{L2 expon bdd}
    &\left\|\phi(t)\right\|^{2} +\alpha^* \int_0^t\int_{\mathbb R}\phi ^{2} \bar{u}_{x} \dif x ~ \dif s + 2\mu\int_0^t\left\| \phi _{x} \right\| ^2 \dif s\\
    \leq& \left\|\phi _0\right\|^{2}+C\ln (1+t)+2\int_{0}^{t} \int_{\mathbb R} \sum_{i=1}^{\infty} \partial_x\sigma_{i}\left( \bar{u}+\phi\right) \phi \dif x~ \dif B_i(s).
    \end{aligned}
\end{equation}
For any $i\in \mathbb N_+$, notice that
    \bae\label{expo mat contr}
		&\partial_{x}\sigma_i\left( \bar{u}+\phi\right)\left|\phi\right|^{p-2} \phi\\	
            =&\partial_{x}\sigma_i\left( \bar{u}\right)\left|\phi\right|^{p-2} \phi+\partial_{x}\left[\sigma_i\left(\bar{u}+\phi\right)-\sigma_i\left( \bar{u}\right)\right]\left|\phi\right|^{p-2} \phi\\
            =&\sigma_i^\prime(\bar u)\bar u_x \left|\phi\right|^{p-2} \phi +(p-1)\bar u_x \int_{0}^{\phi}\left(\sigma_i^{\prime}\left( \bar{u}+y\right)-\sigma_i^{\prime}\left( \bar{u}\right)\right)|y|^{p-2} d y \\
			&+\partial_{x}\left[\left|\phi\right|^{p-2} \phi\big(\sigma_i \left(\bar {u}+\phi\right)-\sigma_i \left( \bar{u}\right)\big)\right]+\partial_{x}\left[-(p-1) \int_{0}^{\phi}\left(\sigma_i \left( \bar{u}+y\right)-\sigma_i \left( \bar{u}\right)\right)|y|^{p-2} d y\right].
	\eae
Recall \eqref{Mp def} for the martingale term in \eqref{L2 expon bdd}.
Setting $p = 2$ in \eqref{expo mat contr} and substituting the result into the inequality above, we observe that the last two terms on the RHS of \eqref{expo mat contr} cancel out upon integration over $\mathbb{R}$. Thus, we conclude that

\bae\label{qv}
\frac{d}{dt}[M_2]_t\leq& \sum_{i=1}^{\infty}\Big[ \int_\R\big(\sigma_i^\prime(\bar u)\bar u_x  \phi +\bar u_x \int_{0}^{\phi}\left(\sigma_i^{\prime}\left( \bar{u}+y\right)-\sigma_i^{\prime}\left( \bar{u}\right)\right) d y\big)dx \Big]^2 \\
\leq& \sum_{i=1}^{\infty}\Big[ \int_\R\big(\|\sigma^\prime_i\|_{L^\infty_\mathfrak M} |\phi|\bar u_x   +2\|\sigma^\prime_i\|_{L^\infty_\mathfrak M}|\phi|\bar u_x \big)dx \Big]^2 \\
\leq&\, C \|\sigma^\prime\|_{L^\infty_\mathfrak M} ^2\Big(\int_{\mathbb R}|\phi|  \bar{u}_{x} \dif x \Big)^2 \leq C\|\sigma^\prime\|_{L^\infty_{\mathfrak M}} ^2 \int_0^t\int_{\mathbb R}\phi ^{2} \bar{u}_{x} \dif x .\\
\eae
In the last inequality, we used Jensen’s inequality with respect to the measure $\bar{u}_x dx$, noting that it has finite mass $u_+ - u_-$. We choose $c = c(\alpha^*, \sigma, u_\pm) > 0$ sufficiently small such that
$$c[M_2]_t-\alpha^* \int_0^t\int_{\mathbb R}\phi ^{2} \bar{u}_{x} \dif x= (c C-\alpha^*)\|\sigma^\prime\|_{L^\infty_{\mathfrak M}} ^2 \int_0^t\int_{\mathbb R}\phi ^{2} \bar{u}_{x} \dif x ~ \dif s\leq 0.$$
Substituting \eqref{qv} and the above inequality into \eqref{L2 expon bdd}, we obtain
    \begin{equation}\label{4-16}
    \begin{aligned}
    \left\|\phi(t)\right\|^{2}  + 2\mu\int_0^t\left\| \phi _{x} \right\| ^2 \dif s
    \leq \left\|\phi _0\right\|^{2}+C\ln (1+t)-c[M_2]_t+M_2(t).
    \end{aligned}
\end{equation}
It follows from \eqref{qv} that $M_2(t)$ is an $L^2$-martingale. Applying the standard exponential martingale inequality, we obtain

\bae\label{expo bd of K}\mathbb P\big(\sup_ {t\geq 0}(-c[M_2]_t+M_2(t))>R\big)\leq e^{-2c R},\quad \forall R>0.\eae
Combining \eqref{4-16} and \eqref{expo bd of K}, we obtain \eqref{4-12} by defining $\tilde{C}_2(t, \omega) := \sup_{0\le s\le t} (- c[M_2]_s + M_2(s))$.


For $p > 4$, setting $\gamma = (p - 3)/4$ and following the arguments in  \eqref{5.22}, we have

\begin{equation}\label{as Lp}
		(2+t)^\gamma \| \phi(t) \|_ {p} ^p  \leq  C\| \phi_0 \|_ {p} ^p +C(2+t)^{\gamma-\frac p2} + C \int_0^t (2+s)^{\gamma-\frac{p+2}{4}} \| \phi \|^p \dif s +  N_{p}(t).
	\end{equation}
Similar to the arguments in \eqref{expo mat contr}, we define
    \bae\label{N_p}
&N_{p}(t):=p\int_0^t (2+s)^\gamma \int_ \mathbb R \sum_{i=1}^{\infty} \partial_x \sigma_i (\phi +\bar u) |\phi|^{p-2}\phi  \dif x~\dif B_i(s)\\
=&p\sum_{i=1}^{\infty}\int_0^t (2+s)^\gamma \int_ \mathbb R \Big[\sigma_i^\prime(\bar u)\bar u_x \left|\phi\right|^{p-2} \phi +(p-1)\bar u_x \int_{0}^{\phi}\left(\sigma_i^{\prime}\left( \bar{u}+y\right)-\sigma_i^{\prime}\left( \bar{u}\right)\right)|y|^{p-2} d y\Big] \dif x~\dif B_i(s).
    \eae
     Specifically, this choice of $\gamma$ ensures that the exponent in the third factor in the RHS of \eqref{as Lp} satisfies $-5/4 < -1$. Combining this with \eqref{4-12}, we have
     \begin{eqnarray}\label{def int upbdd}
         \begin{aligned}
             \int_0^t (2+s)^{\gamma-\frac{p+2}{4}} \| \phi \|^p \dif s
\leq \int_0^t (2+s)^{-5/4} [C\ln(1+s)+\tilde C_2(s)]^p \dif s
\leq  C_p \big(1+\tilde C_2(t)^p\big).
         \end{aligned}
     \end{eqnarray}
Using the similar argument in \eqref{qv} on \eqref{N_p}, one has
\baee
[N_{p}]_t=&\,p^2\sum_{i=1}^{\infty}\int_0^t (2+s)^{2\gamma} \Big(\int_ \mathbb R  \partial_x \sigma_i (\phi +\bar u) |\phi|^{p-2}\phi  \dif x\Big)^2\dif s\\
\leq&\,  Cp^2\int_0^t (2+s)^{2\gamma} \Big(\int_ \mathbb R |\phi|^{p-1}\bar u_x\dif x\Big)^2\dif s\\
\leq&\,  Cp^2\int_0^t (2+s)^{2\gamma} \int_ \mathbb R |\phi|^{2p-2}\bar u_x\dif x~\dif s.\\
\eaee
Let $q > 1$ and $q' = \frac{q}{q-1}$ be its conjugate exponent. Set $\beta = \frac{(4\gamma+p+1)q-3}{4q}$, a direct computation  for $N_p(t)$ yields

\baee
 \mathbb E[N_{p}]_t^q\leq&\, C^q p^{2q} \mathbb E\Big(\int_0^t (2+s)^{2\gamma-\beta}(2+s)^{\beta}  \int_ \mathbb R |\phi|^{2p-2}\bar u_x\dif x~\dif s\Big)^q\\
 \leq&\, C^q p^{2q}\Big(\int_0^t (2+s)^{2\gamma q^\prime-\beta q^\prime} \dif s\Big)^{\frac{q}{q^\prime}}\mathbb E\int_0^t (2+s)^{\beta q} \Big(\int_ \mathbb R |\phi|^{2p-2}\bar u_x\dif x\Big)^q\dif s\\
 \leq&\, C^q p^{2q}\int_0^t (2+s)^{\beta q} \mathbb E\int_ \mathbb R |\phi|^{2q(p-1)}\bar u_x\dif x~\dif s\\
 \leq&\, C^q p^{2q}\int_0^t (2+s)^{\beta q-1} \mathbb E\|\phi\|^{2q(p-1)}_{2q(p-1)}\dif s\leq C^q p^{2q},
\eaee
where the last inequality follows from \eqref{5.20}. Thus, $N_p(t)$ is an $L^2$-uniformly integrable martingale, and we define
\bae\label{4-21}
\bar C_{p}(t,\omega):=\sup_{0\leq s\leq t}|N_{p}(s)|,\quad \bar C_{p}(\omega):=\lim_{t\rightarrow\infty}\bar C_{p}(t,\omega).\eae
By Fatou's lemma and the Burkholder-Davis-Gundy inequality, for any $q>1$, we have
\bae\label{4-22}
\mathbb E \bar C_{p}(\omega)^{2q}\leq & \lim_{t\rightarrow\infty}\mathbb E \sup_{0\leq s\leq t}|N_{p}(s)|^{2q}\leq (C q)^q\lim_{t\rightarrow\infty}\mathbb E [N_{p}]_t^q\leq  (C q p^2 )^q.
\eae
The growth of moments implies that $\bar{C}_p(\omega)$ exhibits sub-Gaussian behavior, ensuring that $\mathbb{E} \exp(c \bar{C}_p^2) < \infty$ for some $c > 0$


Substituting \eqref{4-12}, \eqref{def int upbdd}, \eqref{4-21} and \eqref{4-22} into \eqref{as Lp}, we obtain
\baee
 \| \phi(t) \|_ {p} ^p  \leq&\,  C(2+t)^{-\frac p2} + C (\phi_0)(2+t)^{\frac{-p+3}{4}}\big(\widetilde C_2(t,\omega)^\frac{p}{2} +  \bar C_{p}(t,\omega)\big)\\
 \leq&\,C (\phi_0)(2+t)^{\frac{-p+3}{4}}\widetilde C_{p}(t,\omega),
\eaee
where we define $\widetilde C_{p}(t,\omega)=\tilde C_2(t,\omega)^\frac{p}{2} +  \bar C_{p}(t,\omega).$ Note that $\widetilde C_{p} $ satisfies the required integrability estimates due to \eqref{expo bd of K} and \eqref{4-22}. This completes the proof.
\end{proof}
\begin{Lemma}\label{as deriv L^2 decay}
If $\phi\in L^2_{\rm{loc}}\big((0,\infty),H^2(\mathbb R)\big)$ almost surely. Then for all $\epsilon>0$, we have
\baee
\mathbb P\Big(\|\phi_x\|^2(t)=O\big({(1+t)^{\frac{3}{14}+\epsilon}}\big)\Big)=1.
\eaee
\end{Lemma}
\begin{proof}
We denote by $C$ a generic positive constant depending on the initial data $u_0$, the noise coefficients $\sigma$, and the viscosity parameters $b, \mu$. 
We use $C_q$ to denote constants depending additionally on the parameter $q$ from the Gagliardo-Nirenberg estimates. Substituting Proposition \ref{lm2.1} and Lemma \ref{as L^p decay} into \eqref{quasi deriv} yields 
\bae\label{pre area}
d\left\|\phi_x\right\|^2 \leq & -\frac{\mu}{2}\left\|\phi_{xx}\right\|^2 d t+ C\left\|\phi_x\right\|^4_4 d t+\widetilde C_q(t)^{\frac{5}{q-1}} (1+t)^{-\frac{5 }{4}\frac{q-3}{q-1}} d t \\
& +C(1+t)^{-2} d t+2 d M_t.
\eae
Notice that
\baee
&\frac{d}{dt}\langle M\rangle_t=\sum_{i=1}^{\infty}\Big(\int_{\mathbb{R}} \phi_x \partial_{x x}^2 \sigma_i(u) d x  \Big)^2 \\
=&\sum_{i=1}^{\infty}\Big[\int_{\mathbb{R}} \phi_{xx}  \sigma_i^\prime(u)(\bar u_x+ \phi_x) d x  \Big]^2 \\
=&\sum_{i=1}^{\infty}\Big(\int_{\mathbb{R}} \phi_{xx}  \sigma_i^\prime(u)\bar u_x d x +\int_{\mathbb{R}} \phi_x\phi_{xx}  \sigma_i^\prime(u)  d x  \Big)^2 \\
\leq&\,C\sum_{i=1}^{\infty}\Big(\int_{\mathbb{R}} \phi_{xx}  \sigma_i^\prime(u)\bar u_x d x  \Big)^2 +C\sum_{i=1}^{\infty}\Big(\int_{\mathbb{R}} \phi_x\phi_{xx}  \sigma_i^\prime(u)  d x  \Big)^2 \\
\leq&\,C\sum_{i=1}^{\infty}\int_{\mathbb{R}} |\phi_{xx}  \sigma_i^\prime(u)|^2\bar u_x d x +C\sum_{i=1}^{\infty}\Big(\frac12\int_{\mathbb{R}} \phi_x^2  \sigma_i^{\prime\prime}(u)(\phi_x+\bar u_x)  d x  \Big)^2 \\
\leq&\,C\|\sigma^\prime\|^2_{l^2(L^\infty_\mathfrak M)}\int_{\mathbb{R}} |\phi_{xx}  |^2\bar u_x d x +C\|\sigma^{\prime\prime}\|^2_{l^2(L^\infty_\mathfrak M)}\Big(\int_{\mathbb{R}} |\phi_x|^3  d x  \Big)^2 +C\|\sigma^{\prime\prime}\|^2_{l^2(L^\infty_\mathfrak M)}\Big(\int_{\mathbb{R}} \phi_x^2\bar u_x  d x  \Big)^2,
\eaee
here, the first inequality follows from Jensen's inequality with respect to the  measure $\bar u_x dx$ (which has finite mass $u_+-u_-$), and the last inequality follows from (H3). Applying Hölder’s inequality, we obtain
\bae\label{mid SVCL quadra}
&\frac{d}{dt}\langle M\rangle_t\\
\leq&\,C\Big( \|\bar u_x \|^2_{\infty}\|\phi_{xx}\|^2   +\| \phi_x\|_3^6 +\|\phi_x\|^4_{4}\|\bar u_x\| ^2   \Big) \\
\leq&\,C\Big( \|\bar u_x \|^2_{\infty}\|\phi_{xx}\|^2   +\| \phi_x\|^5\|\phi_{xx}\|  +\|\phi_x\|^{3}\|\phi_{xx}\|\cdot\|\bar u_x\|^2   \Big) \\
\leq&\,C\Big( (1+t)^{-2}\|\phi_{xx}\|^2   +\| \phi_x\|^{5-y}\|\phi_x\|^y\|\phi_{xx}\| +(1+t)^{-2/p}\|\phi_x\|^{2+2/p}\|\phi_{xx}\|^{2-2/p}   \Big),\\
\leq&\,C\Big( (1+t)^{-2}\|\phi_{xx}\|^2  +C_q\| \phi_x\|^{5-y}  \|\phi(t)\|_q^{\frac{2qy}{3q+2}}\|\phi_{xx}(t)\|^{\frac{(q+2)y}{3q+2}+1} +(1+t)^{-1}\|\phi_x\|^{3}\|\phi_{xx}\|   \Big),\\
\eae
 the last three lines follow from the Gagliardo–Nirenberg interpolation inequality
 and Proposition \ref{lm2.1} with $y\in[3/2,5)$.


Substituting Lemma \ref{as L^p decay} into \eqref{mid SVCL quadra}, for $y\in[3/2,5)$ and $q>3$, we obtain 
\bae\label{SVCL quadra}
&\frac{d}{dt}\langle M\rangle_t\leq C\Big( (1+t)^{-2}\|\phi_{xx}\|^2  +\tilde C_q(t)^{\frac{2y}{3q+2}}(1+t)^{\frac{(-q+3)y}{6q+4}}\| \phi_x\|^{5-y}  \|\phi_{xx}(t)\|^{\frac{(q+2)y}{3q+2}+1} +(1+t)^{-1}\|\phi_x\|^{3}\|\phi_{xx}\|  \Big),\\
\eae
In view of \eqref{SVCL quadra} and  that $\phi\in L^2_{\rm{loc}}\big((0,\infty),H^2(\mathbb R)\big)$ almost surely, we conclude that $M(t)$ is a continuous local martingale.  By setting
\baee
X_t=\|\phi_x(t)\|^2,Y_t=\|\phi_{xx}(t)\|^2, \quad t\geq0,
\eaee
\eqref{SVCL quadra} becomes
\bae\label{quadra svcl}
\frac{d}{dt}\langle M\rangle_t\leq&\,C\sum_{j=1}^3 (1+t)^{\xi_j }X_t^{\gamma_j\theta^\prime_j} Y_t^{1-\theta^\prime_j},
\eae
with coefficients
\bae\label{mid para1}
\gamma_1=& 0,\,\xi_1=-2 ,\,\theta^\prime_1=0 ,\quad \gamma_3=3,\,\xi_3=-1,\,\theta^\prime_3=\frac{1}{2},\\
\gamma_2=&\frac{(5-y)(3q+2)}{(q+2)y+3q+2} ,\,\xi_2=\frac{(-q+3)y}{6q+4} ,\ \ \,\, \theta^\prime_2=\frac{(q+2)y+3q+2}{6q+4}.\\
\eae
Here $y\in[3/2,5)$ and $q>3$. Let $\theta_j^\prime> 0$ and recall that $\psi_j=\gamma_j-1/\theta^\prime_j-1$, it is easy to check that
\bae\label{psi<0}
\psi_2= 2\frac{3q+2-2(q+1)y}{3q+2+(q+2)y}<0,\quad \psi_3=0.\\
\eae
We now apply the Stochastic Area Inequality (Lemma \ref{stochastic area}) to \eqref{pre area} and \eqref{SVCL quadra} by checking Assumptions {\rm{(A1) }-\rm{(A3)}}.

\,
\textbf{Assumption {\rm{(A1) }}:} 
It is directly from \eqref{psi<0} and \eqref{mid para1}.

\,
\textbf{Assumption {\rm{(A2) }}:} To maximize $\bar{\rho}$ in \eqref{rho^*}, we take $y \nearrow 5$ and let $q \nearrow  +\infty$, then

\bae\label{nega ceiling}
\displaystyle\lim_{y \nearrow 5}\displaystyle\lim_{q \nearrow +\infty}\frac{\theta^\prime_2+\xi_2}{\theta^\prime_2 \psi_2}=\displaystyle\lim_{y \nearrow 5}\displaystyle\lim_{q \nearrow +\infty} \frac{2+3q+5y}{4+6q-4y-4qy}=\frac{-3}{14}.\eae
For $q>3$, applying the Gagliardo–Nirenberg interpolation inequality and Lemma \ref{as L^p decay}, we obtain
\bae\label{theta value}
X_t= \|\phi_x(t)\|^2\leq C_q \|\phi(t)\|_q^{\frac{4q}{3q+2}}\|\phi_{xx}(t)\|^{\frac{2q+4}{3q+2}}
\leq   \widetilde C_q(t)^\frac{4}{3q+2} (1+t)^{\frac{-q+3}{q}\frac{q}{3q+2}}Y_t^{1-\frac{q}{3q+2}}.
\eae
Taking $0<z<3,$ by a similar argument, we obtain 
\bae\label{alpha value}
 &\left\|\phi_x\right\|_4^4 \leq C\left\|\phi_{xx}\right\|\left\|\phi_x\right\|^3\\
=&\,C\left\|\phi_{xx}\right\|\left\|\phi_x\right\|^{3-z}\left\|\phi_x\right\|^{z}\\
\leq&\, C_q\left\|\phi_{xx}\right\|\left\|\phi_x\right\|^{3-z} \|\phi\|_q^{\frac{2qz}{3q+2}}\|\phi_{xx}\|^{\frac{(q+2)z}{3q+2}} \\
\leq&\, C_q \mathcal C_t  \left\|\phi_{xx}\right\|^{1+\frac{(q+2)z}{3q+2}}(1+t)^{\frac{(-q+3)z}{6q+4}}\left\|\phi_x\right\|^{3-z}  \\
\leq&\, \frac14 \mu \left\|\phi_{xx}\right\|^2+CC_q\|\phi\|_q^{\frac{4q}{(q+2)(3-z)-4}}\left\|\phi_x\right\|^{\frac{2(3q+2)(3-z)}{(q+2)(3-z)-4}}\\
\leq&\, \frac14 \mu \left\|\phi_{xx}\right\|^2+C\widetilde C_q(t)^{\frac{4}{(q+2)(3-z)-4}} (1+t)^{\frac{-q+3}{(q+2)(3-z)-4}}\left\|\phi_x\right\|^{\frac{2(3q+2)(3-z)}{(q+2)(3-z)-4}}.
\eae
Substituting \eqref{theta value} and \eqref{alpha value} into \eqref{pre area}, we obtain
\bae\label{nonl}
d\left\|\phi_x\right\|^2 \leq  -\frac{\mu}{4}Y_t d t+\mathcal C_t (1+t)^\kappa X_t^\alpha d t+\mathcal C_t (1+t)^{-\beta} d t +2 d M_t.\\
\eae
where in the above $$\alpha=\alpha(z, q) := \frac{(3q+2)(3-z)}{(q+2)(3-z)-4}, \quad \quad \kappa=\kappa(z, q) := \frac{-q+3}{(q+2)(3-z)-4}.$$ Our goal is to find the maximal $\bar \rho$ such that Assumption {\rm{(A2) }} is valid. To achieve this, we need to keep $\alpha$ close to $3$. We achieve this by examining the behavior of  $\alpha$ and $\kappa$ near the singularity of the denominator. Let $D=D(z, q) := (q+2)(3-z) - 4$, the root  is given by
$z^* = 3 - \frac{4}{q+2}.$ We choose $q$ sufficiently large (specifically $q > 3$) such that $0 < z^* < 3$. As  $z$ approaches $z^*$ from below ($z \nearrow z^*$),   $D(z,q)$ tends to $0^+$. Consequently
$$
\lim_{z \nearrow z^*} \kappa(z, q) = -\infty,\quad \lim_{z \nearrow z^*} \alpha(z, q) = +\infty.
$$
However, comparing the asymptotic scaling of $\kappa$ and $\alpha$ with respect to    $D$, we observe that $\kappa \sim -q/D$ and $\alpha \sim 3/D$ as $z \nearrow z^*$ (ignoring lower order terms). For the stability condition $\kappa < (\alpha-1)\bar{\rho} - 1$ to hold, it suffices to have
$$
-\frac{q}{D} \ll \frac{3\bar{\rho}}{D}.
$$
This is guaranteed by choosing $q$ sufficiently large. Therefore, by first fixing $q$ large enough and then choosing $z$ sufficiently close to $z^*$, we ensure that $\kappa$ is sufficiently negative to satisfy Assumption {{\rm{(A2) }}}. 

\,

\textbf{Assumption {{\rm{(A3) }}}:} For any positive $\epsilon$ small enough, using similar arguments, we can choose  $\zeta,\theta,\beta$ as follows
\bae\label{para2}
-1-\epsilon>\zeta>-1,\frac13-\epsilon<\theta<\frac13, \frac54-\epsilon<\beta<\frac54,
\eae
and 
\baee
\mathcal C_t= C\big(1+ \widetilde C_{q_0}(t)^{-\kappa/q_0}\big),\quad q_0 \text{ large enough.}
\eaee
 Assumption  {\rm{(A3) }} follows directly from Lemma \ref{as L^p decay}. 
 
Finally, by \eqref{quadra svcl}, \eqref{mid para1}, \eqref{nega ceiling}, \eqref{nonl} and \eqref{para2}, we conclude that $\bar \rho=\frac{-3}{14},$ which implies $X_t=O(t^{\frac{3}{14}+\epsilon}),$ for any $\epsilon>0.$ The proof is complete. 
\end{proof}
\begin{remark}
    We highlight a distinct mechanism here compared to the global existence arguments in Lemma 2.3. While Lemma 2.3 imposes a smallness condition on $\|\phi_0\|_\infty$ to ensure that the viscous dissipation $\mu\|\phi_{xx}\|^2$ dominates the nonlinear noise injection at $t = 0$, the present asymptotic analysis relies on the decay established in Lemma 2.6. Since $\|\phi(t)\|_q$ decays algebraically in time almost surely, the coefficients of the destabilizing terms in (2.70) vanish as $t \to \infty$. Consequently, for sufficiently large time, the viscous dissipation eventually dominates the stochastic energy injection regardless of the initial data size, thereby justifying the stability analysis under the non-blow-up assumption.
\end{remark}
\begin{proof}[\textbf{Proof of Theorem \ref{as L^infty decay}}]
      Let $\phi(t,x) = u(t,x) - \bar u(t,x)$ denote the perturbation of the smoothed rarefaction wave. By the Gagliardo--Nirenberg interpolation inequality, for any $p > 2$, we have
\begin{equation} \label{eq:GN_infty}
    \|\phi(t)\|_{\infty} \le C_p \|\phi(t)\|_{p}^{\frac{p}{p+2}} \| \phi_x(t)\|^{\frac{2}{p+2}}.
\end{equation}
We estimate the terms on the right-hand side using the almost sure decay results obtained from the a priori estimates

\begin{enumerate}
    \item \textbf{$L^p$ decay:} From Lemma \ref{as L^p decay}, for sufficiently large $p$, there exists a random variable $\tilde{C}_p(\omega)$ such that almost surely
    \begin{equation*}
        \|\phi(t)\|_{p}^p \le \tilde{C}_p (1+t)^{-\frac{p-3}{4} + \delta_1}.
    \end{equation*}
    Taking the $p$-th root yields
    \begin{equation} \label{eq:Lp_decay}
        \|\phi(t)\|_{p} \le \tilde{C}_p^{\frac{1}{p}} (1+t)^{-\frac{1}{4} + \frac{3}{4p} + \frac{\delta_1}{p}}.
    \end{equation}

    \item \textbf{Derivative estimates:} From Lemma \ref{as deriv L^2 decay}, under the non-blow-up assumption $u - \bar u \in L_{\rm{loc}}^2 H^2$, the derivative decays almost surely as
    \begin{equation} \label{eq:deriv_decay}
        \|\phi_x(t)\| \le C (1+t)^{\frac{3}{28} + \frac{\epsilon_0}{2}}.
    \end{equation}
\end{enumerate}

Substituting \eqref{eq:Lp_decay} and \eqref{eq:deriv_decay} into \eqref{eq:GN_infty}, we have
\baee
\|\phi(t)\|_{\infty} \le C_p \tilde C_p^{1/p}(1+t)^{\frac{p}{p+2} \big(-\frac{1}{4} + \frac{3}{4p}\big) + \frac{2}{p+2} \big(-\frac{1}{8}\big) + O\big(\frac{1}{p}\big)}.
\eaee
As $p \to \infty$, the exponent of $(1+t)$ goes to $ -\frac{1}{4}$. Therefore, for any given $\epsilon > 0$, we can choose $p$ sufficiently large such that the decay rate is bounded by $-\frac{1}{4} + \epsilon$.

Finally, by the triangle inequality,
\[
    \|u(t) - u^r(t)\|_{\infty} \le \|\phi(t)\|_{\infty} + \|\bar{u}(t) - u^r(t)\|_{\infty}.
\]
Since the smoothed profile $\bar{u}$ converges to the exact rarefaction wave $u^r$ at the faster algebraic rate of $(1+t)^{-1/2}$, the stochastic decay dominates. Thus
\[
    \|u(t) - u^r(t)\|_{\infty} = O(t^{-\frac{1}{4} + \epsilon}), \quad \mathbb{P}\text{-a.s.}
\]
This completes the proof.
\end{proof}

\begin{Lemma}[Derivative decay]\label{td}
	
	Under the assumptions of Lemma {\rm{\ref{lm2.4}}}, we obtain the derivative estimate
	\bae\label{5.13}
	\E \|\phi_x\|^2\leq C (1+t)^{-1+\epsilon}, \quad \forall \epsilon>0.
	\eae
\end{Lemma}
\begin{proof}
	We formally derive the decay estimate here, noting that the rigorous justification on the Galerkin approximations is provided in Lemma \ref{truncated-derivative-estimate}. Applying It\^o's formula on $\|\phi_x\|^2,$ we obtain
	\baee\label{deriv}
	& \E \ \left\| \phi_{x}(T)  \right\|^2+\E \int_0^T\|\phi_{  xx}\|^2\dif t   \leq \|\phi_x(0)\|^2+C+  C \E\int_0^T (1+t)^{-2}\left\| \phi  \right\|_{\ H^1(\R)} ^2\dif t .
	\eaee
	In particular, utilizing \eqref{real-L^2}, we have
	\baee
	\frac{d}{\dif t } \E \|\phi_x\|^2\leq C (1+t)^{-2} \ln(1+t)+ C (1+t)^{-2}\E \|\phi_x\|^2. 
	\eaee
	Combining the above inequality with \eqref{real-L^2}, \eqref{H^2}, and the area inequality \cite[Theorem 1.2]{dong2025large}, we conclude that
	\baee
	\E \|\phi_x\|^2\leq C (1+t)^{-1+\epsilon}, \quad \forall \epsilon>0.
	\eaee
 This proves \eqref{5.13}. 
\end{proof}
From  Proposition \ref{t1} and Lemma \ref{td}, we obtain the following $L^\infty$ decay estimate.
\begin{Lemma}[$L^\infty$ decay]\label{t2}
Assume that the conditions of Lemma {\rm{\ref{lm2.4}}} hold. Let $\phi\in X_2(T)$ be the solution of {\rm{(\ref{perturbation})}}. Then  
\begin{eqnarray*}
\mathbb E \| \phi \|_{\infty }  \leq  C(2+t)^{-\frac{1}{4}+\epsilon}, \quad \forall \epsilon>0.
\end{eqnarray*}
\end{Lemma}
\begin{proof}
By the Gagliardo–Nirenberg interpolation inequality, we have that
\begin{eqnarray*}
\| \phi \|_{\infty } \leq C_p  \| \phi \| _{p } ^{\frac{p}{p+2}} \| \phi_x  \| ^\frac{2}{p+2},
\end{eqnarray*}
which yields
\begin{eqnarray*}
\begin{array}{lll}
\displaystyle \mathbb E \| \phi \|_{\infty } &\leq& C_p \mathbb E \Big(\| \phi \|_ {p}^{\frac{p}{p+2}} \| \phi_x \|^\frac{2}{p+2}\Big) \leq C_p \left(\mathbb E \| \phi \|_ {p}^p\right) ^{\frac{1}{p+2}}\Big(\mathbb E \| \phi_x \|^2\Big)^{\frac{1}{p+2}}\\
&\leq&   \displaystyle C(2+t)^{-\frac{1}{4}+\epsilon}
\end{array}
\end{eqnarray*}
provided that $p$ is chosen sufficiently large.  

\end{proof}

\section{Uniqueness via Stochastic $\rm{L_{loc}^{1}}$ Contraction}

The proof of uniqueness employs Kružkov's method of doubling variables, adapted to the stochastic setting.  By choosing a suitable stochastic test function, we  establish the  $L^1_{\rm{loc}}(\R)$ contraction \textit{without requiring $L^1(\R)$ initial data.} 
\subsection{The Locally Stochastic Kružkov Estimate}\label{subsec kruz}
\begin{Lemma}[$L^1_{\rm{loc}}$ contraction]\label{lockruz}
	 For $T>0$, let $u,v$ be two solutions of equation \eqref{conservative} in $X_1(T)$ with initial data $u_0 - \bar{u}_0, v_0 - \bar{u}_0 \in H^1(\mathbb{R})$ and let 
$S:={\|f^\prime(u_0)\|_{\infty}\vee\|f^\prime(v_0)\|_{\infty}}.$ 
Then for all $r>0$ and almost all $0<t<T$ we have
   \bae\label{locL1}
\mathbb E\int_{|x|<r}|u(x, t)-v(x, t)| d x   &\leq \mathbb E \int_{|x|<r+S t}\left|u_0(x)-v_0(x)\right| d x.
\eae
 
\end{Lemma}
The proof of Lemma \ref{lockruz} follows the strategy in \cite{feng2008stochastic}, and \cite{biswas2014stochastic},  which relies on the stochastic Kružkov estimate. We present some necessary technical preparations before proceeding to the proof.

Let $\rho$ and $\varrho$ be the standard mollifiers on $\mathbb{R}$ such that $\operatorname{supp}\,\rho \subset[-1,0]$ and $\operatorname{supp}\,\varrho=[-1,1]$. For $\delta>0$ and $\delta_0>0$, let 

$$\rho_{\delta_0}(r)=\frac{1}{\delta_0} \rho\left(\frac{r}{\delta_0}\right)\quad {\rm{and}} \quad \varrho_\delta(x)=\frac{1}{\delta} \varrho\left(\frac{x}{\delta}\right).$$

For a non-negative test function $\psi \in C_c^{1,2}\left([0, \infty) \times \R\right)$ and $\delta, \delta_0>0$, define
\bae\label{mol}
\eta_{\delta, \delta_0}(t, x ; s, y)=\rho_{\delta_0}(t-s) \varrho_\delta(x-y) \psi(s, y) .
\eae
Note that $\rho_{\delta_0}(t-s) \neq 0$ only if $s-\delta_0 < t < s$, and therefore $\eta_{\delta, \delta_0}(t, x ; s, y)=0$ outside $s-\delta_0 < t < s$.
Let $\beta: \mathbb{R} \rightarrow \mathbb{R}$ be a non negative smooth function satisfying
$$
\beta(0)=0, \quad \beta(-r)=\beta(r), \quad \beta^{\prime}(-r)=-\beta^{\prime}(r), \quad \beta^{\prime \prime} \geq 0,
$$
and
$$
\beta^{\prime}(r)=\left\{\begin{array}{l}
-1, \quad \text { when } r \leq-1; \\
\in[-1,1], \quad \text { when }|r|<1; \\
+1, \quad \text { when } r \geq 1.
\end{array}\right.
$$
For any $\varepsilon>0$, define $\beta_\varepsilon: \mathbb{R} \rightarrow \mathbb{R}$ by
$$
\beta_\varepsilon(r)=\varepsilon \beta\left(\frac{r}{\varepsilon}\right).
$$
Define $ \Pi_T=[0,T]\times\R$ and introduce the notation
$$
[g](r)\coloneqq\int_0^r g( s) d s,~[g, \varepsilon](r, a)\coloneqq\left[g \beta_\varepsilon^{\prime}(\cdot-a)\right](r),~f_\varepsilon(r,a):=\int_a^r \beta_\varepsilon^\prime(s-a) f^\prime(s)ds.
$$

We now outline the overall strategy for the subsequent $\rm{L_{loc}^{1}}$ estimates. Let $u,v$ be two solutions of \eqref{conservative} such that $u-u_0, v-v_0\in X_1(T)$. Using the doubling of variables method, we can apply the stochastic Kružkov estimate to approximate $\mathbb E\|u(s)-v(s)\|_{L^1_{\rm{loc}}}$ by  
\bae\label{molimit}
\int_{\Pi_T} \int_{\Pi_T} \beta_\varepsilon\left(u(t, x)-v(s, y)\right) \eta_{\delta,\delta_0}(t,x,s,y)\, d x d y d tds.
\eae
Since $u,v$ are solutions of \eqref{conservative} and $u-v\in X_1(T)$, applying It\^o's formula to \eqref{molimit} combined with the It\^o product rule and integration over $\R$ yields
\baee
& -\int_{\mathbb{R}} \beta_{\varepsilon}(u(t, x)-v(s, y)) \partial_t \eta_{\delta , \delta_0} d x \\
=& \int_{\mathbb{R}} \beta_{\varepsilon}^{\prime}(u(t, x)-v(s, y)) \partial_x(b(u)u_x) \eta_{\delta, \delta_0} d x d t \\
&+\sum_k\int_{\mathbb{R}} \beta_{\varepsilon}^{\prime}(u(t, x)-v(s, y)) \partial_x\left(\frac{1}{2} \sigma_k^{\prime 2}(u) u_x+f(u)\right) \eta_{\delta , \delta_0} d x d t \\
& +\frac{1}{2}\sum_k \int_{\mathbb{R}} \beta_{\varepsilon}^{\prime \prime}(u(t, x)-v(s, y)) \sigma_k^{\prime 2}(u) u_x^2 \eta_{\delta, \delta_0} d x d t\\
&+\sum_k\int_{\mathbb{R}} \beta_{\varepsilon}^{\prime}(u(t, x)-v(s, y)) \partial_x \sigma_k(u) \eta_{\delta, \delta_0} d x d B_k(t) \\
 =&-\int_{\mathbb{R}} \beta_{\varepsilon}^{\prime}(u(t, x)-v\left(s, y\right) b(u) u_x \partial_x \eta_{\delta , \delta_0} d x d t  -\int_{\mathbb{R}} \beta_{\varepsilon}^{\prime \prime}(u(t, x)-v(s, y)) b(u) u_x^2 \eta_{\delta , \delta_0} d x d t \\
& +\int_{\mathbb{R}} \beta_{\varepsilon}^{\prime}(u(t, x)-v(s, y)) f^{\prime}(u) u_x \eta_{\delta, \delta_0} d x d t  -\frac{1}{2}\sum_k \int_{\mathbb{R}} \beta_{\varepsilon}^{\prime}(u(t, x)-v(s, y)) \sigma_k^{\prime 2}(u) u_x \partial_x \eta_{\delta, \delta_0} d x d t \\
& +\sum_k\int_{\mathbb{R}} \beta_{\varepsilon}^{\prime}(u(t, x)-v(s, y)) \partial_x \sigma_k(u) \eta_{\delta, \delta_0} d x d B_k(t) \\
 =&\int_{\mathbb{R}}\left[b, \varepsilon\right](u, v) \partial_x^2 \eta_{\delta, \delta_0} d x d t-\int_{\mathbb{R}}f_\varepsilon(u, v) \partial_x \eta_{\delta, \delta_0} d x d t +\frac{1}{2} \sum_k\int_{\mathbb{R}}\left[\sigma_k^{\prime 2}, \varepsilon\right](u, v) \partial_{x}^2 \eta_{\delta, \delta_0} d x d t \\
& -\int_{\mathbb{R}} \beta_{\varepsilon}^{\prime \prime}(u-v) b(u) u_x^2 \eta_{\delta, \delta_0} d x d t -\sum_k\int_{\mathbb{R}}\left[\sigma_k^{\prime}, \varepsilon\right](u, v) \partial_x \eta_{\delta, \delta_0} d x d B_k(t).
\eaee
Integrating the drift term by parts, we observe that the resulting $\beta_\epsilon''$ term cancels exactly with the It\^o correction $\frac{1}{2}\beta_\epsilon'' \sigma_k'^2 u_x^2$. Thus, we obtain
\bae\label{kruz2.5}
& -\int_{\Pi_T}\int_\R \beta_{\varepsilon}\left(u_0(x)-v(s, y)\right) \eta_{\delta, \delta_0}\dif x dyds-\int_{\Pi_T}\int_{\Pi_T}\beta_{\varepsilon}(u(t,x)-v(s,y)) \partial_t \eta_{\delta, \delta_0} \dif x~\dif t dyds \\
 =&\int_{\Pi_T}\int_{\Pi_T}\left[b, \varepsilon\right](u, v) \partial_x^2 \eta_{\delta, \delta_0}\dif x~\dif t dyds-\int_{\Pi_T}\int_{\Pi_T}f_\varepsilon(u, v) \partial_x \eta_{\delta, \delta_0}\dif x~\dif t dy ds  \\
&+\frac{1}{2} \sum_k\int_{\Pi_T}\int_{\Pi_T}\left[\sigma_k^{\prime 2}, \varepsilon\right](u, v) \partial_x^2 \eta_{\delta, \delta_0}\dif x~\dif t  dy ds -\int_{\Pi_T}\int_{\Pi_T}\beta_{\varepsilon}^{\prime \prime}(u-v) b(u) u_x^2 \eta_{\delta, \delta_0}\dif x~\dif t  dy ds  \\
& -\left.\sum_k\int_0^T \int_{\mathbb{R}} \int_{\Pi_T}\left[\sigma_k^{\prime}, \varepsilon\right](u, a) \partial_x \eta_{\delta, \delta_0} \dif x  d B_k(t) \right|_{a=v(s, y)}   d y d s. 
\eae
Exchanging $u(t, x)$ and $v(s, y)$ yields that
\bae\label{kruz2.6}
\begin{aligned}
& -\int_{\Pi_T}\int_\R \beta_{\varepsilon}\left(v_0(y)-u(t, x)\right) \eta_{\delta , \delta_0}dy\dif x~\dif t -\int_{\Pi_T}\int_{\Pi_T} \beta_{\varepsilon}\left(v(s, y)-u(t, x)) \partial_s \eta_{\delta, \delta_0}\right. dy ds \dif x~\dif t  \\
 =& \int_{\Pi_T}\int_{\Pi_T} \left[b, \varepsilon\right](v, u) \partial_y^2 \eta_{\delta, \delta_0}dy ds \dif x~\dif t -\int_{\Pi_T}\int_{\Pi_T} f_\varepsilon(v, u) \partial_y \eta_{\delta, \delta_0}dy ds \dif x~\dif t  \\
& +\frac{1}{2}\sum_k \int_{\Pi_T}\int_{\Pi_T} \left[\sigma_k^{\prime 2}, \varepsilon\right](v, u) \partial_y^2 \eta_{\delta, \delta_0}dy ds \dif x~\dif t -\int_{\Pi_T}\int_{\Pi_T} \beta_{\varepsilon}^{\prime \prime}(v-u) b(v) v_y^2 \eta_{\delta, \delta_0} dy ds \dif x~\dif t \\
& -\sum_k\int_{\Pi_T}\int_{\Pi_T} \left[\sigma_k^{\prime}, \varepsilon\right](v, a) \partial_y \eta_{\delta, \delta_0} d y d B_k(s) d x  d t.  \\
\end{aligned}
\eae
By adding \eqref{kruz2.5} and \eqref{kruz2.6}, we arrive at
\bae\label{kruz}
0=&\mathbb{E}\int_{\Pi_T} \int_{\mathbb{R}} \beta_{\varepsilon}\left(u_0(x)-v(s, y)\right) \psi(s, y) \rho_{\delta_0}(-s) \varrho_\delta(x-y) d x d y d s \\
& +\mathbb{E} \int_{\Pi_T} \int_{\Pi_T} \beta_{\varepsilon}(v(s, y)-u(t, x)) \partial_s \psi(s, y) \rho_{\delta_0}(t-s) \varrho_\delta(x-y) d y d s d x d t \\
& +\mathbb{E} \int_{\Pi_T} \int_{\Pi_T}\left(\left[b, \varepsilon\right](u, v) \partial_x^2 \eta_{\delta, \delta_0}+\left[b, \varepsilon\right](v, u) \partial_y^2 \eta_{\delta, \delta_0}\right) d y d s d x d t \\
& -\mathbb{E} \int_{\Pi_T} \int_{\Pi_T}\left(f_\varepsilon(u, v) \partial_x \varrho_\delta(x-y)+f_\varepsilon(v, u) \partial_y \varrho_\delta(x-y)\right) \rho_{\delta_0} \psi d y d s d x d t \\
& -\mathbb{E} \int_{\Pi_T} \int_{\Pi_T}f_\varepsilon(v, u) \varrho_\delta\rho_{\delta_0} \partial_y \psi d y d s d x d t \\
& +\frac{1}{2}\sum_k \mathbb{E} \int_{\Pi_T} \int_{\Pi_T}\left(\left[\sigma_k^{\prime 2}, \varepsilon\right](u, v) \partial_x^2 \eta_{\delta, \delta_0}+\left[\sigma_k^{\prime 2}, \varepsilon\right](v, u) \partial_y^2 \eta_{\delta, \delta_0}\right) d y d s d x d t \\
& -\mathbb{E} \int_{\Pi_T} \int_{\Pi_T} \beta_{\varepsilon}^{\prime \prime}(u-v)\left(b(u) u_x^2+b(v) v_y^2\right) \eta_{\delta, \delta_0} d y d s d x d t \\
& +\left.\sum_k\mathbb{E} \int_{\Pi_T} \int_{\mathbb{R}} \int_0^T\left[\sigma_k^{\prime}, \varepsilon\right](u, a) \partial_y \eta_{\delta, \delta_0} d B_k(t)\right|_{a=v(s, y)} d y \\
 =:&I_1+I_2+I_3+I_4+I_5+I_6+I_7+I_8.
\eae
The terms $I_1, I_2, I_4,$ and $I_5$ correspond to the standard transport and temporal error estimates inherent to the doubling-of-variables method in the stochastic setting. As the convergence of these terms largely follows the entropy framework established in \cite{biswas2014stochastic,dareiotis2019entropy,dareiotis2020nonlinear}, we relegate their detailed proofs to Appendix \ref{app kruz}.   We focus the main text on the viscous dissipation terms $I_3, I_7$ and the stochastic correction $I_6,I_8$, where the interaction between the nonlinear noise and the viscous regularization requires specific adaptation.  Note that since typically $u$ and $v$ do not vanish at infinity, global $L^1$ estimates are not directly applicable. However, the test functions possess compact support in the spatial variables.
\begin{prp}\label{prp3.1}
     Let $I_1, I_2, I_4,$ and $I_5$ be the terms defined in \eqref{kruz}. Then
     \begin{enumerate}[{\textup{1}.}]
         \item { $\displaystyle\lim _{(\delta,\varepsilon)\rightarrow (0,0) }  \lim_{\delta_0 \rightarrow 0 } I_1 = \mathbb{E} \int_{\mathbb{R}} |u_0(x) - v_0(x)| \psi(0,x) \, dx,$}
         \item [{\textup{2}.}]{ $\displaystyle\lim _{(\delta,\varepsilon)\rightarrow (0,0) }  \lim_{\delta_0 \rightarrow 0 }  I_2 = \mathbb{E} \int_{\Pi_T} |v(s,x) - u(s,x)| \partial_s \psi(s,x) \, dx ds,$}
         \item[{\textup{3}.}]{$\displaystyle\lim _{\left(\varepsilon, \delta^{-1} \varepsilon, \delta_0\right) \rightarrow(0,0,0)}  I_4 = 0,$}
         \item [{\textup{4}.}]{ $\displaystyle\lim _{(\epsilon,\delta)\rightarrow (0,0) }  \lim_{\delta_0 \rightarrow 0 }  I_5 = \mathbb{E} \int_{\Pi_T} \operatorname{sign}(u-v)(f(u)-f(v)) \partial_x \psi \, dx dt$.}
     \end{enumerate}
  \end{prp}
\begin{proof}
    See Appendix \ref{app kruz}, Lemmas \ref{lm5.3} to \ref{lm56}.
\end{proof}


\begin{Lemma}
    \baee
 &\limsup _{(\varepsilon, \varepsilon \delta^{-1},\delta) \rightarrow(0,0,0)} \lim_{\delta_0 \downarrow 0}( I_3 +I_7)=0 .
    \eaee
\end{Lemma}

\begin{proof}
Following the arguments as in \cite[(4.6)-(4.10)]{dareiotis2019entropy} and \cite{dareiotis2020nonlinear}, we have
\bae\label{IBPC1C2}
I_3=\mathbb{E} \int_{\Pi_T} \int_{\Pi_T}\left(\left[b, \varepsilon\right](u, v) \partial_x^2 \eta_{\delta, \delta_0}+\left[b, \varepsilon\right](v, u) \partial_y^2 \eta_{\delta, \delta_0}\right) d y d s d x d t=:C_1+C_2,
\eae
where
    \baee\label{C1}
C_1= & -\mathbb{E} \int_{\Pi_T}\int_\R \partial_{x}^2 \eta_{\delta,\delta_0} \int_{v}^u \beta_\varepsilon^{\prime}(\zeta-v) b(\zeta)\, d \zeta \dif x  dy\dif t  \\
= & -\mathbb{E} \int_{\Pi_T}\int_\R \partial_{x}^2 \eta_{\delta,\delta_0} \int_{v}^u \int_{v}^\zeta \beta_\varepsilon^{\prime \prime}(\zeta-r) b(\zeta)\, d r d \zeta \dif x  dy\dif t  \\
= & -\mathbb{E} \int_{v \leq u} \partial_{x}^2 \eta_{\delta,\delta_0} \int_{v}^u \int_{v}^u I_{r \leq \zeta} \beta_\varepsilon^{\prime \prime}(\zeta-r) b(\zeta) \,d r d \zeta  \\
& -\mathbb{E} \int_{v \geq u} \partial_{x}^2 \eta_{\delta,\delta_0} \int_u^{v} \int_u^{v} I_{r \geq \zeta} \beta_\varepsilon^{\prime \prime}(\zeta-r) b(\zeta)\, d r d \zeta
    \eaee
and
\bae\label{C2}
C_2=& -\mathbb{E} \int_{u \leq v} \partial_{y}^2 \eta_{\delta,\delta_0} \int_{u}^v \int_{u}^v I_{r \leq \zeta} \beta_\varepsilon^{\prime \prime}(r-\zeta) b(r) \,d r d \zeta \\ 
& -\mathbb{E} \int_{u \geq v} \partial_{y}^2 \eta_{\delta,\delta_0} \int_v^{u} \int_v^{u} I_{r \geq \zeta} \beta_\varepsilon^{\prime \prime}(r-\zeta) b(r) \,d r d \zeta .
\eae
Since $\beta^{\prime\prime}$ is even, it follows from \eqref{IBPC1C2} and \eqref{C2} that
\baee
I_3\leq \mathbb E \int_{\Pi_T} \int_{\Pi_T} |\Delta \eta_{\delta, \delta_0}| \int_v^u \int_v^u \beta_{\varepsilon}^{\prime \prime}(\xi-r)(b(r)+b(\zeta)) \,d r d\xi d y d s d x d t.
\eaee
Due to the convexity of $\beta_\varepsilon$,
\baee
I_7&=-\mathbb{E} \int_{\Pi_T} \int_{\Pi_T} \beta_{\varepsilon}^{\prime \prime}(u-v)\left(b(u) u_x^2+b(v) v_y^2\right) \eta_{\delta, \delta_0} \,d y d s d x d t\\
&\leq -2\mathbb{E} \int_{\Pi_T} \int_{\Pi_T} \beta_{\varepsilon}^{\prime \prime}(u-v)\left(b(u) b(v)\right)^\frac12 | u_x v_y| \eta_{\delta, \delta_0} \,d y d s d x d t.
\eaee
Thus,
\baee
I_3+I_7\leq   \mathbb{E} \int_{\Pi_T} \int_{\Pi_T} |\Delta \eta_{\delta, \delta_0}| \int_v^u \int_v^u \beta_{\varepsilon}^{\prime \prime}(\xi-r)\left|b^{\frac{1}{2}}(r)-b^{\frac{1}{2}}(\xi)\right|^2 d r d\xi d y d s d x d t.
\eaee
Note that, by \eqref{mp}, there exists a constant $\mathfrak{M} > 0$ (depending on the initial data $u_0, v_0$ and the rarefaction profile) such that almost surely $|u(t,x)|,|v(t,x)| \le \mathfrak{M}$. Then
\baee
& \int_v^u \int_v^u \beta_{\varepsilon}^{\prime \prime}(\xi-r)\left|b^{\frac{1}{2}}(r)-b^{\frac{1}{2}}(\xi)\right|^2 d r d\xi \\
= & \int_v^u \int_v^u \beta_{\varepsilon}^{\prime \prime}(\xi-r) I_{\{| r-\xi|<\varepsilon,| \xi |<2 \varepsilon\}}\left|b^{\frac{1}{2}}(r)-b^{\frac{1}{2}}(\xi)\right|^2 d r d\xi \\
\leq & \int_v^u \int_v^u \beta_{\varepsilon}^{\prime\prime}(\xi-r) \varepsilon^2 \sup _{z \in[\xi-\varepsilon, \xi+\varepsilon]}\left|(b^{\frac{1}{2}}(z))^\prime\right|^2 d r d\xi \\
\leq &\, 2K_{b, \mathfrak{M}}^2 \varepsilon^2|u-v|, 
\eaee
where the last equality follows from  {\rm{(H1)}} with 
$$K_{b, \mathfrak{M}} := \sup_{|z| \le \mathfrak{M} + 1} \left|  (b^{1/2}(z))' \right|\le \frac{\|b'\|_{L^\infty_{\mathfrak{M}+1}}}{2\sqrt{\mu}}<\infty$$
combined with the estimate
\baee
\left|\int_v^u \int_v^u \beta_{\varepsilon}^{\prime\prime}(\xi-r)  d r d\xi\right|\leq 2|u-v|.
\eaee
Finally, we have
\baee
I_3+I_7  \leq&\, C \varepsilon^2 \mathbb{E} \int_{\Pi_T} \int_{\Pi_T} |\Delta \eta_{\delta, \delta_0}|\cdot |u(t,x)-v(s,y)| \psi\, d x d y d s d t \\
 \leq&\, C  \varepsilon^2\delta^{-2} \mathbb{E} \int_{\Pi_T} \int_{\Pi_T} \rho_{\delta_0}(t-s)  \delta^{-1} \left| \varrho^{\prime \prime}\left(\frac{x-y}{\delta}\right) \right||u(t,x)-v(s,y)| \psi(s, y)\, d x d y d s d t \\
 &+C  \varepsilon^2\mathbb{E} \int_{\Pi_T} \int_{\Pi_T} \rho_{\delta_0}(t-s) \delta^{-1} \left|  \varrho\left(\frac{x-y}{\delta}\right) \right||u(t,x)-v(s,y)| \cdot|\psi^{\prime \prime}(s, y)|\, d x d y d s d t .\\
\eaee
Using the same reasoning as before and letting $\delta_0 \downarrow 0$, we have
\baee
&\lim_{\delta_0\rightarrow0}(I_3+I_7) \\
\leq&\, C  \varepsilon^2\delta^{-2} \mathbb{E} \int_{\R} \int_{\Pi_T}  \delta^{-1} \left| \varrho^{\prime \prime}\left(\frac{x-y}{\delta}\right) \right||u(s,x)-v(s,y)| \psi(s, y) \,d x d y d s  \\
 &+C  \varepsilon^2\mathbb{E} \int_{\R} \int_{\Pi_T}  \delta^{-1} \left|  \varrho\left(\frac{x-y}{\delta}\right) \right||u(s,x)-v(s,y)|\cdot |\psi^{\prime \prime}(s, y)| \,d x d y d s\\
 \leq&\, C  \varepsilon^2\delta^{-2} \sup_{0\leq s \leq T}\mathbb E\left(\|u(s)\|+\|v(s)\|\right)+C  \varepsilon^2 \sup_{0\leq s \leq T}\mathbb E\left(\|u(s)\|+\|v(s)\|\right).
\eaee
Hence the Lemma follows.
\end{proof}
We defer the proofs of the next two lemmas to Appendix  \ref{app kruz}. First, we need a similar version of the stochastic strong entropy solution \cite[Lemma 4.25]{feng2008stochastic} to handle the non-adaptedness  of the stochastic integral $I_8$.
\begin{Lemma}\label{adaptlm}
    For each $T>0$, there exists a deterministic function $A\left(\delta, \delta_0\right)$ such that
\bae\label{adapt11}
&I_8= \sum_k \mathbb E\left[\left.\int_0^T \int_\R \int_0^T \int_\R \partial_x \sigma_k( {u}(r, x)) \beta^{\prime}({u}(r, x)-v) \eta_{\delta, \delta_0}(r, x, s, y) d x d B_k(r)\right|_{v=v(s, y)} d y d s\right] \\
& \leq-\sum_k\mathbb {E}  \int_{\Pi_T} \int_{\Pi_T}  \left[\sigma_k^{\prime} \beta^{\prime \prime}(\cdot-v(t,y))\right](u(t,x))  \partial_y \sigma_k(v(t,y)) \partial_x \eta_{\delta , \delta_0} d t \dif x  dy ds+A\left(\delta, \delta_0\right) .
\eae
Furthermore, for fixed $\delta, \psi$ and $\beta$, it holds that
\bae\label{adapt}
\lim _{\delta_0 \rightarrow 0} A\left(\delta, \delta_0\right)=0.
\eae
\end{Lemma}

\begin{Lemma}\label{molast}
    \baee
 &\lim _{(\delta, \varepsilon \delta^{-2}) \rightarrow(0,0)} \lim_{\delta_0\rightarrow0}(I_6 +I_8) =0 .
    \eaee
\end{Lemma}

\subsection{The Stochastic $\rm{L^1_{{loc}}}$ Contraction}
\begin{proof}[\textbf{Proof of Lemma \ref{lockruz}}]
By Proposition \ref{prp3.1} and Lemmas \ref{mol} to \ref{molast}, letting $\delta_0 \downarrow 0$ and setting $\delta=\varepsilon^{\frac{1}{3}} $ as $\varepsilon \downarrow 0$, we have
    \begin{equation}\label{test}
    \begin{aligned}
&\mathbb E\left[\int_{\R}\left|u_0(x)-v_0(x)\right| \psi(0, x) d x\right]+\mathbb E\left[\int_{\Pi_T}|u(t, x)-v(t, x)| \partial_t \psi(t, x) d t d x\right] \\
+&\mathbb E\left[\int_{\Pi_T} \operatorname{sign}(u-v) f(u(t, x)-v(t, x))  \partial_x \psi(t, x) d t d x \right] \geq 0.
     \end{aligned}
     \end{equation}
Set     $$
    S:= \max\big\{\|f^\prime(u(t))\|_{L^\infty(\R)},\|f^\prime(v(t))\|_{L^\infty(\R)}\big\}.
     $$ 
     Using  \eqref{mp}, we have
     \baee
\left| f(u(t,x))-f(v(t,x))\right| \leq S |u(t,x)-v(t,x)|.
     \eaee
Let $t_0\in(0,T)$, $r>0$ and $\varepsilon>0$ be small. Choose the test function $\psi_\varepsilon(t,x)=\omega_\varepsilon(t)\chi_\varepsilon(t,x)$,
     $$
\omega_\varepsilon(t)= \begin{cases}1, & \text{if } 0 \leq t<t_0;\\ \varepsilon^{-1}(t_0-t)+1, & \text{if } t_0 \leq t<t_0+\varepsilon; \\ 0, & \text{if } t_0+\varepsilon \leq t<\infty,\end{cases}
$$
and
$$
\chi_\varepsilon(t, x)= \begin{cases}1, &\text{if } |x|-r-S(t_0-t)<0; \\ \varepsilon^{-1}\left[r+S(t_0-t)-|x|\right]+1, &\text{if } 0 \leq|x|-r-S(t_0-t)<\varepsilon; \\ 0,&\text{if } |x|-r-S(t_0-t) \geq \varepsilon.\end{cases}
$$
Let $\psi=\psi_\varepsilon$ in \eqref{test}. From the approach in \cite[Theorem 6.2.3]{dafermos2005hyperbolic}, we have
\begin{equation*}
\begin{aligned}
&\frac{1}{\varepsilon} \mathbb E\int_{t_0}^{t_0+\varepsilon} \int_{|x|<r}|u(t, x)-v(t, x)| d x d t+\frac{1}{\varepsilon} \mathbb E\int_{0}^{t_0} \int_{r+S(t_0-t)<x<r+S(t_0-t)+\varepsilon}S|u(t, x)-v(t, x)| d x d t\\
=& -\frac{1}{\varepsilon} \mathbb E\int_0^{t_0} \int_{r+S(t_0-t)<x<r+S(t_0-t)+\varepsilon}\operatorname{sign}(x)\operatorname{sign}(u-v)\left(f(u)-f(v)\right)  d x d t \\
& +\int_{|x|<r+S t_0}\left|u_0(x)-v_0(x)\right| d x+\mathcal B(\varepsilon),
\end{aligned}
\end{equation*}
where
\baee
\mathcal B(\varepsilon) \leq &  \int_{r+S t_0 \leq |x|\leq r+S t_0+\varepsilon} \left|u_0(x)-v_0(x)\right| d x+\frac1 \varepsilon \mathbb E \left[  \int_{t_0}^{t_0+\varepsilon} \int_{r \leq |x|\leq r+\varepsilon} |u(t, x)-v(t, x)| d x d t\right ]\\
&+\frac1 \varepsilon\mathbb E \left[ \int_{t_0}^{t_0+\varepsilon} \int_{r+S(t_0-t) \leq |x|\leq r+S(t_0-t)+\varepsilon} |u(t, x)-v(t, x)| d x d t\right ]\\
&+\frac{1}{\varepsilon} \mathbb E \left[\int_{t_0}^{t_0+\varepsilon} \int_{r+S(t_0-t)<x<r+S(t_0-t)+\varepsilon}\left|f(u(t,x))-f(v(t,x))\right|  d x d t\right]\\
\leq& \varepsilon^\frac12 \|u_0-v_0\|+ 2 \varepsilon^{-\frac12} \mathbb E \left[  \int_{t_0}^{t_0+\varepsilon} \|u(t)-v(t)\|  d t\right ]+\varepsilon^{-\frac12} \mathbb E \left[ S_{t_0+\varepsilon} \int_{t_0}^{t_0+\varepsilon}  \|u(t)-v(t)\|  d t\right ]\\
\leq& \varepsilon^\frac12 \|u_0-v_0\|+ (2+ S) \mathbb E \left[  \int_{t_0}^{t_0+\varepsilon} \|u(t)-v(t)\|^2  d t\right ]^{\frac12}\\
=&O(\varepsilon),
\eaee
where in the last equality we have used \eqref{mp}. Thus we have
\begin{equation}\label{lockr}
\begin{aligned}
&\frac{1}{\varepsilon} \mathbb E\int_{t_0}^{t_0+\varepsilon} \int_{|x|<r}|u(t, x)-v(t, x)| d x d t  \leq \int_{|x|<r+S t_0}\left|u_0(x)-v_0(x)\right| d x \\
& -\frac{1}{\varepsilon} \mathbb E\int_0^{t_0} \int_{r+S(t_0-t)<x<r+S(t_0-t)+\varepsilon}\left[S |u-v|+\operatorname{sign}(x)\operatorname{sign}(u-v)\left(f(u)-f(v)\right)  \right] d x d t+O(\varepsilon).
\end{aligned}
\end{equation}
Let
$$
A(s)=\mathbb E\left[\int_{|x|<r}|u(s, x)-v(s, x)| d x\right].
$$
Then $A \in L_{\text {loc }}^1([0, \infty))$. It follows that any right Lebesgue point of $A(t)$ is also a right Lebesgue point of
$$
A_\varepsilon(s)=\mathbb E\left[\int_{\mathbb{R}} \psi_\varepsilon(s,x)|u(s, x)-v(s, x)| d x\right].
$$
Let $t_0$ be a right Lebesgue point of $A$. Letting $\epsilon \downarrow 0$ in \eqref{lockr}, we obtain
$$
\mathbb E\int_{|x|<r}|u( t_0,x)-v( t_0,x)| d x   \leq \mathbb E\int_{|x|<r+S t_0}\left|u_0(x)-v_0(x)\right| d x.
$$
This completes the proof of Lemma \ref{lockruz}.
\end{proof}

\begin{proof}[\textbf{Proof of Theorem \ref{loccon}}]
   The result follows immediately from Lemma \ref{lockruz} and \eqref{mp}.
\end{proof}
\begin{remark}\label{invicidremark}
     We point out that this $L^1_{\rm{loc}}$ contraction remains valid for the inviscid case. As shown in {\rm{\cite[Prop 3.7]{gassiat2019regularization}}}, they showed that for $\sigma(u)=f(u)=\frac12 u^2$ and $\mu=0$ with $u_0\in W^{1,\infty}(\R)\bigcap L^\infty(\R)$,
     $$\|u\|_{L^\infty(\R)}(t)\leq  \sqrt{2}\|u_0\|_{L^1(\R)}^\frac{1}{2}\left( \sup _{0 \leq s \leq t} |L(s)| \right)^{-\frac12}, $$
     where $L(t)$ is an Ornstein-Uhlenbeck process starting from $\|u_0^\prime\|_{L^\infty(\R)}^{-1}$. 
\end{remark}
\begin{remark}\label{rm 3.2}
   The $L^1_{\rm{loc}}$ contraction also holds for the inviscid case, i.e., $b=0$ in equation \eqref{conservative} with the coefficients satisfying that
   \begin{enumerate}
        \item[\rm{($\widetilde{\rm{H1}}$)}] { $f,f^\prime$ and $f^{\prime\prime}$ admit at most polynomial growth, and}
		\item[\rm{($\widetilde{\rm{H2}}$)}] { all $\sigma_k^\prime$ have at most polynomial growth, and there exist $\gamma>0$ and $\lambda_k$ such that 
  \bae\label{noisepoly}
&|\sigma_k^\prime(x)|\leq \lambda_k (1+|x|^\gamma), ~\sum_k \lambda_k ^4&<\infty, 
  \eae
  \bae\label{noise-holder}
&\|\sigma^\prime\|_{l^2(C^{1/2}(\R))}^2=\sum_k \|\sigma_k^\prime\|_{C^{1/2}(\R)}^2<\infty.
  \eae
  
		}
  \end{enumerate}
 \eqref{noisepoly} and \eqref{noise-holder} are used in the proof of the inviscid version of Lemmas {\rm{\ref{adaptlm}}} and {\rm{\ref{molast}}}, respectively. 
  
Let {\rm{($\widetilde{\rm{H1}}$)}} and {\rm{($\widetilde{\rm{H2}}$)}} hold. Assume $T>0$ and $u,v$ are two entropy solutions such that at least one is a stochastic strong entropy solution (see \cite{feng2008stochastic}) satisfying
 \baee
\mathbb E \sup_{0\leq t \leq T}  \left[ \|u\|_{p}^p \right]\leq C_p(T)<\infty,\,\text{ for all $2\leq p< \infty.$}
 \eaee
Then the $L^1_{\rm{loc}}$ contraction \eqref{locL1} holds by similar arguments.
\end{remark}

\section{Global Existence of Strong Solutions}\label{sec global}
\subsection{Cut-off Approximation}\label{sec cutoff}
We first solve a sequence of cut-off initial value problems in the bounded domain $(-N, N)$.
For any $N \ge 1$, we truncate the coefficients and initial data as follows

\begin{enumerate}
    \item $f_N \in C^2$, $\text{supp } f_N \subseteq [-N+\frac{1}{2}, N-\frac{1}{2}]$, $\text{supp}(f_N - f) \subseteq [-N+1, N-1]^c$,
    \[
    \lim_{N\to\infty} f_N' = f' \text{ in } L^\infty_{\text{loc}}(\mathbb{R}), \quad \lim_{N\to\infty} f_N'' = f'' \text{ in } L^\infty_{\text{loc}}(\mathbb{R}), \quad \inf_{-N \le x \le N} f_N'' \ge \alpha.
    \]
    \item $\sigma_{i,N} \in C^2$, $\text{supp } \sigma_{i,N} \subseteq [-N+\frac{1}{2}, N-\frac{1}{2}]$, $\text{supp}(\sigma_{i,N} - \sigma_i) \subseteq [-N+1, N-1]^c$,
    \[
    \lim_{N\to\infty} \sigma_N' = \sigma' \text{ in } l^2(H^1_{\text{loc}}(\mathbb{R})).
    \]
    \item $\phi_0^{(N)} \in H_0^1(-N, N)$, $\|\phi_0^{(N)}\|_{L^\infty} \le \|\phi_0\|_{L^\infty(\mathbb{R})}$, $\displaystyle\lim_{N\to\infty} \phi_0^{(N)} = \phi_0$ in $H^1(\mathbb{R})$.
\end{enumerate}

To ensure the global well-posedness of the Galerkin approximations and derive uniform derivative estimates, we introduce a smooth cut-off function $\theta_R \in C^\infty(\mathbb{R})$ associated with a threshold $R > \|\phi_0\|_{L^\infty(\mathbb{R})} + (u_+ - u_-)$ (to be determined in Lemma \ref{lm2.4}), satisfying
\[
\theta_R(s) = \begin{cases} s & \text{if } |s| \le R, \\ 2R\text{sign}(s) & \text{if } |s| \ge 2R, \end{cases}
\]
where $|\theta_R'| \le 1$ and $|\theta_R''| \le C/R$. We define the modified coefficients in terms of the unknown $\phi^{(N)}$ as follows
\begin{equation} \label{eq:cutoff_coeffs}
\tilde{b}_N(v) := b_N(\theta_R(v+\bar{u})), \quad \tilde{\sigma}_{i,N}(v) := \sigma_{i,N}(\theta_R(v+\bar{u})).
\end{equation}
Note that the derivatives $\tilde{b}_N'$ and $\tilde{\sigma}_{i,N}''$ are compactly supported in $\{v : |v+\bar{u}| \le 2R\}$.

We consider the following {quasilinear} stochastic parabolic equation on the interval $(-N, N)$ with Dirichlet boundary conditions
\begin{equation} \label{d1.2}
\begin{cases}
d\phi^{(N)} + \partial_x \left[ f_N(\phi^{(N)}+\bar{u}) - f_N(\bar{u}) \right] dt \\
\quad = \partial_x \left[ \left( \tilde{b}_N(\phi^{(N)}) + \frac{1}{2}\sum_{i=1}^N (\tilde{\sigma}_{i,N}')^2(\phi^{(N)}) \right) (\partial_x \phi^{(N)} + \partial_x \bar{u}) \right] dt \\
\quad \quad + \sum_{i=1}^N \partial_x \left[ \tilde{\sigma}_{i,N}(\phi^{(N)}) \right] dB_i(t), \\
\phi^{(N)}(t, \pm N) = 0, \quad \phi^{(N)}(0) = \phi_0^{(N)}.
\end{cases}
\end{equation}
In this approximation scheme, we truncate the noise sum to $N$ terms and restrict the spatial domain to $(-N, N)$. Note that, under Hypothesis (H3), the tail of the stochastic series vanishes asymptotically. Thus, considering the diagonal sequence $N \to \infty$ suffices to capture the full stochastic dynamics.

We define the solution space $$X_i^{(N)}(T) := L^2\big(\Omega; L^2(0,T; H_0^i(-N,N))\big) \cap L^2\big(\Omega; C([0,T]; H_0^{i-1}(-N,N))\big),i\in \mathbb N.$$

Next we give the $L^p$ estimates which include a useful locally maximal principle.
	\begin{Lemma}\label{Lemma4.3}
		Let $\phi^{(N)}\in X_1^{(N)}(T)$ be the solution of \eqref{d1.2}. Then there exists a constant $C > 0$, independent of the cut-off parameter $N$, such that,
  \begin{enumerate}
      \item [{\rm{1.}}] {If $p=2,$ then
			\begin{equation}\label{8.18}
		\begin{aligned}
			& \mathbb E\left[\left\|\phi^{(N)}(t)\right\|_{L^{2} _N}^{2} + \int_0^t\int_{B(N)}\left|\phi^{(N)}\right|^{2} \bar{u}_{x} d x ~ d s + \int_0^t\left\| \phi^{(N)}_{x} \right\|_{L_N^2}^2 d s\right] \leq \left\|\phi^{(N)}_0\right\|_{L^{2} _N}^{2}+ C\ln(1+t).
		\end{aligned}
	\end{equation}}
     \item [{\rm{2.}}] {If $p>2,$ then
		\begin{equation}\label{8.17}
			\begin{aligned}
				& \left\|\phi^{(N)}(t)\right\|_{L^{p} _N}^{p} + \int_0^t \int_{B(N)}\left|\phi^{(N)}\right|^{p} \bar{u}_{x} d x~ d s + \int_0^t\int_{B(N)}\left|\phi^{(N)}\right|^{p-2} |\phi^{(N)}_{x}|^2 d x~ d s \\
				&\leq \left\|\phi_0^{(N)}\right\|_{L^{p} _N}^{p}+ C \int_0^t (1+s)^{-\frac{p}{2}} ds+p C\int_0^t (1+s)^{-1}\left\|\phi^{(N)}\right\|_{L^{p} _N}^{p}ds + p dM_p^{(N)}(t),\quad {\rm{a.s.}}.
			\end{aligned}
		\end{equation}}
        \item [{\rm{3.}}] {If $p=\infty,$ then
		\begin{equation*}\label{r8.11}
			\begin{aligned}
			\sup_{0\le t\le T}	\left\|\phi^{(N)}(t)\right\|_{L^\infty _N} \leq  \| \phi_{0} \|_{L^\infty (\R)}, \quad \text { {\rm{a.s.}}.}
			\end{aligned}
		\end{equation*}}
  \end{enumerate}

	\end{Lemma}
	\begin{proof}
		Multiplying both sides of \eqref{d1.2} by $\left|\phi^{(N)}\right|^{p-2}\phi^{(N)}$ and  integrating over $(-N,N)$, we obtain
	\begin{equation}\label{8.12}
		\begin{aligned}
             & \int_{B(N)} \left|\phi^{(N)}\right|^{p-2} \phi^{(N)}d \phi^{(N)}  d x\\
=&-\int_{B(N)} \partial_{x}\left[f_N\left(\bar{u}+\phi^{(N)}\right)-f_N\left( \bar{u}\right)\right]\left|\phi^{(N)}\right|^{p-2} \phi^{(N)} d x~ d t \\
			&+ \int_{B(N)} \partial_{x}\left[\left (\tilde b_N+\frac{1}{2}\sum_{i=1}^{N}\tilde \sigma_{i,N}^{\prime 2}\right) \partial_x\left(\phi^{(N)}+\bar u\right)\right]\left|\phi^{(N)}\right|^{p-2}\phi^{(N)} \dif x~\dif t +dM_p^{(N)}(t)\\
			=& -\int_{B(N)} \partial_{x}\left[f_N\left(\bar{u}+\phi^{(N)}\right)-f_N\left( \bar{u}\right)\right]\left|\phi^{(N)}\right|^{p-2} \phi^{(N)} d x~ d t \\
			&-(p-1)\int_{B(N)} \left[\left (\tilde b_N+\frac{1}{2}\sum_{i=1}^{N}\tilde \sigma_{i,N}^{\prime 2}\right) \left(\phi_x^{(N)}+\bar u_x\right)\right]\left|\phi^{(N)}\right|^{p-2}\phi^{(N)}_x \dif x~\dif t +dM_p^{(N)}(t),\\
		\end{aligned}
	\end{equation}
	where $$M_p^{(N)}(t):=\int_{0}^{t} \int_{B(N)} \sum_{i=1}^{N} \tilde \sigma_{i,N}^{\prime}\left(\phi^{(N)}_{x}+\bar{u}_{x}\right)\left|\phi^{(N)}\right|^{p-2} \phi^{(N)} d x d B_i(t).$$
 First, we estimate the right-hand side of the last equality in \eqref{8.12}. Using integration by parts, we have
	\baee
			&\partial_{x}\left[f_N\left(\bar{u}+\phi^{(N)}\right)-f_N\left( \bar{u}\right)\right]\left|\phi^{(N)}\right|^{p-2} \phi^{(N)}\\
            =&(p-1)\bar u_x \int_{0}^{\phi^{(N)}}\left(f_N^{\prime}\left( \bar{u}+y\right)-f_N^{\prime}\left( \bar{u}\right)\right)|y|^{p-2} d y \\
			&+\partial_{x}\left[\left|\phi^{(N)}\right|^{p-2} \phi^{(N)}\left(f_N \left(\bar {u}+\phi^{(N)}\right)-f_N \left( \bar{u}\right)\right)\right]\\
            &+\partial_{x}\left[-(p-1) \int_{0}^{\phi^{(N)}}\left(f_N \left( \bar{u}+y\right)-f_N \left( \bar{u}\right)\right)|y|^{p-2} d y\right].
		\eaee
	By $\rm{(H2)}$, it has
	\begin{equation}\label{f_N}
		\begin{aligned}
			\int_{B(N)}\int_{0}^{\phi^{(N)}}\left(f_N ^{\prime}\left( \bar{u}+y\right)-f_N ^{\prime}\left(\bar {u}\right)\right)|y|^{p-2} d y  \bar{u}_{x} d x \geq \frac\alpha p \int_{B(N)}\left|\phi^{(N)}\right|^{p} \bar u_x \dif x .
		\end{aligned}
	\end{equation}
 Substituting \eqref{f_N} and the It\^o formula
	\begin{equation*}
 \begin{aligned}
		&\frac{1}{p}   d \int_{B(N)}\left|\phi^{(N)}\right|^{p} d x\\
  =&\int_{B(N)}\phi^{(N)}\left|\phi^{(N)}\right|^{p-2} d \phi^{(N)} d x+\frac{(p-1)}{2}\sum_{i=1}^N \int_{B(N)}\tilde \sigma_{i,N}^{\prime2}\left|\phi^{(N)}\right|^{p-2} \left(\phi^{(N)}_{x}+\bar u_x\right)^2  d x\dif t 
   \end{aligned}
	\end{equation*}
 into \eqref{8.12}, by cancellation and $\rm{(H3)}$, we get
	\begin{equation}\label{3.21n}
		\begin{aligned}
			& \frac{1}{p} d \int_{B(N)}\left|\phi^{(N)}\right|^{p} d x+\frac{(p-1)\alpha} p \int_{B(N)}\left|\phi^{(N)}\right|^{p} \bar u_x \dif x~\dif t  +(p-1)\mu \int_{B(N)}  \left|\phi^{(N)}\right|^{p-2} |\phi^{(N)}_{x}|^2 d x d t \\
			\leq &-\frac{p-1}{2} \sum_{i=1}^{N}\int_{B(N)} \tilde \sigma_{i,N}^{\prime2}\left| \phi^{(N)}\right|^{p-2}(\phi^{(N)}_{x}+\bar u_x) \bar{u}_{x}\dif x~\dif t \\
   &-(p-1) \int_{B(N)}\tilde  b_N \left|\phi^{(N)}\right|^{p-2} \phi^{(N)}_{x} \bar{u}_{x} d x d t+dM_p(t)\\
   =:& -\frac{p-1}{2}I(p)\dif t -(p-1)J(p)\dif t +dM_p(t).
		\end{aligned}
	\end{equation}
\textbf{Case 1. $\bf{p=2}$}, first we estimate $I(2)$
\bae\label{L^2 I}
I(2)=&\int_{B(N)}\sum_{i=1}^{N}\left(\tilde \sigma_{i,N}^{\prime}\right)^{2}  (\phi^{(N)}_{x}+\bar u_x) \bar{u}_{x} d x\leq \| \sigma^\prime\|^2_{l^2(L^\infty_R)}\int_{B(N)} \left|\phi^{(N)}_{x}+\bar u_x\right| \bar{u}_{x} d x\\
\leq& \varepsilon \left\|\phi^{(N)}_{x}\right\|^2+ C\|\bar u_x\|^2\leq \varepsilon \left\|\phi^{(N)}_{x}\right\|^2+ C(1+t)^{-1}.\\
\eae
Next, $J(2)$ is controlled as follow
	\begin{equation}\label{L^2 b_N}
		\begin{aligned}
&J(2)=	- \int_{B(N)}\tilde b_N \phi^{(N)}_{x} \bar{u}_{x} d x\leq   \|b\|_{L^\infty_R}\int_{B(N)} \left|\phi^{(N)}_{x} \right|\bar{u}_{x} d x \\
&\leq \varepsilon \left\| \phi^{(N)}_x  \right\|^2+C\|\bar u_x\|^2\leq  \varepsilon \left\| \phi^{(N)}_x  \right\|^2+C(1+t)^{-1}.
		\end{aligned}
	\end{equation}
Substituting \eqref{L^2 I} and \eqref{L^2 b_N} into \eqref{3.21n} and integrating over $[0,t]$ yields 
 \begin{equation}\label{cutoffL^2 orbit}
		\begin{aligned}
			&\left\|\phi^{(N)}(t)\right\|_{L^{2} _N}^{2} + \int_0^t\int_{B(N)}\left|\phi^{(N)}\right|^{2} \bar{u}_{x} d x ~ d s + \int_0^t\left\| \phi^{(N)}_{x} \right\|_{L_N^2}^2 d s \leq \left\|\phi^{(N)}_0\right\|_{L^{2} _N}^{2}+ C\ln(1+t)+2M_2^{(N)}(t).
		\end{aligned}
	\end{equation}
 Taking expectation of both sides of \eqref{cutoffL^2 orbit}, we obtain \eqref{8.18}.
 
\textbf{Case 2. $\bf{p>2}$}	, first divide $I(p)$ as follows
	\begin{equation*}
		\begin{aligned}
			I(p)=& \int_{B(N)} \sum_{i=1}^{N}\left(\tilde \sigma_{i,N}^{\prime}\right)^{2} \left|\phi^{(N)}\right|^{p-2} (\phi^{(N)}_{x}+\bar u_x) \bar{u}_{x} d x \\
			=&  \sum_{i=1}^{N} \int_{B(N)} \left(\tilde \sigma_{i,N}^{\prime}\right)^{2} \left|\phi^{(N)}\right|^{p-2} \phi^{(N)}_{x}\bar{u}_{x} d x+\int_{B(N)} \sum_{i=1}^{N}\left(\tilde \sigma_{i,N}^{\prime}\right)^{2} \left|\phi^{(N)}\right|^{p-2}  \bar{u}_{x}^2 d x \\
			=:& I_{1}(p)+I_{2}(p).
		\end{aligned}
	\end{equation*}
\bae\label{2.20nn}
I_1(p)=&\int_{B(N)}\sum_{i=1}^{N}\left(\tilde \sigma_{i,N}^{\prime}\right)^{2} \left|\phi^{(N)}\right|^{p-2} \phi^{(N)}_{x}\bar{u}_{x} d x\leq \| \sigma^\prime\|^2_{l^2(L^\infty_R)}\int_{B(N)} \left|\phi^{(N)}\right|^{p-2} \left|\phi^{(N)}_{x}\right| \bar{u}_{x} d x\\
\leq& C \left\| \sigma^\prime\right\|^2_{l^2(H^1_R)}\int_{B(N)}\left(\left|\phi^{(N)}\right|^\frac{p-2}2 \left| \phi^{(N)}_x  \right|\right)\left(\left| \phi^{(N)}  \right|^{\frac{p-2}2} \bar{u}_{x}^{\frac{p-2}{2p}}\right) \bar u_x^{\frac{p+2}{2p}} d x \\
	&\leq {\varepsilon} \int_{B(N)}\left|\phi^{(N)}\right|^{p-2} |\phi^{(N)}_{x}|^2 d x+ \varepsilon^{-1}\|\bar{u}_{x}\|_\infty\int_{B(N)}\left|\phi^{(N)}\right|^{p}  d x+p^{-1} C\| \bar{u}_{x}\|^{\frac{1}{2} p+1}_{\frac{1}{2} p+1} d x \\
			& \leq {\varepsilon} \int_{B(N)}\left|\phi^{(N)}\right|^{p-2} |\phi^{(N)}_{x}|^2 d x+ C(1+t)^{-1}\varepsilon^{-1}\left\|\phi^{(N)}\right\|_{L^p_N}^{p} +p^{-1} C(1+t)^{-\frac p2},
\eae
where the last inequality follows from Young's inequality. Next
\bae\label{real-I_2(p)}
I_2(p) &=\int_{B(N)}  \sum_{i=1}^{N}\left(\tilde \sigma_{i,N}^{\prime}\right)^{2} \left|\phi^{(N)}\right|^{p-2}  \bar{u}_{x}^2 d x \leq  C \left\| \sigma^\prime\right\|^2_{l^2(H^1_R)}\int_{B(N)}  \left|\phi^{(N)}\right|^{p-2}  \bar{u}_{x}^2 d x\\
&\leq  C\int_{B(N)}  \left|\phi^{(N)}\right|^{p-2} \bar{u}_{x}^\frac{p-2}{p} \bar{u}_{x}^\frac{p+2}{p} d x \leq C \|\bar{u}_{x}\|_\infty  \int_{B(N)}\left|\phi^{(N)}\right|^{p}   \dif x  + C\|\bar u_x\|_{\frac{p+2}{2}}^\frac{p+2}{2} \\
&\leq \varepsilon (1+t)^{-1} \left\|\phi^{(N)}\right\|^{p}_{L^p_N}   + C\varepsilon^{-1}(1+t)^{-\frac p2} .
\eae
Combining \eqref{2.20nn} and \eqref{real-I_2(p)} yields
\bae\label{real I(p)}
I(p)\leq& \varepsilon \int_{B(N)} \left|  \phi^{(N)}\right|^{p-2}|\phi_x^{(N)}|^2\dif x  +C\varepsilon (1+t)^{-1} \left\|\phi^{(N)}\right\|^{p}_{L^p_N}+C \varepsilon^{-1}  (1+t)^{-\frac p2}.
\eae
Similarly, we can estimate
	\begin{equation}\label{b_N}
		\begin{aligned}
&	J(p)=\int_{B(N)}b_N\left|\phi^{(N)}\right|^{p-2} \phi^{(N)}_{x} \bar{u}_{x} d x\leq   \|b\|_{L^\infty_R}\int_{B(N)}\left|\phi^{(N)}\right|^{p-2} \left|\phi^{(N)}_{x} \right|\bar{u}_{x} d x \\
			& \leq \varepsilon \int_{B(N)}\left|\phi^{(N)}\right|^{p-2} |\phi^{(N)}_{x}|^2 d x+ C(1+t)^{-1}\varepsilon^{-1}\left\|\phi^{(N)}\right\|_{L^p_N}^{p} +p^{-1} C(1+t)^{-\frac p2}.
		\end{aligned}
	\end{equation}
Substituting \eqref{real I(p)} and \eqref{b_N} into \eqref{3.21n}, we obtain
	\begin{equation}
		\begin{aligned}\label{9.7}
			&\frac{1}{p} d\left\|\phi^{(N)}\right\|_{L^p_N}^{p}+\frac{p-1}{p}\alpha \int_{B(N)}\left|\phi^{(N)}\right|^{p} \bar{u}_{x} d x d t\\
   &+(p-1)\left(\mu -  2\varepsilon\right) \int_{B(N)}\left|\phi^{(N)}\right|^{p-2} |\phi^{(N)}_{x}|^2 d x d t \\
		 \leq&   C {\varepsilon}^{-1}(1+t)^{-1}\left\|\phi^{(N)}\right\|_{L^p_N}^{p}\dif t  +C\varepsilon^{-1}(1+t)^{\frac{-p}{2}}\dif t  + d M_p^{(N)}(t).
		\end{aligned}
	\end{equation}
Let $0<\varepsilon<\frac \mu2.$ Integrating \eqref{9.7} over $[0,t]$, we obtain \eqref{8.17}. Then, taking the expectation and applying the Gr\"{o}nwall inequality, we conclude that 
	\begin{equation}
		\begin{aligned}
			\mathbb E \left\|\phi^{(N)}(t)\right\|_{L^{p} _N}^{p} &\leq 
 C\varepsilon^{-1} e^{C \varepsilon^{-1}ln(1+t)}\left\|\phi^{(N)}_0\right\|_{L^{p} _N}^{p}\ln (1+t) \\
&=
 C\varepsilon^{-1} (1+t)^{C \varepsilon^{-1}}\left\|\phi^{(N)}_0\right\|_{L^{p} _N}^{p}\ln (1+t).
		\end{aligned}
	\end{equation}
 Fixing $n<p$, we get
	\begin{equation*}
 \begin{aligned}
		&\mathbb{E}\left\|\phi^{(N)}(t)\right\|_{L^p _N}^{n}\leq\left(\mathbb{E}\left\|\phi^{(N)}\right\|_{L^p _N}^{p}\right)^{\frac{n}{p}}\leq\left[ C\varepsilon^{-1} (1+t)^{C \varepsilon^{-1}}\left\|\phi^{(N)}_0\right\|_{L^{p} _N}^{p}\ln (1+t)\right]^{\frac{n}{p}}.
  \end{aligned}
	\end{equation*}
	Let $p\rightarrow +\infty$. Then it follows from Fatou's Lemma that
	\begin{equation*}
		\mathbb E  \left\|\phi^{(N)} (t)\right\|_{L^\infty _N}^{n} \leq \left\| \phi_{0}^{(N)} \right\|_{L^\infty _N}^{n}\leq \left\| \phi_{0} \right\|_{L^\infty (\R)}^{n}.
	\end{equation*}
  This implies
	\begin{equation}\label{3.28mp}
		\left(\mathbb E \left\| \phi^{(N)}(t)  \right\|_{L^\infty _N}^{n}\right)^{\frac{1}{n}} \leq  \left\| \phi_{0}\right \|_{\infty } \quad \text { for all }n.
	\end{equation}
	Letting $n\rightarrow +\infty$ in \eqref{3.28mp}, we establish the maximum principle
	\begin{equation}
		\left\|\phi^{(N)}(t)\right\|_{L^\infty _N} \leq  \| \phi_{0} \|_{\infty}, \quad \text { {\rm{a.s.}}}.
	\end{equation}
 This completes the proof. 
	\end{proof}

\subsection{The Galerkin Approximation}
In this section, we establish higher regularity for the solution  by the  Galerkin approximation $\tilde\phi^{(N)}$.

Take $\{e_k\}_{k=1}^\infty$ to an orthonormal basis of $L^2_N$ and $\rm{\span}\{e_1,\cdots,e_k,\cdots\}$ is dense in $ H^{2}_N$. Define $P_k^N:H^{-2}_N\rightarrow H^{2}_N$ by 
$$P_k^N u\coloneqq\sum_{j=1}^{k}\langle u, e_j \rangle e_j,\quad k\geq1.$$

Let $R > ||\phi_0||_\infty \,+\, (u_+ - u_-)$ be the threshold \eqref{small data treshhold} from  Lemma \ref{lm2.4}. Recall $\theta_R $ and the modified coefficients in \eqref{eq:cutoff_coeffs}. A key property is that the derivatives $\tilde{b}_N'(u)$ and $\tilde{\sigma}_{i,N}''(u)$ are compactly supported. They vanish identically when $|u| \ge 2R$.

For $k\geq 1$, let $H_k^N\coloneqq\Ima \left(P_k^N\big|_{L^2_N}\right)$ and consider the following SDEs in $H_N^2$
\bae\label{galerkin}
\begin{cases}
d\tilde{\phi}_k^{(N)} = \tilde{A}_{k,N} (\tilde{\phi}_k^{(N)}) dt + \tilde{B}_{k,N} (\tilde{\phi}_k^{(N)}) dt + \tilde{\Sigma}_{k,N} (\tilde{\phi}_k^{(N)}) dQ_N(t), \\
\tilde{\phi}_k^{(N)}(0) = P_k^N \phi_0^{(N)}.
\end{cases} 
\eae
For $v \in H_{k}$,
\baee	
\tilde{A}_{k,N}(v) :=& P_k^N \partial_x \left[ \Big(\tilde{b}_N(v+\bar{u}) + \frac{1}{2}\sum \tilde{\sigma}_{i,N}'^2(v+\bar{u})\Big) (v_x + \bar{u}_x) \right] ,\\
\tilde{B}_{k,N}(v) :=& - P_k^N \partial_x \left[ f_N(v+\bar u) - f_N(\bar u) \right],\\
\tilde{\Sigma}_{k,N}(v) dQ_N(t) :=& \sum_{i=1}^{N} P_k^N \partial_x \left[ \tilde{\sigma}_{i,N}(v+\bar u) \right] dB_i(t),
\eaee
see \eqref{eq:cutoff_coeffs} for the detailed definitions of the coefficients. Since the coefficients $\tilde{b}_N$ and $\tilde{\sigma}_{i,N}$ involve the smooth cut-off $\theta_R$, the coefficients of the system of SDEs \eqref{galerkin} are globally Lipschitz continuous in the finite-dimensional space $H_k^N$. Consequently, for each fixed $k$ and $N$, there exists a unique global strong solution $\tilde{\phi}_k^{(N)}$.
\begin{Lemma}\label{lm1.2}
	Fix $T>0$. Let $\phi_k^{(N)}$ be the solution of the SDEs \eqref{galerkin}. Then we have $$\mathbb{E}\left(\| \tilde{\phi}_{k}^{(N)}(T)\|^2 + \int_{0}^{T}\|\partial_{x} \tilde{\phi}_{k}^{(N)}\|^2 dt\right) \le \|\phi_{0}\|^2 + C\ln(1+T) + C,$$
    where $C$ is a constant independent of the Galerkin dimension $k$ and the cut-off parameter $N$.
\end{Lemma}
\begin{proof}
	It follows similarly by taking the projection $P_k^N$ in the $L^2$ estimates in Lemma \ref{Lemma4.3}.
\end{proof}
%
\begin{Lemma}[Derivative estimate]\label{truncated-derivative-estimate}
	Fix $T>0$. Let $\tilde\phi_k^{(N)}$ be the solution of the SDEs \eqref{galerkin}. Assume that the cut-off width $R$ satisfies $R > \mathfrak{M}$ and the stability condition holds for this $R$, i.e., $\|\phi_0\|_\infty < C^*(R)$, as in Lemma {\rm{\ref{lm2.4}}}. Then  we have
    \baee\label{induc galerkin}
    \mathbb{E}\left(\|\partial_x \tilde{\phi}_{k}^{(N)}(T)\|^2 + \int_{0}^{T}\|\partial_{xx} \tilde{\phi}_{k}^{(N)}\|^2 dt\right) \le \|\phi_{0}\|_{H^{1}}^2 + C \ln(1+T) + C.
    \eaee
\end{Lemma}
\begin{proof}
We apply Itô's formula to  $\frac{1}{2}||\partial_x v||^2$ for the process $\tilde{\phi}_k^{(N)}$. This is equivalent to testing the equation (4.23) against the test function $-\partial_{xx}\tilde{\phi}_k^{(N)} \in H_k^N$. We obtain the energy identity
\begin{equation*}\label{eq:energy_ident}
    \frac{1}{2}d||\partial_{x}\tilde\phi_{k}^{(N)}||^{2} + \int_{B(N)}\tilde  b_N(\tilde\phi_{k}^{(N)}+\bar u) (\partial_{xx}\tilde\phi_{k}^{(N)})^2 dx dt = \sum_{j=1}^5 J^j dt + d\mathcal  M^{(N)}_k(t),
\end{equation*}
where the terms $J^1$ through $J^5$ correspond to the It\^o correction and the projections of the drift terms
\baee
    J^1 &:= \frac{1}{2}\sum_{i=1}^{N}||\tilde \sigma_{i,N}''(\partial_{x}\tilde\phi_{k}^{(N)}+\bar{u}_{x})^2||^{2}, \\
    J^2 &:= \int_{B(N)}\left(\frac{1}{2}\sum_{i=1}^{N}\tilde \sigma_{i,N}'^2-\tilde b_{N}\right)\partial_{xx}\tilde\phi_{k}^{(N)}\bar{u}_{xx}dx, \\
    J^3 &:= \int_{B(N)}\left(\sum_{i=1}^{N}\tilde \sigma_{i,N}'\tilde \sigma_{i,N}''\bar{u}_{xx}-\tilde b_{N}'\partial_{xx}\tilde\phi_{k}^{(N)}\right)(\partial_{x}\tilde\phi_{k}^{(N)}+\bar{u}_{x})^2dx, \\
    J^4 &:= \frac{1}{2}\sum_{i=1}^{N}\int_{B(N)}\tilde \sigma_{i,N}'^2\bar{u}_{xx}^{2}dx, \\
    J^5 &:= -\langle B_{N}(\tilde \phi_{k}^{(N)}),\partial_{xx}\tilde\phi_{k}^{(N)}\rangle.
\eaee
Here we use the simplified notations $\tilde b_N = \tilde b_N(\tilde\phi^{(N)}_k+\bar{u})$ and $\tilde\sigma_{i,N} =\tilde \sigma_{i,N}(\tilde\phi^{(N)}_k+\bar{u})$. We estimate these terms individually.

\textbf{Estimate of $J^4$.}
\bae
J^4 = \frac{1}{2}\sum_{i=1}^{N}\int_{B(N)}\tilde \sigma_{i,N}'^2\bar{u}_{xx}^{2}dx\leq C \|\tilde{\sigma}_N'\|_\infty^2 (1+t)^{-2} 
\eae
by Proposition \ref{lm2.1}.

\textbf{Estimate of $J^2$ (Viscosity Perturbation).}
By Young's inequality and the uniform bounds on coefficients
\[
    J^2 \le \epsilon \int_{B(N)} \tilde b_N ( \partial_{xx}\tilde\phi_{k}^{(N)})^2 dx + C||b||_{L^\infty}||\bar{u}_{xx}||^2 \le \epsilon \|\partial_{xx}\tilde\phi_{k}^{(N)}\|^2 + C(1+t)^{-2}.
\]

\textbf{Estimate of $J^3$.}
We expand the product
\baee
    J^3 &= -\int_{B(N)}\tilde b_N' \partial_{xx}\tilde\phi_{k}^{(N)} (\partial_{x}\tilde\phi_{k}^{(N)} + \bar{u}_x)^2 dx + \mathcal{R}_1(t),
\eaee
where $\mathcal{R}_1(t)$ collects lower-order terms involving $\bar{u}_{xx}$.
The most singular term is $$J^3_{\rm{crit}} := \int \tilde{b}_N'(\tilde{\phi}_k^{(N)}) (\partial_x \tilde{\phi}_k^{(N)})^2 \partial_{xx} \tilde{\phi}_k^{(N)} dx.$$
Recall that $\tilde{b}_N(u) = b_N(\theta_R(u))$. Its derivative $\eta(u) := \tilde{b}_N'(u)$ vanishes for $|u| > 2R$. Thus, the integrand is supported strictly on the set where $|\tilde{\phi}_k^{(N)}| \le 2R$. We apply the localized Gagliardo-Nirenberg integration by parts
\begin{equation*}
    \begin{aligned}
        &\int \eta(\tilde{\phi}_k^{(N)}) (\partial_x \tilde{\phi}_k^{(N)})^4 dx = - \int \tilde{\phi}_k^{(N)} \partial_x \left[ \eta(\tilde{\phi}_k^{(N)}) (\partial_x \tilde{\phi}_k^{(N)})^3 \right] dx\\
=& - \int \tilde{\phi}_k^{(N)} \big[ \eta'(\tilde{\phi}_k^{(N)}) (\partial_x \tilde{\phi}_k^{(N)})^4 + 3\eta(\tilde{\phi}_k^{(N)}) (\partial_x \tilde{\phi}_k^{(N)})^2 \partial_{xx} \tilde{\phi}_k^{(N)} \big] dx.
    \end{aligned}
\end{equation*}
Then, rearranging and estimating yields $$\left| \int \eta(\tilde{\phi}_k^{(N)}) (\partial_x \tilde{\phi}_k^{(N)})^4 dx \right| \le C \sup_{x \in \text{supp}(\eta)} |\tilde{\phi}_k^{(N)}(x)| \cdot \int |\partial_x \tilde{\phi}_k^{(N)}|^2 |\partial_{xx} \tilde{\phi}_k^{(N)}| dx.$$
Since the support condition ensures $|\tilde{\phi}_k^{(N)}| \le 2R$ wherever the integral is non-zero, it holds that
\baee
\int\eta(\tilde{\phi}_k^{(N)})   (\partial_x \tilde{\phi}_k^{(N)})^4 dx \le C R^2 \|\partial_{xx} \tilde{\phi}_k^{(N)}\|^2.
\eaee
Substituting this back into the estimate for $J^3_{\rm{crit}}$
\baee
|J^3| \le& \frac{\mu}{4} \|\partial_{xx} \tilde{\phi}_k^{(N)}\|^2 + C \frac{\|\tilde{b}_N'\|_\infty^2}{\mu} \int_{\{|\tilde{\phi}_k^{(N)}| \le 2R\}} (\partial_x \tilde{\phi}_k^{(N)})^4 dx + \mathcal{R}_1(t)\\
\le&  \big( \frac{\mu}{4} + C \frac{R^2}{\mu} \big) \|\partial_{xx} \tilde{\phi}_k^{(N)}\|^2 + C(1+t)^{-2}.
\eaee
\textbf{Estimate of $J^1$ (Noise Correction).}
Similar to $J^3$ $$J^1 = \frac{1}{2} \sum_{i=1}^N \left\| \tilde{\sigma}_{i,N}''(\tilde{\phi}_{k}^{(N)}+\bar u) (\partial_x \tilde{\phi}_{k}^{(N)}+\bar u_x)(\partial_x \tilde{\phi}_{k}^{(N)}+\bar u_x) \right\|^2.$$
 Since $\tilde{\sigma}''(u)$ is supported on $|u| \le 2R$, directly applying the Localized Gagliardo-Nirenberg inequality, we get that 
\baee
J^1_{\rm{crit}} \le&\, C \|\tilde{\sigma}_N''\|_\infty^2 \int_{|\tilde{\phi}_{k}^{(N)}+\bar u|\le 2R} \left( (\partial_x \tilde{\phi}_{k}^{(N)})^4 + \bar u_x^4\right) dx\\
\le& \, C R^2 \|\tilde{\sigma}_N''\|_\infty^2  \|\partial_{xx} \tilde{\phi}_{k}^{(N)}\|^2 +C(1+t)^{-2}.
\eaee
This term does not carry a $\mu^{-1}$ factor as it does not originate from a cross-term with the viscous dissipation.

\textbf{Estimate of $J^5$ (Flux Remainder).}
Using integration by parts and the convexity of $f_N$, we decompose $J^5 = J^5_1 + J^5_2 + J^5_3:$
\baee
    J^5 &= -\int_{B(N)} \frac{1}{2} f_N''(\bar{u})\bar{u}_x (\partial_{x}\tilde \phi_{k}^{(N)})^2 dx  + \int_{B(N)} [f_N'(\bar{u}+\tilde \phi_k^{(N)}) - f_N'(\bar{u})]\bar{u}_x  \partial_{xx}\tilde\phi_{k}^{(N)} dx \\
        &\quad + \int_{B(N)} [f_N'(\bar{u}+\tilde \phi_k^{(N)}) - f_N'(\bar{u})]\partial_{x}\tilde \phi_{k}^{(N)} \partial_{xx}\tilde\phi_{k}^{(N)} dx = J^5_1 + J^5_2 + J^5_3.
\eaee
Since $f''_N \ge \alpha > 0$ and $\bar{u}_x > 0$, the first term is negative and aids stability (or is dropped) as in the proof of Lemma \ref{lm2.4}. The remaining terms are bounded by
\[
    J^5 \le \epsilon || \partial_{xx}\tilde\phi_{k}^{(N)}||^2 + C(1+t)^{-2}||\tilde \phi_k^{(N)}||^2 + C||\partial_{x}\tilde \phi_{k}^{(N)}||^2.
\]
By Lemma \ref{lm1.2}, the term $||\tilde{\phi}_k^{(N)}||^2$ is uniformly bounded, allowing us to absorb it into the remainder.

Summing all estimates and substituting back into the energy identity
$$d\|\partial_x \tilde{\phi}_{k}^{(N)}\|^2 + \Big( \mu - \frac{\mu}{4} - C R^2 \big( {\mu}^{-1}{\|\tilde{b}_N'\|_\infty^2} + \|\tilde{\sigma}_N''\|_\infty^2 \big) \Big) \|\partial_{xx} \tilde{\phi}_{k}^{(N)}\|^2 dt \le \mathcal{R}_1(t) dt + d\mathcal  M^{(N)}_k(t).$$
Let $c_{\rm{diss}} := \frac{3\mu}{4} - C R^2 (\mu^{-1}\|\tilde{b}_N'\|_\infty^2 + \|\tilde{\sigma}_N''\|_\infty^2)$. By the small data assumption \eqref{small data treshhold}, we ensure $c_{\rm{diss}} > 0$. Integrating over $[0, T]$, taking expectations, and applying Gr\"{o}nwall's inequality, it holds that
$$\mathbb{E}\|\partial_x \tilde{\phi}_{k}^{(N)}(T)\|^2 + c_{\rm{diss}} \mathbb{E}\int_0^T \|\partial_{xx} \tilde{\phi}_{k}^{(N)}\|^2 dt \le \|\phi_0\|_{H^1}^2 + C \int_0^T \mathbb{E}\|\partial_x \tilde{\phi}_{k}^{(N)}\|^2 dt + C \ln(1+T).$$ 
This completes the proof.
\end{proof}

\subsection{Convergence and existence of solution}
In this subsection, we address the convergence of the Galerkin approximations $\tilde{\phi}_k^{(N)}$ to the solution of the cut-off equation \eqref{eq:cutoff_coeffs} as $k \to \infty$. Since the coefficients $\tilde{b}_N$ and $\tilde{\sigma}_{i,N}$ depend on the solution, the operator is quasilinear. To pass the limit in the term $\tilde{b}_N(\tilde{\phi}_k^{(N)})\partial_x \tilde{\phi}_k^{(N)}$, we require strong convergence in $L^2(0,T; L^2_N)$ or $L^2(0,T; H^1_N)$.

We employ the compactness method via the Aubin-Lions-Simon theorem to establish tightness, and the Gy\"ongy-Krylov framework \cite{gyongy1996existence} to establish convergence in probability.

\begin{Lemma}[Time Regularity and Tightness] \label{lem:tightness_k}
Let $\tilde{\phi}_k^{(N)}$ be the solution to the Galerkin system \eqref{galerkin}. For fixed $N$, the laws of the sequence $\{\tilde{\phi}_k^{(N)}\}_{k \geq 1}$ are tight in the space $\mathcal{X}_N := L^2(0, T; H_N^1)$.
\end{Lemma}

\begin{proof}
By Lemma 4.3, we have the following uniform bound
\begin{equation*} \label{eq:uniform_bound_k}
\sup_{k \in \mathbb{N}} \mathbb{E} \left[ \sup_{t \in [0,T]} \|\tilde{\phi}_k^{(N)}(t)\|_{H^1}^2 + \int_0^T \|\tilde{\phi}_k^{(N)}(t)\|_{H^2}^2 dt \right] \le C_N.
\end{equation*}
To apply the compactness theorem, we need estimates on the time derivative. The equation \eqref{galerkin} can be written in integral form:
\begin{equation*}
    \tilde{\phi}_k^{(N)}(t) = P_k^N \phi_0^{(N)} + \int_0^t \big[\tilde{A}_{k,N}(\tilde{\phi}_k^{(N)}(s)) + \tilde{B}_{k,N}(\tilde{\phi}_k^{(N)}(s))\big] ds + \int_0^t \tilde{\Sigma}_{k,N}(\tilde{\phi}_k^{(N)}(s)) dQ_N(s).
\end{equation*}
Since $\tilde{\phi}_k^{(N)}$ is uniformly bounded in $L^2\big(\Omega; L^2(0,T; H^2)\big)$ (by Lemmas \ref{lm1.2} and \ref{truncated-derivative-estimate}), and the coefficients are bounded due to the smooth cut-off $\theta_R$, it follows that the drift term is uniformly bounded in $L^2(\Omega; L^2(0,T; L^2_N))$. The diffusion term is a square-integrable martingale in $H^1$. 
By standard fractional Sobolev estimates for the It\^o processes (see, e.g., Flandoli and Gatarek \cite{flandoli1995martingale}), for any $\alpha \in (0, 1/2)$, there exists a constant $C_N$ such that
\begin{equation*}
    \mathbb{E} \| \tilde{\phi}_k^{(N)} \|_{W^{\alpha, 2}(0, T; L^2_N)}^2 \le C_N.
\end{equation*}
We invoke the Aubin-Lions-Simon compact embedding \cite{aubin1963theoreme,lions1969quelques,simon1986compact}
$$
L^2(0, T; H^2_N) \cap W^{\alpha, 2}(0, T; L^2_N) \overset{c}{\hookrightarrow} L^2(0, T; H_N^1).
$$
Combining this with Chebyshev's inequality, we conclude that the laws of $\{\tilde{\phi}_k^{(N)}\}_{k \geq 1}$ are tight in $\mathcal{X}_N$.
\end{proof}

\begin{prp}[Convergence of the Galerkin approximations] \label{prop:conv_galerkin}
For fixed $N$, as $k \to \infty$, the sequence $\tilde{\phi}_k^{(N)}$ converges in probability in $L^2(0, T; H_N^1)$ to a unique strong solution $\phi^{(N)}$ of the cut-off equation \eqref{eq:cutoff_coeffs}.
\end{prp}

\begin{proof}
We define the path space $\mathcal{Z}_N = \mathcal{X}_N \times \mathcal{X}_N$. Let $\mu_k$ be the law of $\tilde{\phi}_k^{(N)}$. Consider any pair of indices $(k, m)$. By the tightness established in Lemma \ref{lem:tightness_k}, the collection of joint laws $\mathcal{L}(\tilde{\phi}_k^{(N)}, \tilde{\phi}_m^{(N)})$ is  tight in $\mathcal{Z}_N$.

Let $\big\{(\tilde{\phi}_{k_j}^{(N)}, \tilde{\phi}_{m_j}^{(N)})\big\}$ be a subsequence converging in law to a measure $\nu$ on $\mathcal{Z}_N$. By Skorokhod's representation theorem, there exists a probability space $(\tilde{\Omega}, \tilde{\mathcal{F}}, \tilde{\mathbb{P}})$ and random variables $(u^*, v^*)$ such that the subsequence converges $\tilde{\mathbb{P}}$-almost surely to $(u^*, v^*)$ in the topology of $\mathcal{X}_N \times \mathcal{X}_N$. Due to the almost sure strong convergence in $L^2(0,T; H^1_N)$, we can pass to the limit in the quasilinear terms. Specifically,
$$
\tilde{b}_N(u^*) \partial_x u^* = \lim_{j \to \infty} \tilde{b}_N\big(\tilde{\phi}_{k_j}^{(N)}\big) \partial_x \tilde{\phi}_{k_j}^{(N)},
$$
where the convergence holds in $L^1\big((0,T) \times (-N,N)\big)$.

Regarding the stochastic integral, recall that the cut-off equation \eqref{d1.2} involves a finite sum of $N$ noise terms. However, to identify the limit as a solution to the target problem, we must ensure the convergence of the full stochastic series. Under Hypothesis (H3), the tail of the noise coefficients satisfies $\sum_{i=1}^\infty \|\sigma'_i\|_{H^2_N}^2 < \infty$. Standard martingale convergence arguments ensure that the stochastic integrals converge uniformly in $L^2\big(\tilde{\Omega}; C([0,T]; L^2_N)\big)$. This allows us to identify that both $u^*$ and $v^*$ as martingale solutions to the cut-off equation \eqref{eq:cutoff_coeffs} with the same initial data.

Although the drift is quasilinear, the solutions $u^*$ and $v^*$ inherit the uniform bound from Lemma \ref{truncated-derivative-estimate}, specifically $u^*, v^* \in L^2(0,T; H^2_N)$. By the Sobolev embedding, this implies $u^*, v^* \in L^2\big(0,T; W^{1,\infty}(-N,N)\big)$. This control on the spatial gradient $\partial_x u^*$ is sufficient to close the $L^2$ energy estimate for the difference $u^*-v^*$, thereby establishing pathwise uniqueness. Thus, $u^*$ and $v^*$ must satisfy $u^* = v^*$ $\tilde{\mathbb{P}}$-almost surely. Consequently, the limit measure $\nu$ is supported on the diagonal $\{(u, v) \in \mathcal{Z}_N : u = v\}$.

By the Gy\"ongy-Krylov Lemma \cite{gyongy1996existence}, this implies that the original sequence $\tilde{\phi}_k^{(N)}$ converges in probability in the topology of $\mathcal{X}_N$ to a process $\phi^{(N)}$. Since the approximations $\tilde{\phi}_k^{(N)}$ are adapted to the original filtration, the limit $\phi^{(N)}$ is a strong solution.
\end{proof}
Finally, we present the proof of Theorem \ref{th2.2}.
\begin{proof}[\textbf{Proof of Theorem \ref{th2.2}}]

\

The proof is devided into four steps.

\textbf{Step 1. Uniform Estimates.}
From the a priori estimates in Section 2 (specifically Lemmas 2.1, 2.3, and 2.9) and the corresponding uniform bounds on the Galerkin approximations (Lemmas \ref{Lemma4.3}, \ref{lm1.2}, and \ref{truncated-derivative-estimate}), which were derived independently of the domain size, we have the uniform bound:
\begin{equation} \label{eq:uniform_bound_N}
\sup_{N \ge 1} \mathbb{E} \left[ \sup_{t \in [0,T]} \|\phi^{(N)}(t)\|_{H^1}^2 + \int_0^T \|\phi^{(N)}(t)\|_{H^2}^2 dt \right] \le C.
\end{equation}
Furthermore, the cut-off maximum principle (Lemma \ref{Lemma4.3}) ensures $\|\phi^{(N)}\|_{L^\infty((0,T)\times \mathbb{R})} \le \|\phi_0\|_{\infty}$ almost surely.

\textbf{Step 2. Local Compactness.}
Let $K \subset \mathbb{R}$ be an arbitrary compact subset. Using the same fractional time regularity argument from Lemma \ref{lem:tightness_k} ($W^{\alpha, 2}$-regularity)  with the uniform $H^2$ bound \eqref{eq:uniform_bound_N},  the Aubin-Lions-Simon theorem implies that the sequence of laws of $\{\phi^{(N)}\}_{N\ge 1}$ restricted to $K$ is tight in $L^2(0,T; H^1(K))$. By a diagonalization procedure over a sequence of compact sets $K_m = [-m, m]$, we conclude that the sequence of laws is tight in the Fréchet space $\mathcal{Y} := L_{\rm{{loc}}}^2([0,T]\times \mathbb{R})$.

\textbf{Step 3. Convergence via the Gy\"ongy-Krylov Theory.}
Consider two sequences $\phi^{(N)}$ and $\phi^{(M)}$ with $N, M \to \infty$. Let $(\phi^*, \psi^*)$ be limits of any subsequence in law defined on a new probability space via Skorokhod representation theorem.
Thanks to the strong local convergence provided by the compactness in Step 2, we have that
$$
b(\phi^{(N)} + \bar u) \partial_x \phi^{(N)} \rightarrow b(\phi^* + \bar u) \partial_x \phi^* 
$$
weakly in $L^2(\Omega \times (0,T); L^2(\mathbb{R}))$ and strongly in $\rm{L_{\rm{{loc}}}^{1}}$. Additionally, the condition (H3) ensures the convergence of the infinite stochastic series in $L^2(\Omega; L^2(0,T; L^2_{\rm{loc}}))$. This allows us to identify both $u^*$ and $v^*$ as martingale solutions to  the modified equation (with globally Lipschitz coefficients $\tilde{b}, \tilde{\sigma}$) on $\mathbb{R}$.

Crucially, Theorem 1.3  establishes the $\rm{L_{loc}^{1}}$ contraction principle, which implies pathwise uniqueness for the target equation. Therefore, $\phi^* = \psi^*$ almost surely.
By the Gy\"ongy-Krylov characterization, the entire sequence $\phi^{(N)}$ converges in probability in $L^2_{\rm{{loc}}}\big((0, T); H^1_{\rm{{loc}}}(\mathbb{R})\big)$ to a process $\phi$.

\textbf{Step 4. Identification of the Limit.}
The limit $\phi$ inherits the uniform bounds from Step 1. Thus $\phi \in L^2\Big(\Omega; L^\infty\big(0,T; H^1(\mathbb{R})\big)\Big) \cap L^2\Big(\Omega; L^2\big(0,T; H^2(\mathbb{R})\big)\Big)$.
The convergence in probability allows us to pass to the limit in the stochastic integral term and the quasilinear drift term. Thus $\phi$ is a strong solution to \eqref{perturbation}.

Finally, we justify the removal of the cut-off. By the hypothesis of Theorem \ref{th2.2}, the initial perturbation is sufficiently small. Specifically, we fix the cut-off parameter $R := 2\mathfrak{M}$, where $\mathfrak{M}$ is the bound on the initial data defined in Theorem \ref{th2.2}. The smallness assumption in Theorem \ref{th2.2} translates to the algebraic condition $\mathfrak{M} < C^*(2\mathfrak{M})$. Under this condition, the Maximum Principle \eqref{mp} ensures $\|\phi(t)\|_\infty \le \mathfrak{M} < R$ almost surely, the solution never exits the region $[-R, R]$. Consequently, $\theta_R(\phi+\bar{u}) \equiv \phi+\bar{u}$ for all $t \ge 0$, and the modified coefficients coincide with the original physical coefficients. This confirms that the constructed limit $\phi$ is indeed the global strong solution to the original problem. Recalling the decay estimates from Lemmas \ref{td} and \ref{t2}, the proof is complete.
\end{proof}

\section{Appendix}
\subsection{Supplements  for Decay Estimates}\label{appen area}

\begin{proof}[\textbf{Proof of Lemma {\rm{\ref{dfn area}}}}]
We proceed the proof by contradiction. Set $\rho^*= \frac{\beta}{\alpha}\wedge \frac{1}{\alpha-1}$ and  assume that 
\bae\label{det contra}
X(t)\neq O\big(t^{-\rho^* }\big).\eae
Let $r:=\frac{(\alpha-1)\beta}{\alpha}.$ Then  we divide the decay proof into the following  two cases.

\textbf{Case 1. $r<1.$}
Define $$Y_n(t):= n t^{-\beta/\alpha},t\geq1.$$ 
By \eqref{det contra}, we know that the graph of $X_t$ intersects the family of graphs of $Y_n$ infinitely many times as $t\rightarrow \infty.$ 
 Then we have sequences $t_{n-1}\leq t_n$  increasing to infinity, such that 
 $$X_{t_n} = n t_n^{-\beta / \alpha}.$$
For $n$ large enough, we construct $\tilde X_t$ equals to $X_t$ at $t_n$ and trace $\tilde X_t$ backward from $t_n$ by the following way 
\baee
\frac{d}{d t} \tilde{X}_t=&2 \tilde{X}_t^\alpha, 0\leq t\leq t_n,\\
\tilde{X}_{t_n}  =& n t_n^{-\beta / \alpha}.
\eaee
For $0\leq t\leq t_n,$ one has 
\bae\label{explicit X}
\tilde{X}_t=\left(2(\alpha-1)\left(t_n-t\right)+n^{1-\alpha} t_n^{\frac{(\alpha-1) \beta}{\alpha}}\right)^{-\frac{1}{\alpha-1}}.
\eae
Using \eqref{explicit X}, we choose $0\leq s_n\leq t_n$ such that $\tilde{X}_{s_n}  =s_n^{-\beta / \alpha}$, which yields
\baee
2(\alpha-1)\left(t_n-s_n\right)+n^{1-\alpha} t_n^{\frac{(\alpha-1) \beta}{\alpha}}=\tilde{X}_{s_n}^{1-\alpha}=s_n{ }^{ \frac{(\alpha-1)\beta}{\alpha}}.
\eaee
By the comparison principle of ODEs, one has 
\bae\label{ode comp1}
X_t\geq \tilde X_t, \quad \forall t\in[s_n,t_n],
\eae
\begin{figure}
    \centering
    \includegraphics[width=0.4\linewidth]{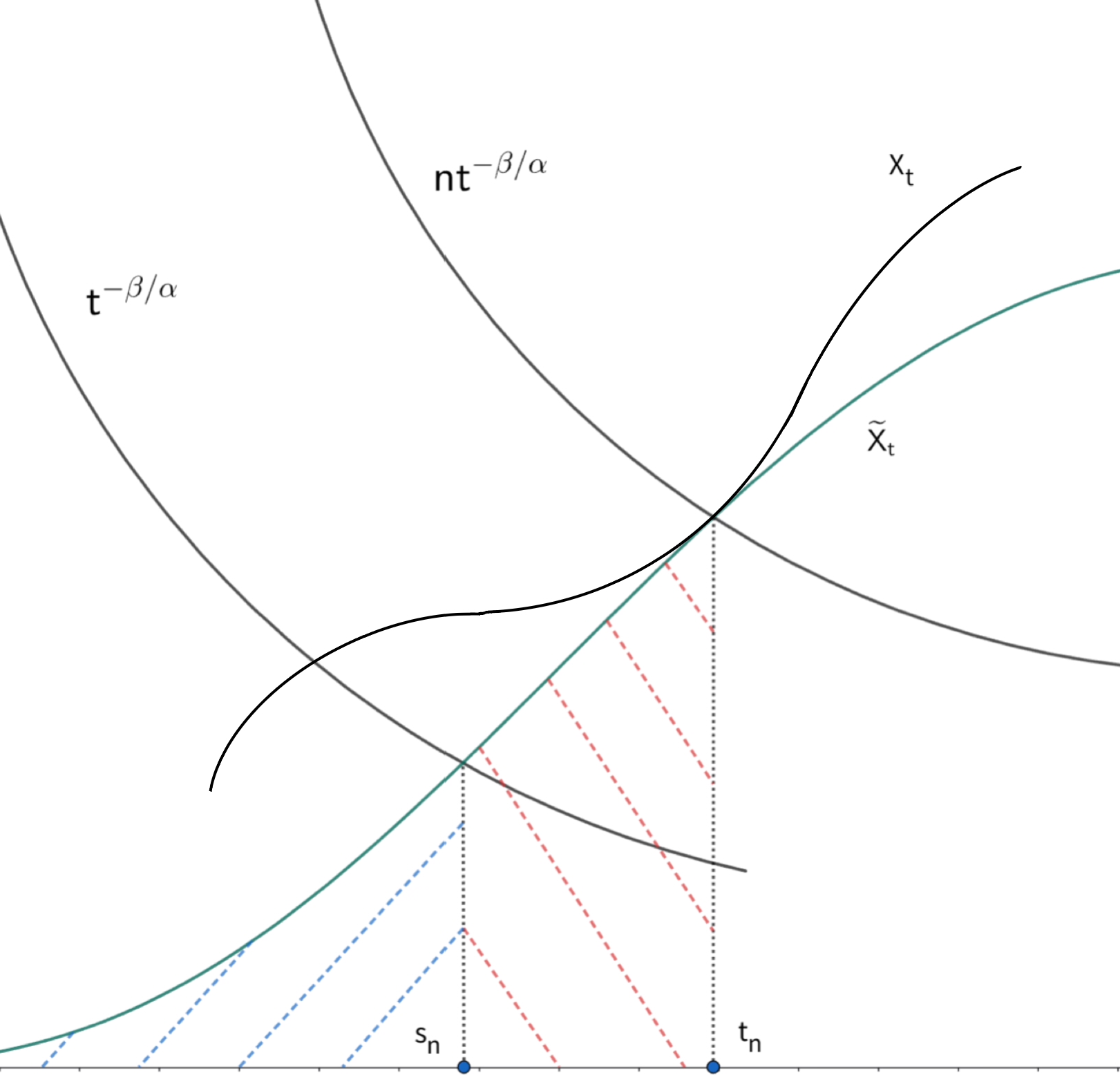}
    \caption{$X_t\geq \tilde X_t $}
    \label{fig:placeholder}
\end{figure}
and
\bae\label{sn formula}
& s_n^r=2(\alpha-1)\left(t_n-s_n\right)+n^{1-\alpha} t_n^r.
\eae
Dividing both sides of \eqref{sn formula} by $t_n$, we obtain
$$\lim_{n\rightarrow \infty} \frac{s_n}{t_n}=1.$$ 
Then we conclude that
\baee
\lim_{n\rightarrow \infty} \frac{t_n-s_n}{t_n^r}=\frac{1}{2(\alpha-1)}\lim_{n\rightarrow \infty}\big[ (s_n/t_n)^r-n^{1-\alpha} \big]=\frac{1}{2(\alpha-1)}.
\eaee
For $n$ large enough, one has
\baee\label{r<1}
 s_n\leq t_n-(2\alpha)^{-1}t_n^r=t_n-(2\alpha)^{-1}t_n^{\frac{\alpha-1}{\alpha}\beta}.
\eaee
By \eqref{ode comp1}, we compute that

\ 

\baee
& \int_{s_n}^{t_n} {X}_\tau d \tau\geq \int_{s_n}^{t_n} \tilde{X}_\tau d \tau \\
& =\int_{t_n-(2\alpha)^{-1} t_n^{\frac{(\alpha-1)\beta}{\alpha}}}^{t_n}\left(2(\alpha-1)\left(t_n-t\right)+n^{1-\alpha} t_n^{\frac{(\alpha-1)\beta}{\alpha}}\right)^{\frac{1}{1-\alpha}} d t \\
& =\frac{1}{2\alpha (\alpha-2)} t_n^{\frac{\alpha-2}{\alpha}\beta} \left[\alpha^{\frac{1}{\alpha-1}}\left(\alpha-1+\alpha n^{1-\alpha}\right) -\alpha n^{2-\alpha}\right] \\
& \geq c_\alpha t_n^{\frac{\alpha-2}{\alpha} \beta}.
\eaee
Also note that, in this case, one has $\frac12 t_n\leq s_n$. Thus
\baee\label{for r>1}
 \int_{\frac12 t_n}^{t_n} \tilde{X}_\tau d \tau \geq  \int_{s_n}^{t_n} \tilde{X}_\tau d \tau\geq c_\alpha t_n^{\frac{\alpha-2}{\alpha} \beta}.
\eaee

\textbf{Case 2: $r\geq1.$}

Recall from Case 1 that the desired estimates hold provided that the ratio $r$ satisfies the condition
\begin{equation*} \label{cond:small_perturbation}
    |r - 1| < \delta
\end{equation*}
for some small constant $\delta > 0$. 

We now observe that the choice of the weight parameter $\beta$ in the previous construction allows for flexibility.  We can choose $\tilde{\beta}<\beta$ such that $|r(\tilde{\beta}) - 1| < \delta$. With this choice of $\tilde{\beta}$, the system falls strictly within the regime of Case 1.

In both cases, we conclude that
\bae\label{final area defi}
\int_{s_n}^{t_n} {X}_\tau d \tau\geq c_\alpha t_n^{({\alpha-2})\rho^*}. 
\eae
However, since $p<({\alpha-2})\rho^*$, the following inequality
$$C (1+t_n)^p\geq \int_{0}^{t_n} {X}_\tau d \tau\geq \int_{s_n}^{t_n} {X}_\tau d \tau \geq c_\alpha t_n^{({\alpha-2})\rho^*}$$
leads to a contradiction for sufficiently large $n$. This completes the proof of the decay rate by contradiction. 

\textbf{Next, we demonstrate that the decay rate obtained in Lemma \ref{dfn area} is \textit{sharp}}. 

Consider the deterministic process $X_t$ satisfying the critical ODE for $t \ge 1$
\begin{equation} \label{eq:sharpness_ode}
    \dot{X}_t = -X_t^\alpha + t^{-\beta}, \quad X_1 > 0,
\end{equation}
we analyze the asymptotic behavior of $X_t$ in the three regimes dictated by the parameter $r := \frac{\alpha-1}{\alpha}\beta$.

\textbf{Case 1. $r < 1$.}
 We claim that the solution decays at the rate of the forcing term, i.e., $X_t\sim t^{-\frac{\beta}{\alpha}}$.
We seek for an asymptotic solution of the form $X_t = C t^{-\frac{\beta}{\alpha}}$. Substituting this ansatz into \eqref{eq:sharpness_ode} yields
\[
    -\frac{\beta}{\alpha} C t^{-\frac{\beta}{\alpha}-1} = -C^\alpha t^{-\beta} + t^{-\beta}.
\]
Multiplying both sides by $t^{\beta}$, we obtain
\[
    -\frac{\beta}{\alpha} C t^{\beta - \frac{\beta}{\alpha} - 1} = -C^\alpha + 1.
\]
Observe that the exponent on the left-hand side satisfies
\[
    \beta - \frac{\beta}{\alpha} - 1 = \beta\left(1 - \frac{1}{\alpha}\right) - 1 = \beta\frac{\alpha-1}{\alpha} - 1 = r - 1.
\]
Since $r < 1$, we have $t^{r-1} \to 0$ as $t \to \infty$. Thus, the differential term becomes negligible compared to the algebraic terms, leading to the asymptotic balance $-C^\alpha + 1 = 0$, which implies $C=1$. This confirms that $X_t$ cannot decay faster than $O(t^{-\frac{\beta}{\alpha}})$.

\textbf{Case 2. $r > 1$.}
 We claim that the solution is limited by the natural decay rate of the nonlinearity, i.e., $X_t \sim t^{-\frac{1}{\alpha-1}}$, regardless of the faster decay of the forcing term.
We employ the ansatz $X_t = C t^{-\frac{1}{\alpha-1}}$. Substituting this into \eqref{eq:sharpness_ode} yields
\[
    -\frac{1}{\alpha-1} C t^{-\frac{1}{\alpha-1}-1} = -C^\alpha t^{-\frac{\alpha}{\alpha-1}} + t^{-\beta}.
\]
Note that the exponents of the differential term and the nonlinear term match: $-\frac{1}{\alpha-1} - 1 = -\frac{\alpha}{\alpha-1}$. Multiplying the equation by $t^{\frac{\alpha}{\alpha-1}}$, we obtain
\[
    -\frac{1}{\alpha-1} C = -C^\alpha + t^{\frac{\alpha}{\alpha-1} - \beta}.
\]
Since $r > 1$ implies $\beta > \frac{\alpha}{\alpha-1}$, the forcing term vanishes asymptotically: $t^{\frac{\alpha}{\alpha-1} - \beta} \to 0$ as $t \to \infty$. The dynamic is dominated by the intrinsic dissipation balance
\[
    -\frac{1}{\alpha-1} C = -C^\alpha \implies C = \left(\frac{1}{\alpha-1}\right)^{\frac{1}{\alpha-1}}.
\]
This confirms that even if the forcing decays very rapidly, the solution $X_t$ cannot decay faster than $O(t^{-\frac{1}{\alpha-1}})$.

\textbf{Case 3.} $r=1$. In this critical regime, the parameter $r = \frac{\alpha-1}{\alpha}\beta = 1$ implies that the decay exponent of the forcing term satisfies $\beta = \frac{\alpha}{\alpha-1}$. We demonstrate that the solution decays at the rate $X_t \sim t^{-\frac{1}{\alpha-1}}$, where the time derivative, the nonlinearity, and the forcing term are all of the same order.

We employ the ansatz $X_t = C t^{-\frac{1}{\alpha-1}}$. Differentiating with respect to time, we have
\begin{equation}
    \dot{X}_t = -\frac{1}{\alpha-1} C t^{-\frac{1}{\alpha-1}-1} = -\frac{C}{\alpha-1} t^{-\frac{\alpha}{\alpha-1}}.
\end{equation}
Substituting this ansatz into the governing ODE \eqref{eq:sharpness_ode}
\begin{equation*}
    -\frac{C}{\alpha-1} t^{-\frac{\alpha}{\alpha-1}} = -(C t^{-\frac{1}{\alpha-1}})^\alpha + t^{-\beta}.
\end{equation*}
Since $r=1$ implies $\beta = \frac{\alpha}{\alpha-1}$, the power of $t$ is identical in all three terms. Dividing by $t^{-\frac{\alpha}{\alpha-1}}$, we obtain the algebraic equation for the constant $C$
\begin{equation*}
    C^\alpha - \frac{1}{\alpha-1}C - 1 = 0.
\end{equation*}
Let $g(C) = C^\alpha - \frac{1}{\alpha-1}C - 1$. Since $\alpha > 1$, we observe that $g(0) = -1$ and by the Intermediate Value Theorem, there exists a unique positive constant $C^*$ satisfying the equation above.

Thus, for $r=1$, the solution decays as $X_t \sim C^* t^{-\frac{1}{\alpha-1}}$. Noting that $\frac{\beta}{\alpha} = \frac{1}{\alpha-1}$ in this regime, this result is consistent with the general sharp decay estimate formula
\begin{equation}
    X_t = O\left(t^{-\left(\frac{\beta}{\alpha} \wedge \frac{1}{\alpha-1}\right)}\right).
\end{equation}
\textbf{Conclusion.}
Combining both cases, the optimal decay rate is exactly in the form
\[
    X_t \sim t^{-(\frac{\beta}{\alpha}\wedge \frac{1}{\alpha-1})}.
\]
This proves that the estimate derived in Lemma 2.4 is sharp.
\end{proof}
In the proof of Lemma \ref{stochastic area}, recall that if $\theta_j^\prime> 0$,   $\psi_j=\gamma_j-1/\theta^\prime_j-1$, $\xi=\max\{\xi_j:\theta^\prime_j=0\}$ and we set
  \bae\label{bar rho 2nd}
  \bar\rho =(1-p)\wedge \big(\beta\wedge({\theta}^{-1}-\zeta)-1\big)\wedge (-\xi)\wedge \min\Big\{\frac{\theta^\prime_j+\xi_j}{\theta^\prime_j \psi_j}:\theta^\prime_j>0\text{ and }\psi_j<0\Big\}.
  \eae
  We have the following reduction of the proper coefficients.
\begin{Lemma}\label{Coe lm}
   If coefficients in Lemma \ref{stochastic area} satisfy the following
    \begin{enumerate}
      \item[\rm{(A1) }]{ $
\xi_j < \theta^\prime_j(\bar\rho\psi_j -1), \text{ if } \theta^\prime_j>0 \text{ and }\psi_j\geq 0;\, j=1,\cdots, \mathfrak n.$}
      \item[\rm{(A2) }]{  $
\kappa<(\alpha-1)\bar\rho-1.$}
      \end{enumerate}
      Then  one can take appropriate $y$, $w\geq0$ and $z>0$, such that 
      \bae\label{ab>1} 
a_j>1,\ \forall j\in \Theta,\quad \mathfrak b>1,\quad \mathfrak c<0,
\eae
 and
 \begin{eqnarray}
-z(y+\rho+\xi)+w&>&0,\label{dissa}\\
     yz-w-1&>&0, \label{sm prob}\\
 w/\theta^\prime_j+\gamma_j+(a_j-1)z&>&1,\ \forall j\in \{j,\theta_j^\prime> 0\},\label{odta1}\\
  (\mathfrak b-1)z&>&1\label{odta2}. 
\end{eqnarray}
\end{Lemma}
\begin{proof}
   Recall  \eqref{eq:ab_def} in Lemma \ref{stochastic area},
\baee
\Theta=&\{0\}\cup\{j:\theta_j^\prime> 0\},\\
a_0 =& (\alpha-1)\rho-\kappa,\quad \alpha_0:=\alpha,\\
a_j =& \psi_j\rho-(y+\xi_j)/\theta^\prime_j,\quad \theta_j^\prime> 0,\\
\alpha_j=&w/\theta^\prime_j+\gamma_j,\quad \theta_j^\prime> 0,\\  
\mathfrak b = &\beta \wedge(1/\theta-\zeta)- \rho,\\
\mathfrak c=&y+\rho+\xi.
\eaee
First, recalling the definition of $\bar\rho$ in  \eqref{bar rho 2nd}, one has $\rho <\bar \rho \leq \beta \wedge (\frac{1}{\theta} - \zeta) - 1$.
Rearranging this inequality and using  the definition of $\mathfrak b$, we conclude that \boxed{\mathfrak b > 1}. 

To show that $\mathfrak c<0$ and $a_j>1$ for all $\theta^\prime_j>0,$ we set $y$ small enough, specifically,
\bae\label{y upper bound}
y<(-\rho-\xi) \wedge\min_{j,\theta_j^\prime> 0}\big\{\psi_j \theta^\prime_j \rho-\theta^\prime_j-\xi_j\big\}.
\eae
It is straightforward to verify that, under Assumption {\rm{(A2)}} and with $y$ chosen as in \eqref{y upper bound}, we have \boxed{a_j>1} for all $j\in \Theta$ and \boxed{\mathfrak c<0}. Thus \eqref{ab>1} is guaranteed under \eqref{bar rho 2nd} and \eqref{y upper bound}. 

It remains to establish \eqref{dissa}-\eqref{odta2}. First,  \eqref{odta2} is a direct consequence of  $z>0$ and \eqref{ab>1}. Next,  \eqref{dissa}-\eqref{odta1} are satisfied provided the following bound holds
\bae\label{w bound}
 \big(z(y+\rho+\xi)\big)\vee \max_{j,\theta_j^\prime> 0}\big\{ \theta^\prime_j \big(1-\gamma_j-(a_j-1)z\big)\big\}  <w<yz-1.
\eae
The existence of $w\geq0$ and $y,z>0$  satisfying \eqref{w bound} is equivalent to the following conditions
\begin{eqnarray}
 yz-1&>&0,\label{trans w bound0}\\
    yz-1&>&z(y+\rho+\xi),\label{trans w bound1}\\
    yz-1&>&  \theta^\prime_j \big(1-\gamma_j-(a_j-1)z\big), \quad \text{ if }\theta_j^\prime> 0.\label{trans w bound2}
\end{eqnarray}
First, recalling \eqref{bar rho 2nd} and $\xi=\max\{\xi_j,\theta^\prime_j=0\}$, we have $\rho+\xi<\bar\rho+\max\{\xi_j,\theta^\prime_j=0\}<0.$ Then,  \eqref{y upper bound}  is equivalent to 
\bae\label{exist y}
0<\psi_j \theta^\prime_j \rho-\theta^\prime_j-\xi_j,\text{ if } \theta_j^\prime> 0.
\eae
 Next, \eqref{trans w bound0} and \eqref{trans w bound1} can be rewritten as
\bae\label{z lower bound}
z>\frac{-1}{\rho+\xi}\vee \frac{1}{y}.
\eae
Furthermore, substituting $a_j=\psi_j\rho-(y+\xi_j)/\theta^\prime_j$, \eqref{trans w bound2} becomes
\bae\label{z bound cases}
(\rho\theta^\prime_j \psi_j-\xi_j-\theta^\prime_j )z>1+(1-\gamma_j)\theta^\prime_j=-\theta^\prime_j\psi_j,\quad \text{ if }\theta_j^\prime> 0.
\eae
Finally, to ensure all the above hold, it  suffices to choose appropriate $y,z$ satisfying \eqref{exist y}-\eqref{z bound cases}. For $\theta_j^\prime> 0$, we discuss the following separate cases.
\begin{enumerate}
    \item [Case 1.]{ $\psi_j=\gamma_j-1/\theta^\prime_j-1>0.$ Using \eqref{exist y}, the existence of a suitable $y$ implies 
    \bae\label{psi_j>0 bar rho}
\bar\rho>\frac{\theta^\prime_j+\xi_j}{\theta^\prime_j\psi_j},\forall \theta_j^\prime> 0.
    \eae
In this case, one has $\bar \rho \theta^\prime_j \psi_j-\xi_j-\theta_j^\prime>0.$ Then, by choosing $z$ sufficiently large, we obtain \eqref{z lower bound} and \eqref{z bound cases}.

    }
     \item [Case 2.]{ $\psi_j=\gamma_j-1/\theta^\prime_j-1=0.$ Then \eqref{exist y} reduces to
      \bae\label{psi_j=0 bar rho}
\xi_j<-\theta^\prime_j.
    \eae
     Note that \eqref{z bound cases} becomes
     \baee
-(\xi_j+\theta^\prime_j ) z>-\theta^\prime_j\psi_j=0.
     \eaee
     Using \eqref{psi_j=0 bar rho} and $z>0$, one obtains \eqref{z lower bound} and \eqref{z bound cases}.
     }
    
     Summarizing Cases 1 and 2, we conclude that Assumption {\rm{(A1)}} implies \eqref{psi_j>0 bar rho} and \eqref{psi_j=0 bar rho}, respectively. 
     \item [Case 3.]{ $\psi_j=\gamma_j-1/\theta^\prime_j-1<0.$ Then \eqref{exist y} is equal to
     \bae\label{psi_j<0 bar rho}
\rho<\frac{\theta^\prime_j+\xi_j}{\theta^\prime_j\psi_j},\quad \forall \theta_j^\prime> 0,
    \eae
    which is guaranteed by the definition of $\bar \rho$ in \eqref{bar rho 2nd} and $\rho<\bar \rho.$
    Note that the RHS of \eqref{z bound cases} is positive. Using \eqref{psi_j<0 bar rho}, one has $\rho\theta^\prime_j \psi_j-\xi_j>\theta_j^\prime>0$. In this case, to ensure \eqref{z lower bound} and \eqref{z bound cases}, it suffices to take $z$ large enough.
     }
\end{enumerate}
\end{proof}
\begin{remark}
    We emphasize that the parameters in Assumptions {\rm{(A1) }} and {\rm{(A2) }} are sharp with respect to the Stochastic Area Inequality framework. As demonstrated in the proof, Assumptions {\rm{(A1) }} and {\rm{(A2) }} are equivalent to the necessary and sufficient conditions for the existence of admissible weights $y, w, z > 0$ required to close the weighted energy estimates. Specifically, the boundary in {\rm{(A1) }} represents the critical scaling threshold where the stochastic energy injection scales identically with the viscous dissipation. Any relaxation of these conditions would result in an empty parameter set  in the derivation of \eqref{middle Xt}.
\end{remark}
\subsection{Proofs of Uniqueness Lemmas}\label{app kruz}
In this Appendix, we provide the detailed convergence arguments for the error terms defined in Section \ref{subsec kruz}. While the algebraic structure of these terms mirrors the inviscid stochastic setting analyzed by \cite{biswas2014stochastic}, we include the full proofs here to track the specific dependencies on the mollification parameters $\epsilon$ and $\delta$ in the presence of the viscous background state.

First we establish a weak $L^1$ continuity Lemma for  later use.
\begin{Lemma}\label{weakL^1cont}
    Let $u(t)-\bar u(t)\in X_1(T)$ be the solution of \eqref{perturbation} with the initial data $u_0-\bar u_0$. Then for every non-negative test function $\psi \in C_c^2\left(\mathbb{R}\right)$,
$$
\lim _{h \rightarrow 0} \mathbb E\left[\frac{1}{h} \int_0^h \int_\R\left|u(t, x)-u_0(x)\right| \psi(x) d x d t\right]=0 .
$$
\end{Lemma}

\begin{proof}
    Since $u-\overline{u} \in C([0,T]; L^2(\mathbb{R}))$ almost surely, the pathwise map $t \mapsto u(t)$ is continuous in $\rm{L_{loc}^{1}}(\mathbb{R})$. The result follows from the dominated convergence theorem and standard density arguments as detailed in \cite[Lemma 2.1]{biswas2014stochastic}.
\end{proof}

\begin{Lemma}\label{lm5.3}
	\baee
	& \lim _{(\delta,\varepsilon)\rightarrow (0,0) }  \lim_{\delta_0 \rightarrow 0 } I_1 
	= & \mathbb{E} \int_{\mathbb{R}}\left|u_0(y)-v_0(y)\right| \psi(0, y) d y .
	\eaee
\end{Lemma}
	\begin{proof}
Define
		\begin{equation*}
			\begin{aligned}
				\mathcal{A}_1:= & \mathbb E \int_{\Pi_T} \int_{\mathbb{R}} \beta\left(u_0(x)-v(s, y)\right) \psi(s, y) \varrho_\delta(x-y) \rho_{\delta_0}(-s) d x d y d s \\
				& -\mathbb E \int_{\mathbb{R}} \int_{\mathbb{R}} \beta\left(u_0(x)-v_0(y)\right) \psi(0, y) \varrho_\delta(x-y) d x d y.
			\end{aligned}
		\end{equation*}
Suppose that $\operatorname{supp} \psi(s, \cdot) \subset K$ compact. Then following the procedures in \cite[Lemma 3.1]{biswas2014stochastic}, we have
		\begin{equation*}
			\begin{aligned}
				\left|\mathcal{A}_1\right| \leq & \left\|\psi_t\right\|_{\infty}\left\|\beta^{\prime}\right\|_{\infty} \delta_0 \mathbb E \int_{\Pi_T} \int_{\mathbb{R}} \chi_K(y) \chi_{K+\delta}(x) \left|u_0(x)-v(s, y)\right| \rho_{\delta_0}(-s) \varrho_\delta(x-y) d x d y d s \\
				& +\left\|\beta^{\prime}\right\|_{\infty}\mathbb  E \int_0^T \int_K \psi(0, y)\left|v(s, y)-v_0(y)\right| \rho_{\delta_0}(-s) d y d s\\
     \leq & \left\|\psi_t\right\|_{\infty}\left\|\beta^{\prime}\right\|_{\infty} \delta_0 \left[\int_{K+\delta}|u_0(x)|\dif x  + \mathbb E \int_0^T \int_K|v(s, y)| \rho_{\delta_0}(-s) d y d s\right] \\
				& +\left\|\beta^{\prime}\right\|_{\infty}\mathbb  E \int_0^T \int_K \psi(0, y)\left|v(s, y)-v_0(y)\right| \rho_{\delta_0}(-s) d y d s\\
     \leq & \left\|\psi_t\right\|_{\infty}\left\|\beta^{\prime}\right\|_{\infty} \delta_0 \left[\left(|K|+\delta\right)^\frac12\|u_0\| + |K|^\frac12 \sup_{0\leq s\leq T}\mathbb E \|v(s)\|\right] \\
				& +C\left\|\beta^{\prime}\right\|_{\infty} \frac{1}{\delta_0} \int_0^{\delta_0} \mathbb E\left(\int_K \psi(0, y)\left|v(r, y)-v_0(y)\right| d y\right) d r .
			\end{aligned}
		\end{equation*}
		By Lemma \ref{weakL^1cont}, we arrive at
  $$
\lim _{\delta_0 \rightarrow 0} I_1=\mathbb E \int_{\mathbb{R}} \int_{\mathbb{R}} \beta\left(u_0(x)-v_0(y)\right) \psi(0, y) \varrho_\delta(x-y) d x d y.
$$
The passage of the limit $(\varepsilon,\delta)\rightarrow(0,0)$ is followed by the same arguments in the Step 2 of \cite[Lemma 3.1]{biswas2014stochastic}. Thus the proof is completed.
	\end{proof}

\begin{Lemma}
    \baee
&\lim _{(\delta,\varepsilon)\rightarrow (0,0) }  \lim_{\delta_0 \rightarrow 0 } I_2
&=\mathbb E\left[\int_{\Pi_T}|v(s, y)-u(s, y)| \partial_s \psi(s, y) d y d s\right] .
\eaee
\end{Lemma}
\begin{proof}
   It follows from the same proof as in \cite[Lemma 3.2]{biswas2014stochastic}. 
\end{proof}
\begin{Lemma}
    \baee
\lim _{\left(\varepsilon, \delta^{-1} \varepsilon, \delta_0\right) \rightarrow(0,0,0)} I_4&=0 .
\eaee
\end{Lemma}
\begin{proof}
   Using the same arguments as before, we have 
   \baee
\lim_{\delta_0\rightarrow 0}I_4&=-\mathbb{E} \int_{\R} \int_{\Pi_T}\left(f_\varepsilon\big(u(s), v(s)\big) \partial_x \varrho_\delta(x-y)+f_\varepsilon\big(v(s), u(s)\big) \partial_y \varrho_\delta(x-y)\right)  \psi d y d s d x\\
&=-\mathbb{E} \int_{\R} \int_{\Pi_T}\left[-f_\varepsilon\big(u(s), v(s)\big)+f_\varepsilon\big(v(s), u(s)\big)\right]\partial_y \varrho_\delta(x-y)  \psi d y d s d x.
   \eaee
Note that for any $m,n\in \mathbb R$
\begin{equation}\label{f1}
    \begin{aligned}
        &\left|-f_\varepsilon(m, n)+f_\varepsilon(n, m)\right|\\
        =&\left|-\int_n^m \beta^\prime_\varepsilon(z-n) f^\prime(z)dz +\int_m^n \beta^\prime_\varepsilon(z-m) f^\prime(z)dz\right|\\
=&\left|-f_\varepsilon(m, n)+\operatorname{\rm{sign}}(m-n)\big(f(n)-f(m)\big)-\operatorname{\rm{sign}}(m-n)\big(f(n)-f(m)\big)+f_\varepsilon(n, m)\right| \\
\leq & \left| \int_{n}^m \left(  \beta^\prime_\varepsilon(z-n)-\operatorname{\rm{sign}}(z-n) \right) f^\prime(z) dz\right|+ \left| \int_{m}^n \left(  \beta^\prime_\varepsilon(z-m)-\operatorname{\rm{sign}}(z-m) \right) f^\prime(z) dz\right|\\
\leq &  \int_{n-\varepsilon}^{n+\varepsilon} \left| \left(  \beta^\prime_\varepsilon(z-n)-\operatorname{\rm{sign}}(z-n) \right) f^\prime(z)\right| dz+  \int_{m-\varepsilon}^{m+\varepsilon}\left| \left(  \beta^\prime_\varepsilon(z-m)-\operatorname{\rm{sign}}(z-m) \right) f^\prime(z)\right| dz\\
\leq & \,C \varepsilon \left( \sup_{z\in [n-\varepsilon,n+\varepsilon]} \left| f^\prime(z) \right|+\sup_{z\in [m-\varepsilon,m+\varepsilon]} \left| f^\prime(z) \right| \right).
    \end{aligned}
\end{equation}
Therefore, by the maximal principle \eqref{mp}, we have
\baee
&\left| \mathbb{E} \int_{\R} \int_{\Pi_T}\left[-f_\varepsilon\big(u(s), v(s)\big)+f_\varepsilon\big(v(s), u(s)\big)\right]\partial_y \varrho_\delta(x-y)  \psi d y d s d x \right|\\
\leq & \, C \varepsilon \sup_{ |z|\leq  \|u_0\|_{\infty}\vee \|v_0\|_{\infty} } \left| f^\prime(z) \right| \int_{\R} \int_{\Pi_T} \left| \partial_y \varrho_\delta(x-y) \right|  \psi d y d s d x \leq  C \varepsilon  \delta^{-1}.
\eaee
This proves the Lemma.
\end{proof}
\begin{Lemma}\label{lm56}
    \bae
\lim _{(\delta,\varepsilon)\rightarrow (0,0) }  \lim_{\delta_0 \rightarrow 0 } I_5&=\mathbb{E} \int_{\Pi_T} \operatorname{sign}(u-v)(f(u)-f(v)) \partial_y \psi d y d s .
\eae
\end{Lemma}
\begin{proof}
Note that for any $m,n,l\in\mathbb R,$ one has
    \begin{equation}\label{f2}
    \begin{aligned}
        &\left|-f_\varepsilon(m, n)+f_\varepsilon(m, l)\right|\\
= & \left| -\int_{n}^m \beta^\prime_\varepsilon(z-n)  f^\prime(z) dz + \int_{l}^m \beta^\prime_\varepsilon(z-l)  f^\prime(z) dz\right|\\
\leq  & \left| \int_{n}^m \beta^\prime_\varepsilon(z-n)  f^\prime(z) dz -\operatorname{\rm{sign}}(m-n)\big(f(m)-f(n)\big)\right|\\
+&\left|  \int_{l}^m \beta^\prime_\varepsilon(z-l)  f^\prime(z) dz - \operatorname{\rm{sign}}(m-l)\big(f(m)-f(l)\big)\right|\\
+ &\left| \operatorname{\rm{sign}}(m-n)\big(f(m)-f(n)\big) - \operatorname{\rm{sign}}(m-l)\big(f(m)-f(l)\big)\right| \\
\leq & C \varepsilon \left( \sup_{z\in [n-\varepsilon,n+\varepsilon]} \left| f^\prime(z) \right|+\sup_{z\in [l-\varepsilon,l+\varepsilon]} \left| f^\prime(z) \right| \right)+ |n-l| \sup_{z\in [{n\wedge l},{n\vee l}]} \left| f^\prime(z) \right|.\\
    \end{aligned}
\end{equation}
Then following the procedures in \cite[Lemma 3.3]{biswas2014stochastic} and \eqref{mp}, we conclude that
$$
\begin{aligned}
\mathcal{B}_1:= & \left|\mathbb E\int_{\Pi_T} \int_{\Pi_T} f_\varepsilon(v(s, y), u(t, x)) \partial_y \psi(s, y) \rho_{\delta_0}(t-s) \varrho_\delta(x-y) d y d s d x d t \right. \\
& \left. -\mathbb E \int_{\mathbb{R}} \int_{\Pi_T} f_\varepsilon(v(s, y), u(s, x)) \partial_y \psi(s, y) \varrho_\delta(x-y) d y d s d x\right|\\
\leq&\, K\left\|\partial_y \psi(s, y)\right\|_{\infty}\mathbb E\int_{s=\delta_0}^T \int_{\mathbb{R}} \int_{t=0}^T \mid u(s, x)  -u(t, x) \mid \sup_{|z|\leq\|u_0\|_\infty\vee \|v_0\|_\infty} \left| f^\prime(z) \right| \\
& \times \rho_{\delta_0}(t-s) d t d x d s+\mathcal{O}\left(\delta_0\right)+\mathcal{O}\left(\varepsilon\right)\\
\leq& \, C(u_0,v_0)\left\|\nabla_y \psi(s, y)\right\|_{\infty}\left[\mathbb E \int_{r=0}^1 \int_{\mathbb{R}} \int_{t=0}^T\left|u\left(t+\delta_0 r, x\right)-u(t, x)\right|^2 \rho(-r) d t d x d r\right]^{\frac{1}{2}}\\
&+\mathcal{O}\left(\delta_0\right)+\mathcal{O}\left(\varepsilon\right).
\end{aligned}
$$
Using \eqref{mp}, \eqref{f1} and \eqref{f2}, the rest of the proof follows from the same arguments in \cite[Lemma 3.3]{biswas2014stochastic}.
\end{proof}

To address the non-adapted ``stochastic integral'' appearing in the calculation of \eqref{kruz} in the rest of the subsection, we assume that $u\in X_1(T)$ is the solution of equation \eqref{conservative}. We demonstrate that such a solution must be the stochastic strong entropy solution defined in \cite{feng2008stochastic}.  Define
 $$   J[\beta,\eta](s,y,a):=-\sum_k\int_0^T \int_\R\left[\sigma_k^{\prime} \beta(\cdot-a)\right]u(t, x)  \eta(t, x, s, y)\dif x  d B_k(t).$$
 In order to prove Lemma \ref{adaptlm}, we require the following Lemmas.
\begin{Lemma}\label{J}
      Let $\beta \in C^{\infty}(\mathbb{R})$ be a function such that $\beta^{\prime}, \beta^{\prime \prime} \in C_c^{\infty}(\mathbb{R})$. Then there exists a constant $C=C\left(\beta^{\prime}, \psi\right)$ and $p_\gamma>2$ such that
$$
\sup _{0 \leq s \leq T}\left(\mathbb E\left\|J\left[\beta, \eta_{\delta, \delta_0}\right](s , \cdot, \cdot)\right\|_{L^{\infty}(\mathbb{R} \times \mathbb{R})}^2\right) \leq \frac{C(\partial_y\psi,\beta^\prime,T,\sigma)  \sup_{0\leq t \leq T}\left(1+ 
\mathbb  E \|u-\bar u(t,\cdot)\|_{p_\gamma}^{p_\gamma} \right)}{\delta_0^{\frac32}}.
$$
\end{Lemma}
\begin{proof}
Note that
   \baee
   J[\beta,\eta_{\delta, \delta_0}](s,y,a) &=-\sum_k\int_0^T \int_\R\int_0^{u(t, x)}\sigma_k^{\prime}(r) \beta(r-a)dr  \eta_{\delta, \delta_0}d x  d B_k(t).\\
   \eaee
   Then we have
   \bae\label{adptech}
& \mathbb E\left[\left\|J\left[\beta, \eta_{\delta, \delta_0}\right](s , \cdot, \cdot)\right\|_4^4\right] =\mathbb E\left[\int_\R \int_\R\left|J\left[\beta, \eta_{\delta, \delta_0}\right](s , y, a)\right|^4 d y d a\right] \\
 =&\sum_k\mathbb E\bigg[\int_\R \int_\R\Big|\int_0^T \int_\R \big(\int_0^{u(t, x)}\sigma_k^{\prime}(r) \beta(r-a)dr\big) \rho_{\delta_0}(t-s) \varrho_\delta(x-y)  d x d B_k(t)\Big|^4 \psi^4(s, y) d y d a\bigg]\\
\leq&\, C(\psi) \sum_k\int_a \int_{y\leq C_\psi} \mathbb E\bigg[\Big(\int_0^T\Big|\int_\R\big|\int_0^{u(t, x)}\sigma_k^{\prime}(r) \beta(r-a)dr\big|  \varrho_\delta(x-y) \dif x \Big|^2 \rho^2_{\delta_0}(t-s) d t\Big)^2\bigg] d y d a\\
\leq&\, C(\psi)\sum_k \int_a \int_{y\leq C_\psi} \mathbb E\bigg[\Big(\int_0^T\int_\R\big|\int_0^{u(t, x)}\sigma_k^{\prime}(r) \beta(r-a)dr\big|^2  \varrho_\delta(x-y)  \rho^2_{\delta_0}(t-s) d t\dif x \Big)^2\bigg] d y d a\\
\leq&\, C(\psi,T)\sum_k \int_a \int_{y\leq C_\psi} \mathbb E\Big(\int_0^T\int_\R\big|\int_0^{u(t, x)}\sigma_k^{\prime}(r) \beta(r-a)dr\big|^4  \varrho_\delta(x-y)\rho^4_{\delta_0}(t-s)d t \dif x  \Big) d y d a\\
\leq&\, C(\psi,T)\sum_k\mathbb E   \Big(\int_0^T\int_{|x|\leq C_\psi+\delta}\int_{|a|\leq C_\beta+|u(t,x)|}\big|\int_0^{u(t, x)}\sigma_k^{\prime}(r) \beta(r-a)dr\big|^4   \rho^4_{\delta_0}(t-s)d tda \Big)  d x\\
\leq&\, C(\psi,T) \sum_k\mathbb E  \Big[\int_0^T \int_{|x|\leq C_\psi+\delta} \int_{|a|\leq C_\beta+|u(t,x)|} \|\sigma_k^\prime\|_{L^2_M}^4 \|\beta(u-a)\|_{\infty}^4 u^4(t,x)  \rho^4_{\delta_0}(t-s)d td a \Big]  \dif x ,\\
\eae
where we have used the Burkholder-Davis-Gundy inequality in the first inequality.  Setting $\mathfrak M=\|u_0\|_{\infty}$ and $p_\gamma=13+4\gamma$, we deduce from \eqref{mp} that
 \baee
 &\, \mathbb E\left[\left\|J\left[\beta, \eta_{\delta, \delta_0}\right](s , \cdot, \cdot)\right\|_4^4\right] =\mathbb E\left[\int_\R \int_\R\left|J\left[\beta, \eta_{\delta, \delta_0}\right](s , y, a)\right|^4 d y d a\right] \\
\leq&\, C(\psi,T,\beta,\sigma) \|\rho_{\delta_0}\|^3_{\infty}\mathbb E   \int_0^T \int_{|x|\leq C_\psi+\delta}\int_{|a|\leq C_\beta+|u(t,x)|} (1+|u(t,x)|)^{4\gamma+4+8}  \rho_{\delta_0}(t-s)da  d x  d t\\
\leq& \,\frac{C(\psi,\beta,T,\sigma) \sup_{0\leq t \leq T}\Big(1+ 
\mathbb  E \|u(t,\cdot)\|_{L^{p_\gamma}_{C_\psi}}^{p_\gamma} \Big)}{\delta_0^{3}},  \\
\leq&\, \frac{C(\psi,\beta,T,\sigma) \sup_{0\leq t \leq T}\left(1+ 
\mathbb  E \|u-\bar u(t,\cdot)\|_{{p_\gamma}}^{p_\gamma} \right)}{\delta_0^{3}}. \\
\eaee
 Consequently, \eqref{adptech} is bounded by invoking Lemma \ref{lmL^p}. Similarly, we have
$$
\begin{aligned}
& \mathbb E\left[\left\|\partial_a J\left[\beta, \eta_{\delta, \delta_0}\right](s ; \cdot, \cdot)\right\|_4^4\right] \leq \frac{C(\psi,\beta^\prime,T,\sigma) \|  \sup_{0\leq t \leq T}\left(1+ 
\|u-\bar u(t,\cdot)\|_{{p_\gamma}}^{p_\gamma}\right) }{\delta_0^{3}} , \\
& \mathbb E\left[\left\|\partial_y J\left[\beta, \eta_{\delta, \delta_0}\right](s ; \cdot, \cdot)\right\|_4^4\right] \leq \frac{C(\partial_y\psi,\beta,T,\sigma)   \sup_{0\leq t \leq T}\left(1+ 
\|u-\bar u(t,\cdot)\|_{{p_\gamma}}^{p_\gamma} \right)}{\delta_0^{3}} .
\end{aligned}
$$
Therefore,
$$
\mathbb E\left[\left\|J\left[\beta, \eta_{\delta, \delta_0}\right](s ; \cdot, \cdot)\right\|_{W^{1,4}(\mathbb{R} \times \mathbb{R})}^4\right] \leq \frac{C(\partial_y\psi,\beta^\prime,T,\sigma)  \sup_{0\leq t \leq T}\left(1+ 
\mathbb  E \|u-\bar u(t,\cdot)\|_{p_\gamma}^{p_\gamma} \right)}{\delta_0^{3}}.
$$
Now simply using the Sobolev embedding along with the Cauchy-Schwartz inequality, we conclude that
$$
\sup _{0 \leq s \leq T}\left(\mathbb E\left[\left\|J\left[\beta, \eta_{\delta, \delta_0}\right](s ; \cdot, \cdot)\right\|_{L^{\infty}(\mathbb{R} \times \mathbb{R})}^2\right]\right) \leq \frac{C(\partial_y\psi,\beta^\prime,T,\sigma)  \sup_{0\leq t \leq T}\left(1+ 
\mathbb  E \|u-\bar u(t,\cdot)\|_{p_\gamma}^{p_\gamma} \right)}{\delta_0^{\frac32}}.
$$
\end{proof}
Denote by
\baee
A_1^l(\delta,\delta_0 )&\coloneqq\mathbb E\left[\int _ { \mathbb { R } } \int _ { \Pi _ { T } } J [ \beta ^ { \prime\prime } , \partial_x\eta _ { \delta , \delta _ { 0 } } ] ( s , y , a ) \left(\int_{s-\delta_0}^s \rho_l(v(\tau, y)-a) \partial_y\left(f(v)\right) d \tau\right)dy ds da\right],\\
A_2^l(\delta,\delta_0 )&\coloneqq\mathbb E\left[\int _ { \mathbb { R } } \int _ { \Pi _ { T } } J [ \beta ^ { \prime\prime } , \partial_x\eta _ { \delta , \delta _ { 0 } } ] ( s , y , a ) \left(\int_{s-\delta_0}^s \rho_l(v(\tau, y)-a) \partial_y(b(v)v_y) d \tau\right)   dy ds da\right],  \\
A_3^l(\delta,\delta_0 )&\coloneqq-\frac12\mathbb E\left[\int _ { \mathbb { R } } \int _ { \Pi _ { T } } J [ \beta ^ { \prime\prime } , \partial_{xy}\eta _ { \delta , \delta _ { 0 } } ] ( s , y , a )   \left(\int_{s-\delta_0}^s \rho_l(v-a)\sigma^{\prime 2}(v) v_y  d \tau\right)   dy ds da\right],\\
B^l(\delta,\delta_0 )&\coloneqq-\sum_k\mathbb{E}  \int_{\Pi_T} \int_\R \int_{s-\delta_0}^s \left(\int_\R \left[\sigma_k^{\prime} \beta^{\prime \prime}(\cdot-a)\right](u)    \rho_l(v-a) da\right) \partial_y \sigma_k(v) \partial_x \eta_{\delta , \delta_0} d \tau \dif x  dy ds.
\eaee

\begin{Lemma}\label{A1}
    It holds that
$$
\lim_{\delta_0\rightarrow0}\lim_{l\rightarrow0} A_1^l\left(\delta, \delta_0\right) = 0.
$$
\end{Lemma}
\begin{proof}
Let
$$
X\left[\eta\right](s , y, a):=\sum_k\int_0^T  \int_\R
\left(\int_0^a \left[\sigma_k^{\prime} \beta^{\prime\prime}(\cdot-r)\right](u)  f^\prime(r)dr\right)\eta(t,x,s,y) \dif x  d B_k(t).
$$
It holds that
$$
\begin{aligned}
& \partial_a X\left[\eta\right](s , y, a)=\sum_k\int_0^T  \int_\R
 \left[\sigma_k^{\prime} \beta^{\prime\prime}(\cdot-a)\right](u)  f^\prime(a)\eta(t,x,s,y) \dif x  d B_k(t),  \\
& \partial_y X\left[\eta\right](s , y, a)=X\left[\partial_y \eta\right](s , y, a) .
\end{aligned}
$$
Following arguments similar to those in Lemma \ref{J}, we have
\baee
 \sup _{0 \leq s \leq T}\left(\mathbb E\left[\left\|X\left[\partial_{xy} \eta_{\delta, \delta_0}\right](s ; \cdot, \cdot)\right\|_{L^{\infty}(\mathbb{R} \times \mathbb{R})}^2\right]\right) \leq \frac{C(\partial_{xy}\psi,\beta^{\prime\prime\prime},T,f,\sigma)  \sup_{0\leq t \leq T} 
\left(1+\mathbb  E \|u-\bar u(t,\cdot)\|_{p_\gamma}^{p_\gamma}\right) }{\delta_0^{\frac32}}.
\eaee
Using integration by parts, we have
\baee
 & \int _ { \mathbb { R } } \int _ { \Pi _ { T } } J [ \beta ^ { \prime\prime } , \partial_x\eta _ { \delta , \delta _ { 0 } } ] ( s , y , a ) \left(\int_{s-\delta_0}^s \rho_l(v(\tau, y)-a) f^{\prime}(a) v_y(\tau,y) d \tau\right)dy ds da\\
& =\sum_k\int _ { \mathbb { R } } \int _ { \Pi _ { T } } \int_0^T \int_\R\left[\sigma_k^{\prime} \beta^{\prime\prime}(\cdot-a)\right](u) \partial_x \eta_{\delta, \delta_0}\dif x  d B_k(t) \left(\int_{s-\delta_0}^s \rho_l(v-a) f^{\prime}(a) v_y d \tau\right)dy ds da\\
& =\sum_k\int _ { \mathbb { R } } \int _ { \Pi _ { T } }\int_{s-\delta_0}^s \left(\int_0^T  \int_\R
\left[\sigma_k^{\prime} \beta^{\prime\prime}(\cdot-a)\right](u)  f^\prime(a)\partial_x \eta_{\delta, \delta_0} \dif x  d B_k(t) \right)  \rho_l(v-a) v_y  d \tau dy ds da\\
& =\int _ { \mathbb { R } } \int _ { \Pi _ { T } }\int_{s-\delta_0}^s \partial_a X\left[\partial_x\eta _ { \delta , \delta _ { 0 } } \right](s , y, a)  \rho_l(v-a) v_y(\tau,y)  d \tau dy ds da\\
& =\int _ { \mathbb { R } } \int _ { \Pi _ { T } }\int_{s-\delta_0}^s  X\left[\partial_x\eta _ { \delta , \delta _ { 0 } } \right](s , y, a)  \rho_l^\prime(v-a) v_y(\tau,y)  d \tau dy ds da\\
& =-\int _ { \mathbb { R } } \int _ { \Pi _ { T } }\int_{s-\delta_0}^s  X\left[\partial_{xy}\eta _ { \delta , \delta _ { 0 } } \right](s , y, a)  \rho_l(v-a)   d \tau dy ds da.\\
\eaee
Letting $l\rightarrow0$ on both sides of the equality above leads to
\bae\label{X1}
 &A_1\left(\delta, \delta_0\right)\coloneqq\lim_{l\rightarrow0} A_1^l\left(\delta, \delta_0\right)=  \int _ { \Pi _ { T } } \int_{s-\delta_0}^s J [ \beta ^ { \prime\prime } , \partial_x\eta _ { \delta , \delta _ { 0 } } ] ( s , y , v(\tau, y) )  \partial_y f(v(\tau, y)) d \tau dy ds  \\
 &=-\mathbb E \left[ \int _ { \Pi _ { T } }\int_{s-\delta_0}^s  X[\partial_{xy}\eta _ { \delta , \delta _ { 0 } } ](s , y, v(\tau,y))     d \tau dy ds  \right].
\eae
Then
\bae\label{X2}
\left|A_1\left(\delta, \delta_0\right)\right| & =\left|\mathbb E \left[ \int _ { \Pi _ { T } }\int_{s-\delta_0}^s  X[\partial_{xy}\eta _ { \delta , \delta _ { 0 } } ](s , y, v(\tau,y))     d \tau dy ds  \right]\right| \\
& \leq C \delta_0 \sup _{0 \leq s \leq T} \left(\mathbb E\left[\left\|X\left[\partial_{xy} \eta_{\delta, \delta_0}\right](s ; \cdot, \cdot)\right\|^2_{\infty}\right]\right)^{\frac12} \\
& \leq C \delta_0 \sup _{0 \leq s \leq T} \left(E\left[\left\|X\left[\partial_y \eta_{\delta, \delta_0}\right](s ; \cdot, \cdot)\right\|_{\infty}^2\right]\right)^{\frac{1}{2}} \\
& \leq C \delta_0 \frac{C(\partial_{xy}\psi,\beta^{\prime\prime\prime},T,f,\sigma)  \sup_{0\leq t \leq T} 
\left(1+\mathbb  E \|u-\bar u(t,\cdot)\|_{p_\gamma}^{p_\gamma/2}\right) }{\delta_0^{\frac34}} \\
& \leq C(\partial_{xy}\psi,\beta^{\prime\prime\prime},T,u_0,f,\sigma) \delta_0^{\frac{1}{4}} .
\eae
\end{proof}

\begin{Lemma}\label{A2}
    It holds that
$$
\lim_{\delta_0\rightarrow0}\lim_{l\rightarrow0} A_2^l\left(\delta, \delta_0\right) = 0.
$$
\end{Lemma}
\begin{proof}
\baee
 A_2\left(\delta, \delta_0\right)\coloneqq& \lim_{l\rightarrow0}A_2^l(\delta,\delta_0 )=\mathbb E\left[\int _ { \Pi _ { T } }\int_{s-\delta_0}^s J [ \beta ^ { \prime\prime } , \partial_x\eta _ { \delta , \delta _ { 0 } } ] ( s , y , v(\tau, y) )  \partial_{y} (b(v)v_y) d \tau  dy ds \right] \\
 =&\mathbb E\left[\int _ { \Pi _ { T } }\int_{s-\delta_0}^s J [ \beta ^ { \prime\prime } , \partial_{xy}\eta _ { \delta , \delta _ { 0 } } ] ( s , y , v(\tau, y) )  b(v) v_y d \tau  dy ds \right]\\
 &-\mathbb E\left[\int _ { \Pi _ { T } }\int_{s-\delta_0}^s J [ \beta ^ { \prime\prime\prime } , \partial_x\eta _ { \delta , \delta _ { 0 } } ] ( s , y , v(\tau, y) )   b{(v)} v_y^2 d \tau  dy ds \right]\\
 \coloneqq& I_1+I_2.
 \eaee
Applying the same arguments as in \eqref{X1} and \eqref{X2}, we have
\baee
|I_1|\leq C(\partial_{xy}\psi,\beta^{\prime\prime\prime},T,u_0,b,\sigma) \delta_0^{\frac{1}{4}}.
\eaee

To estimate the quadratic dissipation term $I_2$, we cannot use integration by parts. Instead, we rely on the time-regularity of the stochastic kernel. Define
$$
N_t\left[\beta^{\prime \prime \prime}, \psi, \delta\right](y, a) \coloneqq \sum_k \int_0^t \int_\R \left[\sigma_k^\prime \beta^{\prime \prime \prime}(\cdot-a)\right](u(r,x)) \psi(s, y) \partial_x \varrho_\delta(x-y) d x d B_k(r),
$$
where $t\geq s-\delta_0$.
Following a similar argument in \cite[Lemma 5.5]{biswas2014stochastic}, we have
\begin{equation*}
\begin{aligned}
& \mathbb E\left[\sup _{0 \leq s \leq T}\left\|J\left[\beta^{\prime \prime \prime}, \partial_x\eta_{\delta, \delta_0}\right](s ; \cdot, \cdot)\right\|_{L^{\infty}(\mathbb{R} \times \mathbb{R})}\right] \\
\leq & \frac{1}{\delta_0} \mathbb E\left[\sup _{0 \leq s \leq T ; s-\delta_0 \leq t<s}\left\|N_t\left[\beta^{\prime \prime \prime}, \psi, \delta\right](\cdot, \cdot)-N_{s-\delta_0}\left[\beta^{\prime \prime \prime}, \psi, \delta\right](\cdot, \cdot)\right\|_{L^{\infty}(\mathbb{R} \times \mathbb{R})}\right].
\end{aligned}
\end{equation*}
Note that
\begin{equation*}
\begin{aligned}
& \mathbb E\left[\|N_t\left[\beta^{\prime \prime \prime}, \psi, \delta\right](\cdot, \cdot)-N_{s}\left[\beta^{\prime \prime \prime}, \psi, \delta\right](\cdot, \cdot)\|_p^p\right] \\
& =\sum_k\int_{\R} \int_{\R} \mathbb E\left[\left|\int_s^t \int_\R \left[\sigma_k^\prime \beta^{\prime \prime \prime}(\cdot-a)\right](u(r,x)) \psi(s, y) \partial_x \varrho_\delta(x-y) d x d B_k(r)\right|^p\right] d y d a \\
& \leq C \int_\R\int_\R \mathbb E\left[\left(\int_s^t \left|\int_\R \left[\sigma^\prime \beta^{\prime \prime \prime}(\cdot-a)\right](u(r,x)) \psi(s, y) \partial_x \varrho_\delta(x-y) d x\right|^2  d r\right)^{p / 2}\right] d y d a \\
& \leq C(\sigma,\beta,\psi,\partial_x \varrho)|t-s|^{\frac p2-1} \mathbb E\left[ \int_{|a|\leq |u(r,x)|+C_\beta} \int_{|y|<C_{\psi}}  \int_0^T \int_{|x|<C_{\psi}+\delta}\left|u(r,x)\right|^p  d x d r\right] d y d a \\
&\leq C(\sigma,\beta,\psi,\partial_x \varrho,u_0,T)|t-s|^{\frac p2-1} .\\
\end{aligned}
\end{equation*}
Invoking a similar modulus of continuity estimate \cite[Lemma 4.28]{feng2008stochastic}, we conclude that
$$
\mathbb E\left[\sup _{s, t \in[0, T] ;|s-t|<\delta_0}\left\|N_t\left[\beta^{\prime \prime \prime}, \psi, \delta\right](\cdot, \cdot)-N_s\left[\beta^{\prime \prime \prime}, \psi, \delta\right](\cdot, \cdot)\right\|_{\infty}^p\right] \leq C \delta_0^\alpha
$$
for some $\alpha> 0$ and $p > 8$. Substituting this uniform bound into the expression for $I_2$, and using the boundedness of $b(v)$, we obtain $$|I_2| \le C \mathbb{E} \left[ \int_{\mathbb{R}} \int_0^T \left( \sup_{y,a} |J| \right) \left( \int_{s-\delta_0}^s |v_y(\tau, y)|^2 d\tau \right) ds dy \right].$$Using the uniform bound for $J$ and applying Fubini's theorem to exchange the time integrals, we get $$|I_2| \le C \delta_0^{\alpha-1} \mathbb{E} \left[ \int_{\mathbb{R}} \int_0^T |v_y(\tau, y)|^2 \left( \int_{\tau}^{(\tau+\delta_0)\wedge T} ds \right) d\tau dy \right].$$The inner integral over $s$ yields a factor of exactly $\delta_0$. Thus, using the global energy bound $\mathbb{E}\|v_y\|_{L^2}^2 < \infty$, we conclude $$|I_2| \le C \delta_0^{\alpha-1} \cdot \delta_0 \cdot \mathbb{E}\|v_y\|_{L^2(\Pi_T)}^2 = C \delta_0^\alpha.$$Since $\alpha > 0$, taking the limit as $\delta_0 \to 0$ yields $\lim I_2 = 0$. Combining this with the estimate for $I_1$ completes the proof.
\end{proof}

\begin{Lemma}\label{A3}
    It holds that
$$
\lim_{\delta_0\rightarrow0}\lim_{l\rightarrow0} A_3^l\left(\delta, \delta_0\right) = 0.
$$
\end{Lemma}
\begin{proof}
\baee
A_3^l(\delta,\delta_0 )&\coloneqq-\frac12\mathbb E\left[\int _ { \mathbb { R } } \int _ { \Pi _ { T } } J [ \beta ^ { \prime\prime } , \partial_{xy}\eta _ { \delta , \delta _ { 0 } } ] ( s , y , a )   \left(\int_{s-\delta_0}^s \rho_l(v-a)\partial_y [\sigma^{\prime2}](v(\tau,y))  d \tau\right)   dy ds da\right].
\eaee
The term $A_3$ involves the kernel $J[\beta'', \partial_{xy}\eta_{\delta, \delta_0}]$, which contains a second-order spatial derivative. While this term is more singular in space ($\sim \delta^{-2}$) than the terms in Lemma \ref{A1}, we observe that for any fixed spatial radius $\delta > 0$, the operator $\partial_{xy} \rho_\delta * \cdot$ is bounded. Consequently, we can apply the same time-averaging argument as in Lemma \ref{A2}. The Hölder continuity of the stochastic process $N_t$ yields a factor of $\delta_0^\alpha$ which vanishes as $\delta_0 \to 0$. Thus, for fixed $\delta$, we have$$|A_3(\delta, \delta_0)| \le C(\delta) \delta_0^\alpha \rightarrow 0, \quad \text{as ${\delta_0 \to 0}$}.$$
This allows us to conclude the limit without additional restrictions on the scaling between $\delta$ and $\delta_0$. Specifically, following the same arguments in \eqref{X1} and \eqref{X2}, we have
\baee
\left|\lim_{l\rightarrow0} A_3^l\left(\delta, \delta_0\right)\right|\leq C(\partial_{xy}\psi,\beta^{\prime\prime\prime},T,u_0,b,\sigma) \delta_0^{\frac{1}{4}}.
\eaee
Finally, taking the limit $\delta_0 \to 0$  before the spatial limit $\delta \to 0$ completes the proof.
\end{proof}
\begin{proof}[\textbf{Proof of Lemma \ref{adaptlm}}]
 Let $\left\{\rho_l\right\}_{l>0}$ be the standard sequence of mollifiers in $\mathbb{R}$ and define
$$
Z_{ \delta, \delta_0, l}:=\int_{\mathbb{R}} \int_{\Pi_T} J\left[\beta^{\prime}, \eta_{\delta, \delta_0}\right](s , y, a) \rho_l\left(v(s, y)-a\right) d y d s d a.
$$
Then
$$I_8=\mathbb E \lim_{l\rightarrow 0}Z_{ \delta, \delta_0, l} .$$
Note that
\begin{equation}\label{6..14}
\begin{aligned}
&\mathbb E Z_{ \delta, \delta_0, l}\\
=&\,\mathbb E\left[\int _ { \mathbb { R } } \int _ { \Pi _ { T } } J [ \beta ^ { \prime } , \partial_x\eta _ { \delta , \delta _ { 0 } } ] ( s , y , a ) \rho _ { l } \left(v\left(s, y\right)-a\right)dy ds da\right]\\
=&\,\mathbb E\left[\int _ { \mathbb { R } } \int _ { \Pi _ { T } } J [ \beta ^ { \prime } , \partial_x\eta _ { \delta , \delta _ { 0 } } ] ( s , y , a ) \left(\rho _ { l } \left(v\left(s, y\right)-a\right)-\rho _ { l } \left(v\left(s-\delta_0, y\right)-a\right)\right)dy ds da\right].\\
\end{aligned}
\end{equation}
Then the It\^o formula deduces that
\begin{equation}\label{6..15}
\begin{aligned}
& \rho_{l}(v(s, y)-a)-\rho_l(v(s-\delta_0, y)-a) \\
=& \int_{s-\delta_0}^s \rho_l^{\prime}(v(\tau, y)-a)\left[\partial_y(b(v)v_y)+\partial_y\left(\frac{1}{2} \sigma^{\prime 2}(v) v_y-f(v)\right)\right] d \tau \\
& +\frac{1}{2} \int_{s-\delta_0}^s \rho_l^{\prime \prime}(v(\tau, y)-a)\left|\partial_y \sigma(v)\right|^2 d \tau+\sum_k\int_{s-\delta_0}^s \rho_l^{\prime}(v(\tau, y)-a) \partial_y \sigma_k(v) d B_k(\tau) \\
=& -\sum_k\partial_a\left[\begin{array}{l}
\int_{s-\delta_0}^s \rho_l(v(\tau, y)-a)\left[\partial_y(b(v)v_y)+\partial_y\left(\frac{1}{2} \sigma^{\prime 2}(v) v_y-f(v)\right)\right] d \tau \\
+\int_{s-\delta_0}^s \rho_l(v(\tau, y)-a) \partial_y \sigma_k(v) d B_k(\tau)+\frac{1}{2} \int_{s-\delta_0}^s \rho_l^{\prime}(v(\tau, y)-a)\left|\partial_y \sigma_k(v)\right|^2 d\tau
\end{array}\right]. \\
\end{aligned}
\end{equation}
Substituting \eqref{6..15} into \eqref{6..14}, we have
\begin{equation*}
\begin{aligned}
&\mathbb E Z_{ \delta, \delta_0, l} \\
=&-\mathbb E\left[\int _ { \mathbb { R } } \int _ { \Pi _ { T } } J [ \beta ^ { \prime\prime } , \partial_x\eta _ { \delta , \delta _ { 0 } } ] ( s , y , a ) \left(\int_{s-\delta_0}^s \rho_l(v(\tau, y)-a) \partial_y\left(f(v)\right) d \tau\right)dy ds da\right] \\
&+\mathbb E\left[\int _ { \mathbb { R } } \int _ { \Pi _ { T } } J [ \beta ^ { \prime\prime } , \partial_x\eta _ { \delta , \delta _ { 0 } } ] ( s , y , a ) \left(\int_{s-\delta_0}^s \rho_l(v(\tau, y)-a) \partial_y(b(v)v_y) d \tau\right)   dy ds da\right] \\
&+\sum_k \mathbb{E} \int_\R \int_{\Pi_T}  \int_\R \int_0^T\left[\sigma_k^{\prime} \beta^{\prime \prime}(\cdot-a)\right](u)  \\
&\cdot\partial_x \eta_{\delta , \delta_0} \left(\int_{s-\delta_0}^s \rho_l(v(\tau, y)-a) \partial_y \sigma_k(v) d B(\tau)\right) d B_k(t) \dif x  d y d s d a \\
& + \sum_k\frac12\mathbb E\Big[\int _ { \mathbb { R } } \int _ { \Pi _ { T } } J [ \beta ^ { \prime\prime } , \partial_x\eta _ { \delta , \delta _ { 0 } } ] ( s , y , a )  \\
&\cdot\Big(\int_{s-\delta_0}^s \rho_l^{\prime}(v-a)\left|\partial_y \sigma_k(v)\right|^2+\rho_l(v-a) \partial_y\left(\sigma_k^{\prime 2}(v) v_y\Big) d \tau\right)   dy ds da\Big].
\eaee
In the above, ingteration by parts gives
\baee
&\mathbb E Z_{ \delta, \delta_0, l} \\
=&-\mathbb E\left[\int _ { \mathbb { R } } \int _ { \Pi _ { T } } J [ \beta ^ { \prime\prime } , \partial_x\eta _ { \delta , \delta _ { 0 } } ] ( s , y , a ) \left(\int_{s-\delta_0}^s \rho_l(v(\tau, y)-a) \partial_y\left(f(v)\right) d \tau\right)dy ds da\right] \\
&+\mathbb E\left[\int _ { \mathbb { R } } \int _ { \Pi _ { T } } J [ \beta ^ { \prime\prime } , \partial_x\eta _ { \delta , \delta _ { 0 } } ] ( s , y , a ) \left(\int_{s-\delta_0}^s \rho_l(v(\tau, y)-a) \partial_y(b(v)v_y) d \tau\right)   dy ds da\right] \\
& -\sum_k\mathbb{E}  \int_{\Pi_T} \int_\R \int_{s-\delta_0}^s \left(\int_\R \left[\sigma_k^{\prime} \beta^{\prime \prime}(\cdot-a)\right](u)    \rho_l(v-a) da\right) \partial_y \sigma_k(v) \partial_x \eta_{\delta , \delta_0} d \tau \dif x  dy ds \\
&-\sum_k\frac12\mathbb E\left[\int _ { \mathbb { R } } \int _ { \Pi _ { T } } J [ \beta ^ { \prime\prime } , \partial_{xy}\eta _ { \delta , \delta _ { 0 } } ] ( s , y , a )   \left(\int_{s-\delta_0}^s \rho_l(v-a)\sigma_k^{\prime 2}(v) v_y  d \tau\right)   dy ds da\right] \\
=& A_1^l(\delta,\delta_0 )+A_2^l(\delta,\delta_0)+B^l(\delta,\delta_0)+A_3^l(\delta,\delta_0). \\
\end{aligned}
\end{equation*}
It holds that
    \baee
&\lim_{l\rightarrow0}B^l(\delta,\delta_0 )= \mathbb {E}  \int_{\Pi_T} \int_\R \int_{s-\delta_0}^s  \left[\sigma^{\prime} \beta^{\prime \prime}(\cdot-v(\tau,y))\right](u(\tau,x))  \partial_y \sigma(v(\tau,y)) \partial_x \eta_{\delta , \delta_0} d \tau \dif x  dy ds\\
&=\mathbb {E}  \int_{\Pi_T} \int_{\Pi_T}  \left[\sigma^{\prime} \beta^{\prime \prime}(\cdot-v(t,y))\right](u(t,x))  \partial_y \sigma(v(t,y)) \partial_x \eta_{\delta , \delta_0} d t \dif x  dy ds.
    \eaee
    Combining Lemmas \ref{A1}, \ref{A2} and \ref{A3}, we complete the proof.
\end{proof}

Finally, we turn to the proof of Lemma \ref{molast}.

\begin{proof}[\textbf{Proof of Lemma \ref{molast}}]
Using an integration by parts argument as in \cite[Lemma 4.2, (4.11)-(4.14)]{dareiotis2020nonlinear}, we have
\bae\label{I_6}
I_6= & -\sum_k\mathbb{E} \int_{v \leq u} \partial_{x y} \eta_{\delta , \delta_0} \int_{v}^u \int_{v}^u I_{\tilde{r} \leq r} \beta_\varepsilon^{\prime \prime}(r-\tilde{r}) \sigma_k^{\prime2}(r) d \tilde{r} d r \\
& -\sum_k\mathbb{E} \int_{v \geq u} \partial_{xy} \eta_{\delta , \delta_0} \int_{v}^u \int_{v}^u I_{\tilde{r} \geq r} \beta_\varepsilon^{\prime \prime}(r-\tilde{r}) \sigma_k^{\prime2}(r) d \tilde{r} d r \\
& -\sum_k\mathbb{E} \int_{v \geq u} \partial_{xy} \eta_{\delta , \delta_0} \int_{v}^u \int_{v}^u I_{r \leq \tilde{r}} \beta_\varepsilon^{\prime \prime}(\tilde{r}-r) \sigma_k^{\prime2}(\tilde{r}) d r d \tilde{r} \\
& -\sum_k\mathbb{E} \int_{v \leq u} \partial_{xy} \eta_{\delta , \delta_0} \int_{v}^u \int_{v}^u I_{r \geq \tilde{r}} \beta_\varepsilon^{\prime \prime}(\tilde{r}-r) \sigma_k^{\prime2}(\tilde{r}) d r d \tilde{r}
\eae
and
\bae\label{K}
&\sum_k\mathbb {E}  \int_{\Pi_T} \int_{\Pi_T}  \left[\sigma_k^{\prime} \beta_\varepsilon^{\prime \prime}(\cdot-v)\right](u)  \partial_y \sigma_k(v) \partial_x \eta_{\delta , \delta_0} d t \dif x  dy ds \\
=& \sum_k\mathbb E\int_{\Pi_T} \int_{\Pi_T}\partial_{xy} \eta_{\delta , \delta_0} \int_{v}^u \int_{v}^r \beta_\varepsilon^{\prime \prime}(r-\tilde{r}) \sigma_k^{\prime}( r) \sigma_k^{\prime}( \tilde{r}) d \tilde{r} d r \dif x  dy ds\\
=&\sum_k \mathbb{E} \int_{v \leq u} \partial_{xy} \eta_{\delta , \delta_0} \int_{v}^u \int_{\mathbb R}^v I_{\tilde{r} \leq r} \beta_\varepsilon^{\prime \prime}(r-\tilde{r}) \sigma_k^{\prime}(r)\sigma_k^{\prime}(\tilde r)  d \tilde{r} d r \\
& +\sum_k\mathbb{E} \int_{v \geq u} \partial_{xy} \eta_{\delta , \delta_0} \int_{v}^u \int_{v}^u I_{\tilde{r} \geq r} \beta_\varepsilon^{\prime \prime}(r-\tilde{r}) \sigma_k^{\prime}(r)\sigma_k^{\prime}(\tilde r) d \tilde{r} d r .
\eae
Summing the expressions in \eqref{I_6} and \eqref{K}, we observe that the contributions from the "self" terms $-\sigma_k'^2$ and the "interaction" terms can be grouped. Utilizing the symmetry of the kernel $\beta_\epsilon''(r - \tilde{r})$ and the integration bounds, we obtain
\begin{equation}\label{bye K}
    \begin{aligned}
        &I_6 + I_8 - A(\delta, \delta_0)\\
=& -\sum_{k} \mathbb{E} \int_{\Pi_T} \partial_{xy}\eta_{\delta,\delta_0} \int_v^u \int_v^u  \beta_\epsilon''(r - \tilde{r}) \Big( \sigma_k'^2(r) - 2\sigma_k'(r)\sigma_k'(\tilde{r}) + \sigma_k'^2(\tilde{r}) \Big) d\tilde{r} dr \, dx dy ds dt \\
=& -\sum_{k} \mathbb{E} \int_{\Pi_T} \partial_{xy}\eta_{\delta,\delta_0} \int_v^u \int_v^u  \beta_\epsilon''(r - \tilde{r}) \left| \sigma_k'(r) - \sigma_k'(\tilde{r}) \right|^2 d\tilde{r} dr \, dx dy ds dt.
    \end{aligned}
\end{equation}
By \eqref{adapt11}, \eqref{I_6} and \eqref{bye K}, we have
\baee
&I_6 +I_8\leq \sum_k \mathbb{E} \int_{\Pi_T} \int_{\Pi_T} |\Delta \eta_{\delta, \delta_0}| \int_v^u \int_v^u \beta_{\varepsilon}^{\prime \prime}(r-\tilde r)\left|\sigma_k^{\prime}(r)-\sigma_k^{\prime}(\tilde r)\right|^2 d r d\tilde r  d y d s d x d t +A(\delta,\delta_0)\\
\leq & \,C  \delta^{-2}\sum_k \mathbb{E}\int_{\Pi_T} \int_{\Pi_T} |\Delta \eta_{\delta, \delta_0}|\int_v^u \int_v^u \beta_{\varepsilon}^{\prime\prime}(r-\tilde r)  \left\|\sigma_k^{\prime}\right\|_{C^{1/2}_M}^2 \varepsilon d r d\tilde rd y d s d x d t +A(\delta,\delta_0) \\
\leq &\, C \varepsilon \delta^{-2} \left\|\sigma^{\prime}\right\|_{l^2(C^{1/2}_M)}^2 \mathbb{E}\int_{\Pi_T} \int_{\Pi_T}   \delta^{-1}\left|\varrho^{\prime \prime}\left(\frac{x-y}{\delta}\right) \rho_{\delta_0} \right|\psi \int_v^u \int_v^u \beta_{\varepsilon}^{\prime\prime}(r-\tilde r)    d r d\tilde rd y d s d x d t  \\
 &+  C \varepsilon \left\|\sigma^{\prime}\right\|_{l^2(C^{1/2}_M)}^2 \mathbb{E}\int_{\Pi_T} \int_{\Pi_T}  |\psi^{\prime\prime} \rho_{\delta_0}  \varrho_\delta | \int_v^u \int_v^u \beta_{\varepsilon}^{\prime\prime}(r-\tilde r)    d r d\tilde rd y d s d x d t +A(\delta,\delta_0),
\eaee
where $M=M(u_0,v_0)={\|u_0\|_{\infty}\vee \|v_0\|_{\infty}}$ is due to the maximal principle \eqref{mp}. The last inequality is valid because of {\rm{(H3)}} and Morrey's inequality. 

By \eqref{adapt}, letting $\delta_0\downarrow0,$ we have
\baee
\lim_{\delta_0\rightarrow0}\left(I_6 +I_8\right) \leq& C\, \varepsilon \delta^{-2} \left\|\sigma^{\prime}\right\|_{l^2(C^{1/2}_M)}^2 \mathbb{E}\int_{\R} \int_{\Pi_T}   \delta^{-1} \left| \varrho^{\prime \prime}\left(\frac{x-y}{\delta}\right)\right|\psi \int_v^u \int_v^u \beta_{\varepsilon}^{\prime\prime}(r-\tilde r)    d r d\tilde rd y d s d x  \\
 &+  C \varepsilon \left\|\sigma^{\prime}\right\|_{l^2(C^{1/2}_M)}^2 \mathbb{E}\int_{\R} \int_{\Pi_T}  |\psi^{\prime\prime}  \varrho_\delta| \int_v^u \int_v^u \beta_{\varepsilon}^{\prime\prime}(r-\tilde r)    d r d\tilde rd y d s d x\\
 \leq&\, C \varepsilon \delta^{-2} \left\|\sigma^{\prime}\right\|_{l^2(C^{1/2}_M)}^2 \mathbb{E}\int_{\R} \int_{\Pi_T}  \delta^{-1}\left|\varrho^{\prime \prime}\left(\frac{x-y}{\delta}\right)\right|\psi |u(s,x)-v(s,y)|d y d s d x  \\
 &+  C \varepsilon \left\|\sigma^{\prime}\right\|_{l^2(C^{1/2}_M)}^2 \mathbb{E}\int_{\R} \int_{\Pi_T}  \left|\psi^{\prime\prime}  \varrho_\delta\right| |u(s,x)-v(s,y)| d y d s d x\\
 \leq &\, C\varepsilon(1+\delta^{-2}).
\eaee
Thus the proof is completed.
\end{proof}

\section{Acknowledgements}
 This paper is partially supported by  NSFC No. 12471138 and 12571171. 

\def\refname{ References}
\bibliographystyle{abbrv}
\bibliography{ref}

\end{document}